\documentclass[a4paper,twoside]{article}  

\usepackage{amssymb,amsmath,amsthm,dsfont,amsfonts,color,latexsym,CJK}

\usepackage{hyperref}

\usepackage{mathrsfs} 

\definecolor{blue}{rgb}{0.00,0.00,1.00}
\definecolor{red}{rgb}{1.00,0.00,0.00}

\renewcommand{\baselinestretch}{1.2}
\def\bq{\begin{equation}}
\def\eq{\end{equation}}
\def\ba{\begin{array}{ccc}}
\def\bal{\begin{array}{lll}}
\def\ea{\end{array}}

\def\dlim{\displaystyle\lim}

 \def\dsum{\displaystyle\sum}
 \def\lt#1{\left#1}\def\rt#1{\right#1}
\def\({\left(}\def\){\right)}
\def\[{\left[}\def\]{\right]}
\def\<{\langle}\def\>{\rangle}

    \def \R   {\mathbb{R}}

    \def\P    {\mathrm{P}}
    \def\i    {\mathrm{i}}

    \def\S    {\mathbb{S}}
    
    \def\intr {\int_{\R^3}}
    \def\intrr {\int_{\R^6}}
    \def\ints {\int_{\S^2}}
    \def\intt {\int^t_0}

    \def \Q    {\mathcal{Q}}
    
    \def \N    {\mathbb{N}}

    \def \pt   {\partial}
    
    \def \Dt   {\frac{\rm d}{{\rm d}t}}
    
    \def \dt    {\partial_t}
    
    \def \da    {\pt^\alpha}

    \def \dx    {\partial_x}

    \def \dxa   {\partial^{\alpha}_x}
    
    \def \dv    {\partial_v}
    
    \def \dvb   {\partial^{\beta}_v}

    \def \divx  {{\rm div}_x}

    \def\Tdx   {\nabla_x}
    \def\Tdv   {\nabla_v}

       \def\bq{\begin{equation}}
       \def\eq{\end{equation}}
       \def\be{\begin{equation}}
       \def\ee{\end{equation}}
       \def\bma#1\ema{{\allowdisplaybreaks\begin{align}#1\end{align}}}
       \def\bmas#1\emas{{\allowdisplaybreaks\begin{align*}#1\end{align*}}}
       \def\bln#1\eln{{\allowdisplaybreaks\begin{aligned}#1\end{aligned}}}
       \def\nnm{\notag}
       \def\bgr#1\egr{\allowdisplaybreaks\begin{gather}#1\end{gather}}
       \def\bgrs#1\egrs{\allowdisplaybreaks\begin{gather*}#1\end{gather*}}

       \theoremstyle{plain}
       \newtheorem{lem}{\bf Lemma}[section]
       \newtheorem{thm}[lem]{\textbf{Theorem}}
       \newtheorem{prop}[lem]{\textbf{Proposition}}
       \newtheorem{rem}[lem]{\textbf{Remark}}

       \newtheorem{remark}[lem]{\bf Remark}

\begin{document}

\title{ Spectrum analysis and optimal decay rates of  the bipolar Vlasov-Poisson-Boltzmann equations }
\author{ Hai-Liang Li$^1$,\, Tong Yang$^2$,\, Mingying Zhong$^3$\\[2mm]
 \emph{\small\it  $^1$Department of  Mathematics,
    Capital Normal University, and BCMIIS, Beijing, P.R.China }\\
    {\small\it E-mail:\ hailiang.li.math@gmail.com}\\
    {\small\it $^2$Department of Mathematics, City University of Hong Kong,  Hong
    Kong}\\
    {\small\it E-mail: matyang@cityu.edu.hk} \\
    {\small\it  $^3$Department of  Mathematics and Information Sciences,
    Guangxi University, P.R.China.}\\
    {\small\it E-mail:\ zhongmingying@sina.com}\\[5mm]
    }
\date{ }

\pagestyle{myheadings}
\markboth{Vlasov-Poisson-Boltzmann equation }%
{H.-L. Li, T. Yang, M.-Y. Zhong }

 \maketitle

 \thispagestyle{empty}

\begin{abstract}\noindent
In the present paper, we consider the initial value problem for the bipolar Vlasov-Poisson-Boltzmann (bVPB) system and its corresponding modified  Vlasov-Poisson-Boltzmann (mVPB). We give the spectrum analysis on the linearized bVPB and mVPB systems around their equilibrium state and show the optimal convergence rate of global solutions. It was showed that the electric field decays exponentially and the distribution function tends to the absolute Maxwellian at the optimal convergence rate $(1+t)^{-3/4}$ for the bVPB system, yet both the electric field and the distribution function converge to equilibrium state at the optimal rate $(1+t)^{-3/4}$  for the mVPB system.

\medskip
 {\bf Key words}. Bipolar Vlasov-Poisson-Boltzmann system, Modified Vlasov-Poisson-Boltzmann, spectrum analysis, optimal time decay rates.

\medskip
 {\bf 2010 Mathematics Subject Classification}. 76P05, 82C40, 82D05.
\end{abstract}

%

\tableofcontents

\section{Introduction}
The bipolar Vlasov-Poisson-Boltzmann  (bVPB)  system of two species can be used to model the time evolution of dilute charged particles (e.g., electrons and ions) in the absence of an external magnetic field \cite{Markowich}. In general, the bVPB system for two species of particles in the whole space take the form
 \bgr
\dt F_++v\cdot\Tdx F_++\Tdx \Phi\cdot\Tdv F_+
=\Q(F_+,F_+)+\Q(F_+,F_-),\label{VPB1}\\
\dt F_-+v\cdot\Tdx F_--\Tdx \Phi\cdot\Tdv F_-
=\Q(F_-,F_-)+\Q(F_-,F_+),\label{VPB2}\\
\Delta_x\Phi=\intr (F_+-F_-) dv,\label{VPB3}\\
 F_+(x,v,0)=F_{+,0}(x,v),\quad F_-(x,v,0)=F_{-,0}(x,v),  \label{VPB3i}
\egr
where $F_+=F_+(x,v,t)$ and $F_-=F_-(x,v,t)$ are number
density functions of ions and electrons, and   $\Phi(x,t)$ denotes the electric potential, respectively. The collision integral $\Q(F,G)$ describes the interaction between particles due to binary collisions by
 \bq
 \Q(F,G)=\intr\ints
|(v-v_*)\cdot\omega|(F(v')G(v'_*)-F(v)G(v_*))dv_*d\omega,  \label{binay_collision}
 \eq
where $$ v'=v-[(v-v_*)\cdot\omega]\omega,\quad v'_*=v_*+[(v-v_*)\cdot\omega]\omega,\quad \omega\in\S^2 .$$

Assume that the electron density is very rarefied and reaches a local equilibrium state with small electron mass compared with the ions, and that the collision $\Q(F_+,F_-)$ between the ions and electrons can be neglected,  the equation \eqref{VPB2} can be reduced to
\be
 v\cdot\Tdx F_-- \Tdx \Phi\cdot\Tdv F_-
=0
\ee
This together with the simple local Maxwellian distribution of electron leads to
 $$
F_-=\rho_-(x)M(v)=\frac{1}{(2\pi)^{\frac32}}e^{-\Phi}e^{-\frac{|v|^2}{2}}
$$
with the normalized Maxwellian $M(v)$ given by
$$
 M=M(v)=\frac1{(2\pi)^{3/2}}e^{-\frac{|v|^2}2}. $$
For rigorous reduction of this type at the hydrodynamical scale, the readers can refer to \cite{Codier-Grenier}. Under the above reduction, we can obtain from  bVPB system~\eqref{VPB1}--\eqref{VPB3} the following modified  Vlasov-Poisson-Boltzmann (mVPB) system:
 \bgr
  F_t+v\cdot\Tdx F+\Tdx \Phi\cdot\Tdv F =\Q(F,F),\label{m_VPB1}
 \\
  \Delta_x\Phi=\intr Fdv-e^{-\Phi},\label{m_VPB2}
\\
  F(x,v,0)=F_0(x,v),\label{m_VPB3}
\egr
where $F=F(x,v,t)$ is the distribution function of ions with $(x,v,t)\in \R^3\times\R^3\times\R^+$,  $\Phi(x,t)$ is the electric potential, and $\Q(F,G)$ is the  binary collision operator defined as \eqref{binay_collision}.

In the case that the effect of electron is totally neglected, the bVPB system can be simplified to the standard unipolar Vlasov-Poisson-Boltzmann (VPB) model similar to \eqref{m_VPB1}--\eqref{m_VPB3} with the term $e^{-\Phi}$ replaced by a given function, a positive constant for instance.
\par

There have been a lot of works on the existence and behavior of solutions to the Vlasov-Poisson-Boltzmann system. The global existence of renormalized solution
for  large initial data was proved in \cite{Mischler}. The first global existence  result on classical solution in torus when the initial data is near a global Maxwellian was established in \cite{Guo2}. And the global existence of classical solution in  $\R^3$  was
given  \cite{Yang1,Yang3} in the same setting.  The case
with general stationary background density function $\bar{\rho}(x)$ was studied in \cite{Duan2}, and the perturbation of vacuum
was investigated  in \cite{Guo3,Duan4}. Recently, Li-Yang-Zhong~\cite{Li2} analyze the spectrum of the linearized VPB system and obtain the optimal decay rate of solutions to the nonlinear system near Maxwellian.

However, in contrast to the works on Boltzmann equation~\cite{Ellis, Ukai1,Ukai2,Ukai3} and VPB system \cite{Li2}, the spectrums   of the linearized bVPB system and modified VPB system have not been given despite of its importance.
On the other hand, an interesting phenomenon was shown recently in \cite{Duan1}
 on the time
asymptotic behavior of the solutions which shows that the global classical solution of one species VPB system tends to the equilibrium at  $(1+t)^{-\frac14}$ in $L^2$-norm. This is slower than the rate
  for the two species VPB system, that is, $(1+t)^{-\frac34}$, obtained in \cite{Yang4}.
Therefore, it is natural to investigate  whether these rates are optimal.

\par

The main purpose of the present paper is to investigate the spectrum and optimal time-convergence rates of global solutions to the 
linearized the bVPB system~\eqref{VPB1}--\eqref{VPB3i} in section~\ref{sect2.1} and the mVPB~~\eqref{m_VPB1}--\eqref{m_VPB3} in section~\ref{sect2.2} respectively. In particular, the main results established in this paper justifies how the electric field and the interplay interaction between ions and electrons influence the asymptotical behaviors of the global solution to the bVPB system~\eqref{VPB1}--\eqref{VPB3i} and  the mVPB~~\eqref{m_VPB1}--\eqref{m_VPB3}.

The rest of this paper will be organized as follows. The main results about the global existences and the optimal time-convergence rates of strong solution to  bVPB system~\eqref{VPB1}--\eqref{VPB3i} and mVPB~~\eqref{m_VPB1}--\eqref{m_VPB3} are stated in Section~\ref{sect2.0}.
In Section~\ref{sect3}, we analyze the spectrum of the bVPB system and mVPB system, and then establish the  exponential  time decay rates of  the  linearized bVPB and  the  algebraic time decay rates of the linearized mVPB equations  in Sections~\ref{sect3.1}--\ref{sect3.m}.
In Sections~\ref{sect4} and \ref{mvpb}, we prove the optimal time decay rates of the
global solution to the original nonlinear bVPB system and mVPB system respectively.

\bigskip

\noindent\textbf{Notations:} \ \ Define the Fourier transform of $f=f(x,v)$ by
$\hat{f}(\xi,v)=\mathcal{F}f(\xi,v)=\frac1{(2\pi)^{3/2}}\intr f(x,v)e^{- \i x\cdot\xi}dx,$
where and throughout this paper we denote $\i=\sqrt{-1}$.

Denote the weight function $w(v)$ by
$$
w(v)=(1+|v|^2)^{1/2}
$$
and the Sobolev spaces $ H^N$ and $ H^N_{w}$ as
$$
 H^N=\{\,f\in L^2(\R^3_x\times \R^3_v)\,|\,\|f\|_{H^N}<\infty\,\},\quad
 H^N_{w}=\{\,f\in L^2(\R^3_x\times \R^3_v)\,|\,\|f\|_{H^N_w}<\infty\,\}
$$
equipped with the norms
$$
 \|f\|_{H^N}=\sum_{|\alpha|+|\beta|\le N}\|\dxa\dvb f\|_{L^2(\R^3_x\times \R^3_v)},
 \quad
 \|f\|_{H^N_{w}}=\sum_{|\alpha|+|\beta|\le N}\|w\dxa\dvb f\|_{L^2(\R^3_x\times \R^3_v)}.
$$
For $q\ge1$, we also define
$$
L^{2,q}=L^2(\R^3_v,L^q(\R^3_x)),\quad
\|f\|_{L^{2,q}}=\bigg(\intr\bigg(\intr|f(x,v)|^q dx\bigg)^{2/q}dv\bigg)^{1/2}.
$$
In the following, we denote by $\|\cdot\|_{L^2_{x,v}}$ and $\|\cdot\|_{L^2_{\xi,v}}$ the norms of the function spaces $L^2(\R^3_x\times \R^3_v)$ and $L^2(\R^3_\xi\times \R^3_v)$ respectively, and denote by $\|\cdot\|_{L^2_x}$, $\|\cdot\|_{L^2_\xi}$ and $\|\cdot\|_{L^2_v}$  the norms of the function spaces $L^2(\R^3_x)$, $L^2(\R^3_\xi)$ and $L^2(\R^3_v)$ respectively. 
For any integer $m\ge1$, we denote by $\|\cdot\|_{H^m_x}$ and $\|\cdot\|_{L^2_v(H^m_x)}$ the norms in the spaces $H^m(\R^3_x)$ and $L^2(\R^3_v,H^m(\R^3_x))$ respectively.

\section{Main results}
\label{sect2.0}
\subsection{bVPB system}\label{sect2.1}
First of all, we consider the  Cauchy problem of the bVPB system~\eqref{VPB1}--\eqref{VPB3i} in the present paper. Define
$$
 F_1=:F_++F_-,\quad F_2=:F_+-F_-.
$$
Then Cauchy problem of the bVPB system~\eqref{VPB1}--\eqref{VPB3i} can be rewritten as
 \bgr
\dt F_1+v\cdot\Tdx F_1+\Tdx \Phi\cdot\Tdv F_2
=\Q(F_1,F_1),\label{VPB4}\\
\dt F_2+v\cdot\Tdx F_2+\Tdx \Phi\cdot\Tdv F_1
=\Q(F_2,F_1),\label{VPB5}\\
\Delta_x\Phi=\intr F_2 dv,\label{VPB6}\\
F_1(x,v,0)=F_{1,0}(x,v)=F_{+,0}+F_{-,0},\quad F_2(x,v,0)=F_{2,0}(x,v)=F_{+,0}-F_{-,0}. \label{VPB6i}
\egr

The bVPB system \eqref{VPB4}-\eqref{VPB6} has an equilibrium state $(F_1^*,F_2^*,\Phi^*)=(M(v),0,0)$.
Define the perturbations $f_1(x,v,t)$ and $f_2(x,v,t)$ by $$F_1=M+\sqrt M f_1,\quad F_2=\sqrt M f_2.$$ Then the bVPB system~\eqref{VPB4}--\eqref{VPB6} for $f_1(x,v,t)$ and $f_2(x,v,t)$ is reformulated into
\bgr
 \dt f_1+v\cdot\Tdx f_1-Lf_1
=\frac12 (v\cdot\Tdx\Phi)f_2-\Tdx\Phi\cdot\Tdv f_2+\Gamma(f_1,f_1),\label{VPB7}
\\
\dt f_2+v\cdot\Tdx f_2-v\sqrt{M}\cdot\Tdx\Phi-L_1f_2
=\frac12 (v\cdot\Tdx\Phi)f_1-\Tdx\Phi\cdot\Tdv f_1+\Gamma(f_2,f_1),\label{VPB8}
\\
\Delta_x\Phi=\intr f_2\sqrt{M}dv,\label{VPB9}
\\
f_1(x,v,0)=f_{1,0}(x,v)= (F_{1,0}-M) M^{-\frac12},\quad
f_2(x,v,0)=f_{2,0}(x,v)= F_{2,0} M^{-\frac12},\label{VPB10}
 \egr
where the operators $Lf$, $L_1f$ and
$\Gamma(f,f)$ are defined by
 \bma
Lf&=\frac1{\sqrt M}[\Q(M,\sqrt{M}f)+\Q(\sqrt{M}f,M)],\label{Lf}\\
L_1f&=\frac1{\sqrt M}\Q(\sqrt{M}f,M),\label{L1f}\\
\Gamma(f,g)&=\frac1{\sqrt M}\Q(\sqrt{M}f,\sqrt{M}g).  \label{Gf}
 \ema
The linearized collision operators $L$ and $L_1$ can be written as \cite{Cercignani, Yu}
 \bmas
(Lf)(v)&=(Kf)(v)-\nu(v) f(v),\quad (L_1f)(v)=(K_1f)(v)-\nu(v) f(v),\\
\nu(v)&=\intr\ints |(v-v_*)\cdot\omega|M_*d\omega dv_*,\\
(Kf)(v)&=\intr\ints
|(v-v_*)\cdot\omega|(\sqrt{M'_*}f'+\sqrt{M'}f'_*-\sqrt{M}f_*)\sqrt{M_*}d\omega
dv_*\\
&=\intr k(v,v_*)f(v_*)dv_*,\\
(K_1f)(v)&=\intr\ints
|(v-v_*)\cdot\omega|\sqrt{M'_*}\sqrt{M_*}f'd\omega
dv_*=\intr k_1(v,v_*)f(v_*)dv_*,
\emas
where $\nu(v)$ is called
the collision frequency, $K$ and $K_1$ are self-adjoint compact operators
on $L^2(\R^3_v)$ with real symmetric integral kernels $k(v,v_*)$ and $k_1(v,v_*)$.
The nullspace of the operator $L$, denoted by $N_0$, is a subspace
spanned by the orthogonal basis $\{\chi_j,\ j=0,1,\cdots,4\}$  with
\bq \chi_0=\sqrt{M},\quad \chi_j=v_j\sqrt{M} \ (j=1,2,3), \quad
\chi_4=\frac{(|v|^2-3)\sqrt{M}}{\sqrt{6}},\label{basis}\eq
and the nullspace of the operator $L_1$, denoted by $N_1$, is a subspace
spanned by $\sqrt{M}$.

 We denote $L^2(\R^3)$ be a Hilbert space of complex-value functions $f(v)$
on $\R^3$ with the inner product and the norm
$$
(f,g)=\intr f(v)\overline{g(v)}dv,\quad \|f\|=\(\intr |f(v)|^2dv\)^{1/2}.
$$
Let $\P_0,P_{\rm d}$ be the projection operators from $L^2(\R^3_v)$
to the subspace $N_0, N_1$ with
\bma
 &\P_0f=\sum_{i=0}^4(f,\chi_i)\chi_i,\quad \P_1=I-\P_0, \label{P10}
 \\
 &P_{\rm d}f=(f,\sqrt M)\sqrt M,   \quad P_r=I-P_{\rm d}. \label{Pdr}
 \ema

From the Boltzmann's H-theorem, the linearized collision operators $L$ and $L_1$ are non-positive and
moreover, $L$ and $L_1$ are locally coercive in the sense that there is a constant $\mu>0$ such that \bma
 (Lf,f)&\leq -\mu \| \P_1f\|^2, \quad  \ f\in D(L),\label{L_3}\\
 (L_1f,f)&\leq -\mu \|P_rf\|^2, \quad  \ f\in D(L_1),\label{L_4}
 \ema
where $D(L)$ and $D(L_1)$ are the domains of $L$ and $L_1$ given by
$$ D(L)=D(L_1)=\left\{f\in L^2(\R^3)\,|\,\nu(v)f\in L^2(\R^3)\right\}.$$
In addition, for the hard sphere model, $\nu$ satisfies
 \be
\nu_0(1+|v|)\leq\nu(v)\leq \nu_1(1+|v|).  \label{nuv}
 \ee

From the system \eqref{VPB7}--\eqref{VPB10} for $(f_1,f_2)$,
we have the following decoupled linearized system for $f_1$ and $f_2$:
\bma
\dt f_1=Ef_1,\quad f_1(x,v,0)=f_{1,0}(x,v),\label{LVPB1}\\
\dt f_2=Bf_2,\quad f_2(x,v,0)=f_{2,0}(x,v),\label{LVPB2}
\ema
where
\bma
Ef_1&=Lf_1-(v\cdot\Tdx)f_1,\\
Bf_2&=L_1f_2-(v\cdot\Tdx)f_2-v\sqrt M\cdot \Tdx(-\Delta_x)^{-1}\intr f_2\sqrt Mdv.
\ema
The equation \eqref{LVPB1} is the linearized Boltzmann equation,  its spectrum analysis and the optimal decay rate of the solution has already  been made for instance in \cite{Ukai1,Zhong2012Sci}. Therefore, we only need to investigate the spectrum analysis and the decay rate of the solution to the linearized Vlasov-Poisson-Boltzmann type equation \eqref{LVPB2}. Indeed, take Fourier transform to \eqref{LVPB1}--\eqref{LVPB2} in $x$ to get
\bma
\dt \hat{f}_1&=\hat E(\xi)\hat f_1, \label{LVPB5a}\\
 \dt \hat{f}_2&=\hat B(\xi)\hat f_2, \label{LVPB5} 
\ema
where the operators $\hat{E}(\xi)$, $\hat{B}(\xi)$ are defined for $\xi\ne 0$ by
\bmas
\hat{E}(\xi)=L_1-\i(v\cdot\xi),\quad
\hat{B}(\xi)=L_1-\i(v\cdot\xi)-\frac{\i(v\cdot\xi)}{|\xi|^2}P_{\rm d}.
\emas

Then, we have
\begin{thm}
\label{rate2}
Let $\sigma(\hat{B}(\xi))$ denotes the spectrum 
of operator $\hat{B}(\xi)$ to the linear equation \eqref{LVPB5}.
There exist 
a constant  $a_1>0$ such that it holds for all $\xi\ne0$ that
\bgr
\sigma(\hat{B}(\xi))\subset\{\lambda\in\mathbb{C}\,|\, {\rm Re}\lambda<-a_1\}. \label{sg1z}
\egr

Let $\sigma(\hat{E}(\xi))$ denotes the spectrum 
of operator $\hat{E}(\xi)$ to the linear equation \eqref{LVPB5a}. Then, 
 for any $r_0>0$  there
exists $\alpha=\alpha(r_0)>0$ so that it holds  for $|\xi|\geq r_0$ that
\be
\sigma(\hat{E}(\xi))\subset\{\lambda\in\mathbb{C}\,|\,{\rm Re}\lambda\le -\alpha\}. \label{sg2}
\ee
There exists a constant $r_0>0$ such that the spectrum $\sigma(\hat{E}(\xi))$ for $\xi=s\omega$ with $|s|\leq r_0$ and $\omega\in \mathbb{S}^2$ consists of five points $\{\mu_j(s),j=-1,0,1,2,3\}$ on the domain $\mathrm{Re}\lambda>-\mu/2$,  which are $C^\infty$ functions of $s$ for $|s|\leq r_0$ and satisfy the following asymptotical expansion for $|s|\leq r_0$
\be
 \left\{\bln  \label{eigen_2}
 \mu_{\pm1}(s)&=\pm \i\sqrt{\frac53} s-b_{\pm1}s^2+o(s^2),\quad \overline{\mu_1(s)}=\mu_{-1}(s),\\
 \mu_{0}(s) &=-b_0s^2+o(s^2),\\
 \mu_2(s) &=\mu_3(s) =-b_2s^2+o(s^2),
 \eln\right.
\ee
with constants $b_j>0$, $-1\le j\le 2$.
\end{thm}

With above spectrum analysis, we can obtain the global existence and the time-asymptotical behavior of unique solution to the Cauchy problem for the linear bVPB system~\eqref{LVPB5a}--\eqref{LVPB5} as follows.

\begin{thm}\label{rate2}
Assume that $f_{1,0}\in L^2_v(H^N_x)\cap L^{2,q}$ for $N\ge 1$ and $q\in[1,2]$. Then there is a globally unique solution $f_1(x,v,t)=e^{tE}f_{1,0}(x,v)$ to  the linearized Boltzmann equation~\eqref{LVPB1}, which satisfies for any $\alpha,\alpha'\in\N^3$ with  $|\alpha|\le N$, $\alpha'\le \alpha$  and $m=|\alpha-\alpha'|$ that
 \bma
 \|(\da_x e^{tE}f_{1,0},\chi_j)\|_{L^2_x}
&\leq C(1+t)^{-\frac 32\(\frac1q-\frac12\)-\frac m2}(\|\da_x
f_0\|_{L^2_{x,v}}+\|\dx^{\alpha'}f_0\|_{L^{2,q}}),\quad j=0,1,2,3,4, \label{V_1}
\\
\|  \P_1 (\da_x e^{tE}f_{1,0})\|_{L^2_{x,v}}
&\leq
 C(1+t)^{-\frac 32\(\frac1q-\frac12\)-\frac{m+1}{2}}
  (\|\da_x f_0\|_{L^2_{x,v}}+\|\dx^{\alpha'}f_0\|_{L^{2,q}}). \label{V_2}
 \ema
In addition, assume that $f_{1,0}\in L^2(\R^3_v;H^N(\R^3_x)\cap L^{1}(\R_x^3))$ for $N\ge 1$ and there exist positive constants $d_0,d_1>0$ and a small constant $r_0>0$ so that the Fourier transform $\hat{f}_{1,0}(\xi,v)$ of the initial data $f_{1,0}(x,v)$ satisfies that $\inf_{|\xi|\le r_0}|(\hat f_{1,0},\sqrt M)|\ge d_0$, $\inf_{|\xi|\le r_0}|(\hat f_{1,0},\chi_4)|\ge d_1\sup_{|\xi|\le r_0}|(\hat f_{1,0},\sqrt M)|$ and $\sup_{|\xi|\le r_0}|(\hat f_{1,0},v\sqrt M)|=0$. Then global solution $f(x,v,t)=e^{tE}f_{1,0}(x,v)$  satisfies for two positive constants $C_2\ge C_1$ that
\bma
  C_1(1+t)^{-\frac34-\frac k2}
 &\leq \|\Tdx^k(e^{tE}f_{1,0},\chi_j)\|_{L^2_x} \leq C_2(1+t)^{-\frac34-\frac k2},\quad j=0,1,2,3,4, \label{H_1}
\\
 C_1(1+t)^{-\frac54-\frac k2}
 &\leq \| \Tdx^k \P_1 (e^{tE}f_{1,0})\|_{L^2_{x,v}}\leq C_2(1+t)^{-\frac54-\frac k2}, \label{H_2}
 \ema
for $t>0$ sufficiently large and $k\ge 0$.

 Furthermore,  if $f_{2,0}\in L^2_v(H^N_x)\cap L^{2,1}$ for $N\ge1$, then the global solution $f_2(x,v,t)= e^{tB}f_{2,0}(x,v)$ to the linear Vlasov-Poisson-Boltzmann type equation \eqref{LVPB2} exists globally in time and satisfies for $t>0$
\bma
&\|\da_x f_2(t)\|_{L^2_{x,v}}+\|\dxa \nabla_x\Phi(t)\|_{L^2_x}\leq
Ce^{-\frac12a_1t}(\|\da_x f_{2,0}\|_{L^2_{x,v}}+\|f_{2,0}\|_{L^{2,1}}) \label{D_2}
\ema
for $0\le  |\alpha|\le N$, where $\Tdx\Phi(t)=\Tdx\Delta_x^{-1}(e^{tB}f_{2,0},\sqrt M)$.
\end{thm}

With the help of optimal time decay rates on the linearized bVPB \eqref{LVPB1}--\eqref{LVPB2} given by Theorem~\ref{rate2}, we can obtain the optimal decay rates of the global solution to original bVPB system~\eqref{VPB1}--\eqref{VPB3} as follows.

\begin{thm}
\label{rate3}
Assume that
$f_{\pm,0}=(F_{\pm,0}-\frac12 M)M^{-\frac12}\in H^N_w\cap L^{2,1}$ for $N\ge 4$  and
 $\| f_{\pm,0}\|_{H^N_{w}\cap L^{2,1}}\le \delta_0$ for  a constant  $\delta_0>0$ small enough. Then there exists  a globally unique solution $(F_{\pm},\Phi)$ with  $F_{\pm}(x,v,t)=\frac12 M+\sqrt M f_{\pm}(x,v,t)$ to the bVPB system~\eqref{VPB1}--\eqref{VPB3}, which satisfies
\bgr
\|\dx^k(f_{+},f_{-})(t)\|_{L^2_{x,v}}\le C\delta_0(1+t)^{-\frac34-\frac k2},\label{t_5z}
\\
\|\dx^k\Tdx\Phi(t)\|_{L^2_{x}}\le C\delta_0e^{-dt},\label{t_7}
\intertext{and in particular }
\|\dx^k( f_{+}(t),\chi_j)\|_{L^2_{x}}+
\|\dx^k( f_{-}(t),\chi_j)\|_{L^2_{x}}\le C\delta_0(1+t)^{-\frac34-\frac k2},\label{t_5}
\\
\|\dx^k (\P_1  f_{+},\P_1  f_{-})(t)\|_{L^2_{x,v}}
\le C\delta_0(1+t)^{-\frac54-\frac k2},\label{t_6}
\\
\|(\P_1 f_{+},\P_1 f_{-})(t)\|_{H^N_w}  +\|\Tdx(\P_0 f_{+},\P_0 f_{-})(t)\|_{L^2_v(H^{N-1}_x)}\le
C\delta_0(1+t)^{-\frac54},\label{t_8}
\egr
 for $j=0,1,2,3,4,$ $k=0,1$ and a constant $d>0$. 

Moreover, the recombination $(f_1,f_2)$ with $f_1=:f_++f_-, f_2=:f_+-f_-$ is the global solution to the system \eqref{VPB7}--\eqref{VPB10} satisfies
\bgr
\|\dx^k(f_1(t),\chi_j)\|_{L^2_{x}}\le C\delta_0(1+t)^{-\frac34-\frac k2},\quad j=0,1,2,3,4, \label{t_1}\\
\|\dx^k  \P_1f_1(t)\|_{L^2_{x,v}}\le C\delta_0(1+t)^{-\frac54-\frac k2},\label{t_2}\\
\|\dx^k f_2(t)\|_{L^2_{x,v}}+\|\dx^k\Tdx\Phi(t)\|_{L^2_x}\le C\delta_0e^{-dt},\label{t_3}\\
\|(\P_1f_1,P_rf_2)(t)\|_{H^N_w} +\| \Tdx( \P_0f_1, P_{\rm d}f_2)(t)\|_{L^2_v(H^{N-1}_x)}\le
C\delta_0(1+t)^{-\frac54},\label{t_4}
\egr
 for  $k=0,1$. 
\end{thm}

\begin{thm}
\label{rate4}
Let the assumptions of Theorem~\ref{rate3} hold.
Assume further that there exist  positive constants $d_0,d_1>0$ and a small constant $r_0>0$ so that  $\hat{f}_{\pm,0}=(\hat{F}_{\pm,0}-\frac12 M)M^{-\frac12}$ satisfies that
 $\inf_{|\xi|\le r_0}|(\hat f_{1,0},\sqrt M)|\ge d_0$,
 $\sup_{|\xi|\le r_0}|(\hat f_{1,0},\chi_j)|=0$ $(j=1,2,3)$ and
 $\inf_{|\xi|\le r_0}|(\hat f_{1,0},\chi_4)|\ge d_1\sup_{|\xi|\le r_0}|(\hat f_{1,0},\sqrt M)|$ with $\hat f_{1,0}=\hat f_{+,0}+\hat f_{-,0}$.
Then, the global solution $(F_{\pm},\Phi)$ with  $F_{\pm}(x,v,t)=\frac12 M+\sqrt M f_{\pm}(x,v,t)$ to the bVPB system~\eqref{VPB1}--\eqref{VPB3} satisfies
 \bgr
 C_1\delta_0(1+t)^{-\frac34-\frac k2} \le
  \|\Tdx^kf_{\pm}(t)\|_{L^2_{x,v}}
 \le
C_2\delta_0(1+t)^{-\frac34-\frac k2},\label{B_4a}
\\
\|\Tdx^k\Tdx\Phi(t)\|_{L^2_{x}}\le C\delta_0e^{-dt},\label{t_7z}
\intertext{and in particular }
 C_1\delta_0(1+t)^{-\frac34-\frac k2} \le
  \|\Tdx^k(f_{\pm}(t),\chi_j)\|_{L^2_{x}}
  \le C_2\delta_0(1+t)^{-\frac34-\frac k2},\label{B_1a}\\
 C_1\delta_0(1+t)^{-\frac54-\frac k2} \le
  \| \Tdx^k\P_1f_{\pm}(t)\|_{L^2_{x,v}}
  \le C_2\delta_0(1+t)^{-\frac54-\frac k2},\label{B_3a}
\egr
for $t>0$ large with two positive constants $C_2>C_1$, $j=0,1,2,3,4,$ and $k=0,1$.

Moreover, the recombination $f_1=f_++f_-, f_2=f_+-f_-$  to the system \eqref{VPB7}--\eqref{VPB10} satisfies
 \bgr
 C_1\delta_0(1+t)^{-\frac34-\frac k2} \le
 \|\Tdx^k(f_1(t),\chi_j)\|_{L^2_{x}}
 \le C_2\delta_0(1+t)^{-\frac34-\frac k2},\label{B_1}
 \\
 C_1\delta_0(1+t)^{-\frac54-\frac k2} \le
 \|  \Tdx^k\P_1f_1(t)\|_{L^2_{x,v}}
 \le C_2\delta_0(1+t)^{-\frac54-\frac k2},\label{B_3}\\
 C_1\delta_0(1+t)^{-\frac34} \le \|f_1(t)\|_{H^N_w}\le
C_2\delta_0(1+t)^{-\frac34},\label{B_4}\\
\|\Tdx^k f_2(t)\|_{L^2_{x,v}} +\|\dx^k\Tdx\Phi(t)\|_{L^2_x}\le C\delta_0e^{-dt},\label{t_3az}
 \egr
 for $t>0$ large with two constants $C_2>C_1$, $j=0,1,2,3,4,$  and $k=0,1$. 
\end{thm}

\subsection{mVPB system}
\label{sect2.2}

Next, we deal with the global existence and uniqueness of solution to the Cauchy problem for the mVPB system~\eqref{m_VPB1}-\eqref{m_VPB3} and the optimal time-convergence rate of the global solutions. The mVPB system~\eqref{m_VPB1}-\eqref{m_VPB2} has an equilibrium state $(F_*,\Phi_*)=(M,0)$ with  $M=M(v)$ being the normalized global Maxwellian defined above.
Define the perturbation $f(x,v,t)$ of $F$ near $M$ by
$$
f =(F-M)M^{-\frac12},
$$
then the modified Vlasov-Poisson-Boltzmann system \eqref{m_VPB1}-\eqref{m_VPB3} for $f(x,v,t)$ reads
 \bgr
  \dt f+v\cdot\Tdx f-v\sqrt{M}\cdot\Tdx \Phi-Lf
=\frac12 (v\cdot\Tdx\Phi)f-\Tdx\Phi\cdot\Tdv f+\Gamma(f,f),\label{m_VPB4}\\
(I-\Delta_x)\Phi=-\intr f\sqrt{M}dv+(e^{-\Phi}+\Phi-1),\label{m_VPB5}\\
f(x,v,0)=f_0(x,v)=:(F_0-M){M}^{-1/2},\label{m_VPB6}
 \egr
where the operators $Lf$ and $\Gamma(f,f)$ are defined by \eqref{Lf} and \eqref{Gf} respectively.

From the modified VPB system~\eqref{m_VPB4}--\eqref{m_VPB6}, we have the following  the linearized mVPB  equation
 \bma
  \dt f=&B_mf,\quad t>0,\label{m_VPB}\\
  f(x,v,0)=&f_0(x,v),\quad (x,v)\in\R^3_x\times\R^3_v,  \label{m_VPB.a}
 \ema
where the linear operator $B_m$ is defined by
$$
 B_mf=Lf-v\cdot\Tdx f-v\sqrt{M}\cdot\Tdx(I-\Delta_x)^{-1} \(\intr f\sqrt{M}dv\).
$$
Take the Fourier transform to \eqref{m_VPB} with respect to $x$ to get
 \be
\dt \hat{f}=\hat{B}_m(\xi) \hat{f}, \label{LVPB5z}
\ee
where  
$$ \hat{B}_m(\xi)=L-\i(v\cdot\xi)-\frac{\i (v\cdot\xi)}{1+|\xi|^2}P_{\rm d}.$$

Then, we have the spectrum analysis of the operator $\hat{B}_m(\xi) $ and the time-decay rates of the global solution to the linearized mVPB system~\eqref{m_VPB}--\eqref{m_VPB.a} and establish its optimal time-decay rates as follows.

\begin{thm}\label{m_eigen_3}  Let $\sigma(\hat{B}_m(\xi))$ denotes the spectrum 
of operator $\hat{B}_m(\xi)$ to the linear equation \eqref{LVPB5z} for all  $\xi\in \R^3$. Then, 
for any $r_0>0$  there
exists $\alpha=\alpha(r_0)>0$ so that it holds  for $|\xi|\geq r_0$ that
$$
\sigma(\hat{B}_m(\xi))\subset\{\lambda\in\mathbb{C}\,|\,{\rm Re}\lambda\le -\alpha\}. 
$$
There exists a constant  $r_0>0$ so that the spectrum $\lambda\in\sigma(B_m(\xi))\subset\mathbb{C}$ for $\xi=s\omega$ with $|s|\leq r_0$ and $\omega\in \mathbb{S}^2$ consists of five points $\{\lambda_j(s),\ j=-1,0,1,2,3\}$ on the domain $\mathrm{Re}\lambda>-\mu /2$, which are $C^\infty$ functions of $s$ for $|s|\leq r_0$and satisfy the following asymptotical expansion for $|s|\leq r_0$
\be                                   \label{m_specr0}
 \left\{\bln
 \lambda_{\pm1}(s)&=\pm \i 2\sqrt{\frac23} s-a_{\pm1}s^2+o(s^2),\quad
 \overline{\lambda_1(s)}=\lambda_{-1}(s),\\
 \lambda_{0}(s) &=-a_0s^2+o(s^2),\\
 \lambda_2(s) &=\lambda_3(s)=-a_2s^2+o(s^2),
 \eln\right.
 \ee
with constants $a_j>0$, $-1\le j\le 2$,  defined in Lemma~\ref{m_eigen_3z}.
\end{thm}

With above spectrum analysis, we can obtain the global existence and the time-asymptotical behavior of unique solution to the Cauchy problem for the linear  mVPB system~\eqref{m_VPB}--\eqref{m_VPB.a} as follows.
\begin{thm}\label{m_rate1} Assume that $f_0\in L^2(\R_v^3;H^N(\R^3_x)\cap L^{q}(\R_x^3))$ for $N\ge 1$ and $q\in[1,2]$. Then there is a globally unique solution $f(x,v,t)=e^{tB_m}f_0(x,v)$ to  the linearized mVPB system~\eqref{m_VPB}--\eqref{m_VPB.a}, which satisfies for any $\alpha,\alpha'\in\N^3$ with  $|\alpha|\le N$, $\alpha'\le \alpha$  and $k=|\alpha-\alpha'|$ that
 \bma
  \mbox{$\sum_{j=0}^4$}\| \da_x  (e^{tB_m}f_0,\chi_j)\|_{L^2_{x}}
&\leq C(1+t)^{-\frac 32\(\frac1q-\frac12\)-\frac k2}(\|\da_x
f_0\|_{L^2_{x,v}}+\|\dx^{\alpha'}f_0\|_{L^{2,q}}), \label{m_V_4}
\\
\|(I-\Delta_x)^{-1}(\da_x e^{tB_m}f_0,\sqrt M)\|_{H^1_x}
&\leq
  C(1+t)^{-\frac 32\(\frac1q-\frac12\)-\frac k2}
   (\|\da_x f_0\|_{L^2_{x,v}}+\|\dx^{\alpha'}f_0\|_{L^{2,q}}),  \label{m_V_4a}
 \\
\|\P_1 (\da_x e^{tB_m}f_0)\|_{L^2_{x,v}}
&\leq
 C(1+t)^{-\frac 32\(\frac1q-\frac12\)-\frac{k+1}{2}}
  (\|\da_x f_0\|_{L^2_{x,v}}+\|\dx^{\alpha'}f_0\|_{L^{2,q}}). \label{m_V_5}
 \ema
In addition, assume that $f_0\in L^2(\R^3_v;H^N(\R^3_x)\cap L^{1}(\R_x^3))$ for $N\ge 1$ and there exist positive constants $d_0,d_1>0$ and a small constant $r_0>0$ so that the Fourier transform $\hat{f}_0(\xi,v)$  of the initial data $f_{0}(x,v)$ satisfies that
 $\inf_{|\xi|\le r_0}|(\hat f_0,\chi_0)|\ge d_0$,   $\inf_{|\xi|\le r_0}|(\hat f_0,\chi_4)|\ge d_1\sup_{|\xi|\le r_0}|(\hat f_0,\chi_0)|$ and $\sup_{|\xi|\le r_0}|(\hat f_{0},v\sqrt M)|=0$.
Then global solution $f(x,v,t)=e^{tB_m}f_0(x,v)$  satisfies for two positive constants $C_2\ge C_1$ that
 \bgr
  C_1(1+t)^{-\frac34-\frac k2}
  \leq\| \Tdx^k(e^{tB_m}f_0,\chi_j)\|_{L^2_{x}} \leq C_2(1+t)^{-\frac34-\frac k2}, \label{m_H_1}
   \\
C_1(1+t)^{-\frac34-\frac k2}
  \leq\|\Tdx^k(I-\Delta_x)^{-1}(e^{tB_m}f_0,\sqrt M)\|_{H^1_x}
 \leq C_2(1+t)^{-\frac34-\frac k2},   \label{m_H_2}
 \\
 C_1(1+t)^{-\frac54-\frac k2}
  \leq \|\Tdx^k\P_1 (e^{tB_m}f_0)\|_{L^2_{x,v}}\leq C_2(1+t)^{-\frac54-\frac k2},\label{m_H_1a}
\egr
for $t>0$ sufficiently large, $ j=0,1,2,3,4,$ and $k\ge 0$.
\end{thm}

Then, we  state the results on the global existence and the optimal time-asymptotical behavior of unique solution to the Cauchy problem for the mVPB system~\eqref{m_VPB4}--\eqref{m_VPB6} blow.

\begin{thm}\label{m_rate3}
Assume that $f_0\in H^N_w\cap L^{2,1}$ for  $N\ge 4$ and $\|f_0\|_{H^N_{w}\cap L^{2,1}}\le \delta_0$ for a constant  $\delta_0>0$ small enough. Then,  there exists a globally unique strong solution $f=f(x,v,t)$ to the mVPB system~\eqref{m_VPB4}-\eqref{m_VPB6} satisfying
 \bma
\mbox{$\sum_{j=0}^4$}\|\dx^k(f(t),\chi_j)\|_{L^2_{x}}+\|\dx^k\Phi(t)\|_{H^1_x}&\le C\delta_0(1+t)^{-\frac34-\frac k2},\\
\|\dx^k\P_1f(t)\|_{L^2_{x,v}}&\le C\delta_0(1+t)^{-\frac54-\frac k2},\\
\|\P_1f(t)\|_{H^N_w} +\|\Tdx \P_0f(t)\|_{L^2_v(H^{N-1}_x)}&\le
C\delta_0(1+t)^{-\frac54},\ema
for $k=0,1$ and $t>0$.
\end{thm}

We shall prove that the above convergence rates are indeed optimal in the following sense.
\begin{thm}\label{m_rate4}
Let the assumptions of Theorem~\ref{m_rate3} hold.
Assume further that there exist positive constants $d_0,d_1>0$ and a small constant $r_0>0$ so that the Fourier transform $\hat{f}_0(\xi,v)$ satisfies
 $\inf_{|\xi|\le r_0}|(\hat f_0,\chi_0)|\ge d_0$, $ \sup_{|\xi|\le r_0}|(\hat f_0,\chi_j)|=0 $ $(j=1,2,3)$ and
 $\inf_{|\xi|\le r_0}|(\hat f_0,\chi_4)|\ge d_1\sup_{|\xi|\le r_0}|(\hat f_0,\chi_0)|$.
Then, the global solution $f$  to the mVPB system~\eqref{m_VPB4}-\eqref{m_VPB6} satisfies
 \bgr
 C_1\delta_0(1+t)^{-\frac34-\frac k2}\le \|\Tdx^k(f(t),\chi_j)\|_{L^2_{x}}\le C_2\delta_0(1+t)^{-\frac34-\frac k2},\label{m_B_1}\\
 C_1\delta_0(1+t)^{-\frac34-\frac k2}\le \|\Tdx^k\Tdx\Phi(t)\|_{L^2_{x}}\le C_2\delta_0(1+t)^{-\frac34-\frac k2},\label{m_B_2}\\
 C_1\delta_0(1+t)^{-\frac54-\frac k2}\le \|\Tdx^k\P_1f(t)\|_{L^2_{x,v}}\le C_2\delta_0(1+t)^{-\frac54-\frac k2},\label{m_B_3}\\
 C_1\delta_0(1+t)^{-\frac34}\le \|f(t)\|_{H^N_{w}}\le
C_2\delta_0(1+t)^{-\frac34},\label{m_B_4}
\egr
for $t>0$ sufficiently large, two positive constants $C_2\ge C_1$, $j=0,1,2,3,4,$  and $k=0,1$. 
\end{thm}
\begin{rem}[Example]
The initial data $f_0=f_1(x,v)$ defined below satisfies the assumptions of Theorem~\ref{m_rate4}
$$
f_1(x,v)
 = d_0e^{\frac{r_0^2}2}e^{-\frac{x^2}2}\chi_0+d_1d_0e^{\frac{r_0^2}2}e^{-\frac{x^2}2}\chi_4.
$$
for a small positive constant $d_0$.
\end{rem}

\begin{rem} The conditions on initial data in above theorems can be applied to the VPB system and the Boltzmann equation to obtain corresponding the optimal time decay rate (refer to \cite{Zhong2012Sci}). Although the global solution to the modified Vlasov-Poisson-Boltzmann (mVPB) equation system takes the same optimal time decay rate $(1+t)^{-3/4}$ as the Boltzmann equation, yet it observed that the hyperbolic waves of the mVPB system propagates at a faster speed due to the influence of electric field.
\end{rem}

\section{Analysis of spectra and semigroup for linear systems}
\setcounter{equation}{0}
\label{sect3}

\subsection{Spectrum and resolvent of linear bVPB system}
\label{sect3.1}

We investigate the spectrum analysis and the decay rate of the solution to the linearized Vlasov-Poisson-Boltzmann type equation \eqref{LVPB2}.
In the followings, we are concerned with the spectral analysis of the operator $\hat B(\xi)$ and optimal time-decay rate of solution to linear VPB type equation \eqref{LVPB5}.

Introduce a weighted Hilbert space $L^2_\xi(\R^3)$ for $\xi\ne 0$
as
$$
 L^2_\xi(\R^3)=\{f\in L^2(\R^3)\,|\,\|f\|_\xi=\sqrt{(f,f)_\xi}<\infty\},
$$
with the inner product defined by
$$
 (f,g)_\xi=(f,g)+\frac1{|\xi|^2}(P_{\rm d} f,P_{\rm d} g).
$$

Since $P_{\rm d}$ is a self-adjoint projection operator, it follows that
 $(P_{\rm d} f,P_{\rm d} g)=(P_{\rm d} f, g)=( f,P_{\rm d} g)$ and hence
 \bq (f,g)_\xi=(f,g+\frac1{|\xi|^2}P_{\rm d}g)=(f+\frac1{|\xi|^2}P_{\rm d}f,g).\label{C_1a}\eq
By
\eqref{C_1a}, we have for any $f,g\in L^2_\xi(\R^3_v)\cap D(\hat{B}(\xi))$,
 \be
 (\hat{B}(\xi)f,g)_\xi=(\hat{B}(\xi) f,g+\frac1{|\xi|^2}P_{\rm d} g)
 =(f,(L+\i(v\cdot\xi)+\frac{\i(v\cdot\xi)}{|\xi|^2}P_{\rm d})g)=(f,\hat{B}(-\xi)g)_\xi. \label{L_7}
\ee

We can regard $\hat{B}(\xi)$ as a linear operator from the space $L^2_\xi(\R^3)$ to itself because
$$
 \|f\|^2\le \|f\|^2_\xi\le(1+|\xi|^{-2})\|f\|^2,\quad \xi\ne 0.
$$

Similarly to the proofs of Lemmas 2.6--2.7 in \cite{Li2}, we have the following lemmas.

\begin{lem}\label{SG_2}
The operator $\hat{B}(\xi)$ generates a strongly continuous contraction semigroup on
$L^2_\xi(\R^3)$, which satisfies \bq
\|e^{t\hat{B}(\xi)}f\|_\xi\le\|f\|_\xi, \quad\mbox{for any}\ t>0,\,f\in
L^2_\xi(\R^3_v). \eq
\end{lem}

\begin{lem}\label{Egn}
For each $\xi\ne 0$, the spectrum of $\hat{B}(\xi)$ on the domain
$\mathrm{Re}\lambda\geq-\nu_0+\delta$ for any $\delta>0$ consists of isolated eigenvalues $\{\lambda_j(\xi)\}$ with
$\mathrm{Re}\lambda_j (\xi)<0$.
\end{lem}

Now denote by $T$ a linear operator on $L^2(\R^3_v)$ or
$L^2_\xi(\R^3_v)$, and we define the corresponding norms of $T$ by
$$
 \|T\|=\sup_{\|f\|=1}\|Tf\|,\quad
 \|T\|_\xi=\sup_{\|f\|_\xi=1}\|Tf\|_\xi.
$$
Obviously, 
 \bq
(1+|\xi|^{-2})^{-1}\|T\|\le \|T\|_\xi\le (1+|\xi|^{-2})\|T\|.\label{eee}
 \eq

First, we consider the spectrum and resolvent sets of $\hat{B}(\xi)$ at
high frequency. To this end, we define
 \bq
  c(\xi)=-\nu(v)-\i(v\cdot\xi),  \label{Cxi}
 \eq
and decompose  $\hat{B}(\xi)$ into \bma
\lambda-\hat{B}(\xi)&=\lambda-c(\xi)-K_1+\frac{\i(v\cdot\xi)}{|\xi|^2}P_{\rm d}\nnm\\
&=(I-K_1(\lambda-c(\xi))^{-1}+\frac{\i(v\cdot\xi)}{|\xi|^2}P_{\rm d}(\lambda-c(\xi))^{-1})(\lambda-c(\xi)).\label{B_d}\ema
Then, we have the estimates on the right hand terms of \eqref{B_d} as follows.

\begin{lem}\label{LP03}
 There exists a constant  $C>0$ so that it holds:
\begin{enumerate}
\item For any $\delta>0$, we have
 \bgr
  \sup_{x\geq-\nu_0+\delta,y\in\R}\|K_1(x+\i y-c(\xi))^{-1}\|
  \leq C\delta^{-15/13}(1+|\xi|)^{-2/13}, \label{T_7}
 \egr

\item For any $\delta>0,\, r_0>0$, there is a constant  $y_0=(2r_0)^{5/3}\delta^{-2/3}>0$ such that
if $|y|\geq y_0$, we have
 \bgr
 \sup_{x\geq -\nu_0+\delta,|\xi|\leq r_0}\|K_1(x+\i y-c(\xi))^{-1}\|
 \leq C\delta^{-7/5}(1+|y|)^{-2/5},\label{T_8}
 \egr
 \item  For any $\delta>0,\, r_0>0$, we have \bgr
   \sup_{x\geq -\nu_0+\delta,y\in\R}
  \|(v\cdot\xi)|\xi|^{-2}P_{\rm d}(x+\i y-c(\xi))^{-1}\|
 \leq C\delta^{-1}|\xi|^{-1},\label{L_9}
 \\
   \sup_{x\geq -\nu_0+\delta,|\xi|\geq r_0}
 \|(v\cdot\xi)|\xi|^{-2}P_{\rm d}(x+\i y-c(\xi))^{-1}\|
 \leq C(r_0^{-1}+1)(\delta^{-1}+1)|y|^{-1}.\label{L_10}
 \egr
\end{enumerate}
\end{lem}
\begin{proof}The proof of \eqref{L_9} and \eqref{L_10} can be found in Lemma 2.3 in \cite{Li2}.
Since $K_1$ satisfies the same properties as $K$ (see  \cite{Yu}): $$\intr |k_1(v,v_*)|dv_*\le C(1+|v|)^{-1},\quad \intr |k_1(v,v_*)|^2dv_*\le C,$$
we can prove \eqref{T_7} and \eqref{T_8} by a same argument as Lemma 2.2.6 in \cite{Ukai3}.
\end{proof}

By Lemma \ref{LP03} and a similar argument as Lemma 2.4 in \cite{Li2}, we have the spectral gap
of the operator $\hat{B}(\xi)$ for high frequency.
\begin{lem}
\label{LP01}
 Let $\lambda(\xi)\in \sigma(\hat{B}(\xi))$ be any eigenvalue of
$\hat{B}(\xi)$ in the domain $\mathrm{Re}\lambda\geq -\nu_0+\delta$
with $\delta>0$ being a constant. Then, for any $r_0>0$, there
exists $\alpha(r_0)>0$ so that
$\mathrm{Re}\lambda(\xi)\leq-\alpha(r_0)$ for all $|\xi|\geq r_0$.
\end{lem}

Then, we investigate the spectrum and resolvent sets of $\hat{B}(\xi)$ at
low frequency. To this end, we decompose $\lambda-\hat{B}(\xi)$ as
follows
\bq \lambda-\hat{B}(\xi)=\lambda P_{\rm d}+\lambda P_r-Q(\xi)+\i P_{\rm d}(v\cdot\xi)P_r+\i P_r(v\cdot\xi)(1+\frac1{|\xi|^2})P_{\rm d},\label{Bd3}\eq
where \bq
Q(\xi)=L_1-\i P_r(v\cdot\xi)P_r.\label{Qxi}
\eq
\begin{lem}\label{LP}
Let $\xi\neq0$ and $Q(\xi)$ defined by \eqref{Qxi}. We have
 \begin{enumerate}
\item If $\lambda\ne0$, then
\bq
\|\lambda^{-1}P_r(v\cdot\xi)(1+\frac1{|\xi|^2})P_{\rm d}\|_\xi\le C(|\xi|+1)|\lambda|^{-1}.\label{S_2}
\eq

\item If $\mathrm{Re}\lambda>-\mu $, then the operator $\lambda P_r-Q(\xi)$ is invertible on $N_1^\bot$ and satisfies
\bma
\|(\lambda P_r-Q(\xi))^{-1}\|&\leq(\mathrm{Re}\lambda+\mu )^{-1},\label{S_3}\\
\|P_{\rm d}(v\cdot\xi)P_r(\lambda P_r-Q(\xi))^{-1}P_r\|_\xi
&\leq C(1+|\lambda|)^{-1}[(\mathrm{Re}\lambda+\mu)^{-1}+1](1+|\xi|)^2.\label{S_5}
\ema
\end{enumerate}
\end{lem}
\begin{proof}
Since
 \bmas
\|\lambda^{-1}P_r(v\cdot\xi)(1+\frac1{|\xi|^2})P_{\rm d}f\|_\xi
\le C|\lambda|^{-1}(|\xi|+\frac1{|\xi|})\|P_{\rm d}f\|\le C|\lambda|^{-1}(|\xi|+1)\|f\|_\xi,
 \emas
we prove \eqref{S_2}.

Then, we show that for any $\lambda\in\mathbb{C}$ with
$\mathrm{Re}\lambda>-\mu $, the operator $\lambda P_r-Q(\xi)=\lambda
 P_r-L_1+\i P_r(v\cdot\xi) P_r$ is  invertible from $N_1^\bot$ to
itself. Indeed, by \eqref{L_4}, we obtain for any $f\in N_1^\bot\cap
D(L_1)$ that
 \bq
 \text{Re}([\lambda  P_r-L_1+\i P_r(v\cdot\xi) P_r]f,f)
 =\text{Re}\lambda(f,f)-(L_1f,f)\geq(\mu+\text{Re}\lambda)\|f\|^2, \label{A_1}
 \eq
which implies that the operator $\lambda P_r-Q(\xi)$ is an one-to-one
map from $N_1^\bot$ to itself so long as $\text{Re}\lambda>-\mu $. The estimate
\eqref{S_3} follows directly from \eqref{A_1}.

By \eqref{S_3} and $\|P_{\rm d}(v\cdot\xi) P_rf\|_\xi\le C(|\xi|+1)\|P_rf\|$, we have
 \bma
 \| P_{\rm d}(v\cdot\xi) P_r(\lambda  P_r-Q(\xi))^{-1} P_rf\|_\xi
 \leq
 C(|\xi|+1)(\mathrm{Re}\lambda+\mu )^{-1}\|f\|.   \label{2.33a}
 \ema
Meanwhile, we can decompose the operator $ P_{\rm d}(v\cdot\xi) P_r(\lambda
 P_r-Q(\xi))^{-1} P_r$ as
 $$
  P_{\rm d}(v\cdot\xi) P_r(\lambda  P_r-Q(\xi))^{-1} P_r
 =\frac1\lambda  P_{\rm d}(v\cdot\xi) P_r+\frac1\lambda  P_{\rm d}(v\cdot\xi) P_rQ(\xi)(\lambda
 P_r-Q(\xi))^{-1} P_r.
 $$
This together with \eqref{S_3} and the fact
 $
 \| P_{\rm d}(v\cdot\xi) P_r Q(\xi)\|\leq
C(1+|\xi|)^2
 $
give
 \bq
 \| P_{\rm d}(v\cdot\xi) P_r(\lambda  P_r-Q(\xi))^{-1} P_rf\|_\xi
 \leq
C|\lambda|^{-1}[(\mathrm{Re}\lambda+\mu )^{-1}+1](1+|\xi|)^2\|f\|. \label{2.33}
 \eq
The combination of the two cases \eqref{2.33a} and \eqref{2.33} yields \eqref{S_5}.
\end{proof}

Consider the eigenvalue problem
\bq \lambda
f=(L_1-\i(v\cdot\xi))f-\frac{\i\sqrt{M}(v\cdot\xi)}{|\xi|^2}\intr
f\sqrt{M}dv.\label{L_2}
\eq
We shall prove that $\hat B(\xi)$ has a spectral gap when $|\xi|$ is sufficiently small.
For convenience,  we shall use the parametrization $\xi=s\omega$ where $s\in\R^1,\ \omega\in \S^2$.

Let $f$ be the eigenfunction of \eqref{L_2}, we rewrite $f$ in the
form $f=f_0+f_1$, where $f_0=P_{\rm d}f=C_0\sqrt M$ and $f_1=(I-P_{\rm d})f=P_rf$.  The
eigenvalue problem \eqref{L_2} can be decomposed into \bma &\lambda
f_0=-P_{\rm d}[\i(v\cdot\xi)(f_0+f_1)],\label{A_2}\\
&\lambda f_1=L_1f_1-P_r[\i(v\cdot\xi)(f_0+f_1)]-\frac{\i(v\cdot\xi)}{|\xi|^2}f_0.\label{A_3}\ema

From Lemma \ref{LP} and \eqref{A_3}, we obtain that for any
$\text{Re}\lambda>-\mu $ \bq f_1=\i[L_1-\lambda
P_r-\i P_r(v\cdot\xi)P_r]^{-1}P_r((v\cdot\xi)f_0+\frac{(v\cdot\xi)}{|\xi|^2}f_0).\label{A_4}\eq

Substituting  \eqref{A_4} into \eqref{A_2} and taking inner product the resulted equation with $\sqrt M$ gives
\bma \lambda C_0=(1+\frac1{|\xi|^2})
(R(\lambda,\xi)(v\cdot\xi)\sqrt{M},(v\cdot\xi)\sqrt{M})C_0.\label{A_6}\ema
where $ R(\lambda,\xi)=[L_1-\lambda P_r-\i P_r(v\cdot\xi)P_r]^{-1}.$

By changing variable $(v\cdot\xi)\to sv_1$ and using the rotational invariance of the operator $L_1$, we have the following transformation.

\begin{lem}Let  $e_1=(1,0,0)$, $\xi=s\omega$ with $s\in \R, \omega\in \S^2$.
Then
\bma
(R(\lambda,\xi)(v\cdot\xi)\sqrt{M},(v\cdot\xi)\sqrt{M})=s^2(R(\lambda,se_1)(v_1\sqrt{M}),v_1\sqrt{M}).\label{T_1}
\ema
\end{lem}

With the help of \eqref{T_1}, we rewrite
\eqref{A_6} in the form  \bma \lambda C_0=
(1+s^2)(R(\lambda,se_1)\chi_1,\chi_1)C_0.\label{eigen}\ema

Denote
\bma
D(\lambda,s)&=(1+s^2)(R(\lambda,se_1)\chi_1,\chi_1).\label{ddd}
\ema

\begin{lem}\label{eigen_1}
There exists a small constant $r_0>0$ such that the equation $\lambda =D(\lambda,s)$ has no solution for ${\rm Re}\lambda\ge -\frac14a_0$ and $|s|\le r_0$, where
 $$
 a_0=\min\{\mu,-(L_1^{-1}\chi_1,\chi_1)\}>0.
 $$
\end{lem}
\begin{proof}We prove this lemma by three steps. First, we  prove
\bq x\ne D(x,0):=D_0(x) \quad {\rm for } \quad x>-a_0.\label{C_4}\eq
 Indeed, since
$D_0'(x)=\|(L_1-x P_r)^{-1}\chi_1\|^2> 0,$
we have
$D_0(x)< D_0(0)=(L_1^{-1}\chi_1,\chi_1)<-a_0$ for $-\mu<x< 0.$ Thus \eqref{C_4} holds for $-a_0< x< 0$.
For $x\ge 0$, since
$D_0(x)
=(L_1g,g)-x\|g\|^2<0$
with
$ g=(L_1-x P_r)^{-1}\chi_1\in N_1^\bot,$
it follows that  \eqref{C_4} holds for $x\ge0$. Thus  we obtain \eqref{C_4}.

Second, we  show
\bq \lambda\ne D(\lambda,0):=D_0(\lambda) \quad {\rm for} \quad {\rm Re}\lambda>-\frac12a_0.\label{C_3}\eq
 In the case of ${\rm Im}\lambda=0$,  \eqref{C_3} follows from \eqref{C_4}.
Assume that $\lambda=x+\i y$ with $y\ne 0$. Note that
\bmas {\rm Re}D_0(\lambda)
=(L_1h,h)-x(h,h),\quad
{\rm Im}D_0(\lambda)
=y\|h\|^2,
\emas
where
$h=(L_1-\lambda P_r)^{-1}\chi_1\in N_1^\bot.$
If there is a $\lambda=x+\i y $ with $x>-\frac12a_0$ such that $\lambda=D_0(\lambda)$, then
\bma x&=(L_1h,h)-x(h,h),\label{C_5}\\
y&=y\|h\|^2.\label{C_6}\ema
By \eqref{C_6} and $y\ne 0$, we have
$ \|h\|=1.$
This together with \eqref{C_5} gives
$2x=(L_1h,h)\le-\mu\|h\|^2=-\mu$,  which  contradicts to $2x>-a_0\ge -\mu$. Thus we obtain \eqref{C_3}.

Finally, by \eqref{S_3} we have
$\lim_{|\lambda|\to\infty}|\lambda-D_0(\lambda)|=\infty$ for ${\rm Re}\lambda> -\mu$.  This togethers with \eqref{C_3} and the continuity of $D_0(\lambda)$ imply that there is a constant $\delta_0>0$ so that
$ |\lambda-D_0(\lambda)|>\delta_0$ for ${\rm Re}\lambda\ge -\frac14a_0$.
Since
\bmas
|D(\lambda,s)-D(\lambda,0)|&\le |([R(\lambda,se_1)-R(\lambda,0)]\chi_1,\chi_1)|+s^2|(R(\lambda,se_1)\chi_1,\chi_1)|\le C(s+s^2), 
\emas
we obtain
$$  |\lambda-D(\lambda,s)|\ge |\lambda-D(\lambda,0)|-|D(\lambda,0)-D(\lambda,s)|>0,\quad {\rm Re}\lambda\ge -\frac14a_0,\,\,|s|\le r_0,$$
for $r_0>0$ small enough. This completes the proof of lemma.
\end{proof}

\begin{lem}\label{spectrum2} It holds that
\begin{enumerate}
\item  There is a constant $a_1>0$ such that it holds for all  $\xi\ne 0$ that
\bq \sigma(\hat{B}(\xi))\subset\{\lambda\in\mathbb{C}\,|\, {\rm Re}\lambda<-a_1\}.\label{sg1}\eq

\item  For any $\delta>0$ and all $\xi\ne 0$, there exists $y_1(\delta)>0$ such that
\bq
 \{\lambda\in\mathbb{C}\,|\,
     \mathrm{Re}\lambda\ge-\mu +\delta,\,|\mathrm{Im}\lambda|\geq y_1\}
 \cup\{\lambda\in\mathbb{C}\,|\,\mathrm{Re}\lambda\ge -a_1\}
 \subset\rho(\hat{B}(\xi)).                           \label{rb1}
\eq
\end{enumerate}
\end{lem}

\begin{proof}
By Lemma \ref{eigen_1} and Lemma \ref{LP01}, we can choose  $0<a_1<\min\{\alpha(r_0),\frac13a_0\}$ with $r_0, a_0>0$ given by  Lemma \ref{eigen_1} and $\alpha(r_0)>0$ given by Lemma \ref{LP01} so that if $\lambda(\xi)\in\sigma(\hat{B}(\xi))$,  
then ${\rm Re}\lambda(\xi)<-a_1$ for all $\xi\ne0$.
This proves \eqref{sg1}.
By \eqref{T_7} and \eqref{L_9}, there exists $r_1>0$ large such that if $\mathrm{Re}\lambda\geq -\nu_0+\delta$ and $|\xi|\geq r_1$, then
\bq
\|K_1(\lambda-c(\xi))^{-1}\|\leq1/4,\quad
\|(v\cdot\xi)|\xi|^{-2}P_{\rm d}(\lambda-c(\xi))^{-1}\|\leq1/4.\label{bound}\eq
This implies that the operator
$I+K_1(\lambda-c(\xi))^{-1}+\i(v\cdot\xi)|\xi|^{-2}P_{\rm d}(\lambda-c(\xi))^{-1}$
is invertible on $L^2_\xi(\R^3_v)$, which together with \eqref{B_d} yields that $(\lambda-\hat{B}(\xi))$ is also invertible on $L^2_\xi(\R^3_v)$ and
satisfies \bq
(\lambda-\hat{B}(\xi))^{-1}=(\lambda-c(\xi))^{-1}(I-K(\lambda-c(\xi))^{-1}+\frac{\i(v\cdot\xi)}{|\xi|^2}P_{\rm d}(\lambda-c(\xi))^{-1})^{-1}.
\label{E_6}\eq
Thus $\rho(\hat{B}(\xi))\supset
\{\lambda\in\mathbb{C}\,|\,{\rm Re}\lambda\ge-\nu_0+\delta_1\}$ for
$|\xi|>r_1$.

By Lemmas \ref{LP}, we have for $\rm{Re}\lambda>-\mu $ and
$\lambda\neq0$  that the operator
$\lambda P_{\rm d}+\lambda P_r-Q(\xi)$ is invertible on
$L^2_\xi(\R^3_v)$ and it satisfies
 \bmas
  (\lambda P_{\rm d}+\lambda P_r-Q(\xi))^{-1} =\lambda^{-1}P_{\rm d}+(\lambda P_r-Q(\xi))^{-1}P_r,
 \emas
because the operator $\lambda P_{\rm d}$ is orthogonal to $\lambda
P_r-Q(\xi)$. Therefore, we can re-write \eqref{Bd3} as
\bma
 \lambda-\hat{B}(\xi)
&=(I+Y_1(\lambda,\xi))(\lambda P_{\rm d}+\lambda P_r-Q(\xi)),\nnm
 \\
Y_1(\lambda,\xi)&=: \i P_{\rm d}(v\cdot\xi)P_r(\lambda P_r-Q(\xi))^{-1}P_r+\i \lambda^{-1}P_r(v\cdot\xi)(1+\frac1{|\xi|^2})P_{\rm d}.  \label{Y_1}
 \ema
For the case $|\xi|\leq r_1$, by \eqref{S_2} and
\eqref{S_5} we can choose $y_1=y_1(\delta)>0$ such that it holds
for $\mathrm{Re}\lambda\ge-\mu +\delta$ and
$|\mathrm{Im}\lambda|\geq y_1$ that
 \bq
 \|\lambda^{-1}P_r(v\cdot\xi)(1+\frac1{|\xi|^2})P_{\rm d}\|_\xi\leq \frac14,
 \quad
 \|P_{\rm d}(v\cdot\xi)P_r(\lambda P_r-Q(\xi))^{-1}P_r\|_\xi\leq\frac14.\label{bound_1}
 \eq
This implies that the operator $I+Y_1(\lambda,\xi)$ is invertible on
$L^{2}_\xi(\R^3_v)$ and thus  $\lambda-\hat{B}(\xi)$ is invertible on $L^{2}_\xi(\R^3_v)$ and satisfies
 \bma
 (\lambda-\hat{B}(\xi))^{-1}
 =&[\lambda^{-1}P_{\rm d}+(\lambda P_r-Q(\xi))^{-1}P_r](I+Y_1(\lambda,\xi))^{-1}.\label{S_8}
 \ema
Therefore, $\rho(\hat{B}(\xi))\supset \{\lambda\in\mathbb{C}\,|\,{\rm
Re}\lambda\ge-\mu+\delta_1, |{\rm Im}\lambda|\ge y_1\}$ for $|\xi|\le
r_1$. This and \eqref{sg1} lead to \eqref{rb1}.
\end{proof}


\subsection{Exponential decay of semigroup  for linear bVPB}  %

\begin{lem}\label{bound_2}Let $Y_1(\lambda,\xi)$ be defined by \eqref{Y_1} and let $r_0>0$ be given by Lemma \ref{eigen_1} and $a_1>0$ be given by Lemma \ref{spectrum2}. If $-a_1<{\rm Re}\lambda<0$, then the operator $I+Y_1(\lambda,\xi)$ is invertible on $L^2_\xi(\R^3_v)$ and satisfies
\bq \sup_{0<|\xi|<r_0,{\rm Im}\lambda\in \R}\|(I+Y_1(\lambda,\xi))^{-1}\|_\xi\le C.\label{S_9}\eq
\end{lem}
\begin{proof}
Let $\lambda=x+\i y$.
By \eqref{S_2} and \eqref{S_5}, there exists $R>0$ large enough  such that if $|y|\ge R$, then
$$\|\lambda^{-1}P_r(v\cdot\xi)(1+\frac1{|\xi|^2})P_{\rm d}\|_\xi\le \frac14,\quad \|P_{\rm d}(v\cdot\xi)P_r(\lambda P_r-Q(\xi))^{-1}P_r\|_\xi\le \frac14,$$
which yields
$$\|(I+Y_1(\lambda,\xi))^{-1}\|_\xi\le 2.$$
Thus it is sufficient to prove \eqref{S_9} for $|y|\le R$. We make use of the argument of contradiction. Indeed, if \eqref{S_9} does not holds for $|y|\le R$, namely, there are subsequences $\xi_n$, $\lambda_n=x+\i y_n$ with $|\xi_n|\le r_0$, $|y_n|\le R$, and $f_n,g_n$ with $\|f_n\|_{\xi_n}\to 0$ $(n\to \infty)$, $\|g_n\|_{\xi_n}=1$ such that
$$(I+Y_1(\lambda_n,\xi_n))^{-1}f_n=g_n.$$
This gives
$$f_n=g_n+\i P_{\rm d}(v\cdot\xi_n)P_r(\lambda_n P_r-Q(\xi_n))^{-1}P_rg_n+\i \lambda^{-1}_n P_r(v\cdot\xi_n)(1+\frac1{|\xi_n|^2}) P_{\rm d}g_n,$$
and then
\bma
P_{\rm d}f_n&=P_{\rm d}g_n+\i P_{\rm d}(v\cdot\xi_n)P_r(\lambda_n P_r-Q(\xi_n))^{-1}P_rg_n,\label{m_2}\\
P_r f_n&=P_r g_n+\i \lambda^{-1}_n P_r(v\cdot\xi_n)(1+\frac1{|\xi_n|^2})P_{\rm d}g_n.\label{m}
\ema
Substituting \eqref{m} into \eqref{m_2}, we obtain
\bma
&P_{\rm d}g_n-P_{\rm d}f_n+\i P_{\rm d}(v\cdot\xi_n)P_r(\lambda_n P_r-Q(\xi_n))^{-1}P_rf_n \nnm\\
&+\lambda_n^{-1}P_{\rm d} (v\cdot\xi_n)P_r(\lambda_n P_r-Q(\xi_n))^{-1} P_r(v\cdot\xi_n)(1+\frac1{|\xi_n|^2})P_{\rm d}g_n=0.\label{m_4}
\ema
Since $\|f_n\|_{\xi_n}\to 0$ as $n\to \infty$, it follows from \eqref{m_4} and \eqref{S_5} that
\bmas
\lim_{n \to\infty}\|P_{\rm d}g_n+\lambda_n^{-1}P_{\rm d} (v\cdot\xi_n)P_r(\lambda_n P_r-Q(\xi_n))^{-1}P_r(v\cdot\xi_n)(1+\frac1{|\xi_n|^2})P_{\rm d}g_n \|_{\xi_n}=0.
\emas
Let $P_{\rm d}g_n=C_n\sqrt M$. We obtain
\bma
\lim_{n \to\infty}\sqrt{1+\frac1{|\xi_n|^2}}\frac{|C_n|}{|\lambda_n|}|\lambda_n+(|\xi_n|^2+1)((\lambda_n P_r-Q(|\xi_n|e_1))^{-1}(v_1\sqrt M),v_1\sqrt M)|=0.
\ema
Since $\sqrt{1+\frac1{|\xi_n|^2}}|C_n|\le 1,|\xi_n|\le r_0, |y_n|\le R$, there is a subsequence $(\xi_{n_j},\lambda_{n_j},C_{n_j})$ such that $\sqrt{1+\frac1{|\xi_{n_j}|^2}}C_{n_j}\to A_0,\xi_{n_j}\to\xi_0,\lambda_{n_j}\to \tilde{\lambda}=x+\i y\ne0$.
Thus \bq \frac{|A_0|}{|\tilde{\lambda}|}|\tilde{\lambda}+(|\xi_0|^2+1)((\tilde{\lambda} P_r-Q(|\xi_0|e_1))^{-1}(v_1\sqrt M),v_1\sqrt M)|=0.\label{m_3}\eq

It is easy to verify  that  $A_0\ne0$. Indeed, if not,  we have
$\dlim_{j\to\infty}\| P_r(v\cdot\xi_{n_j})(1+|\xi_{n_j}|^{-2})P_{\rm d}g_{n_j}\|=0 $ and then $\dlim_{j\to\infty}\|P_r g_n\|=0$ due to \eqref{m}. Thus $\dlim_{j\to\infty}\|g_{n_j}\|_{\xi_{n_j}}=0$, which  contradicts to $\|g_n\|_{\xi_n}=1$. Therefore, from \eqref{m_3} we have
$$\tilde{\lambda}=D(\tilde{\lambda},|\xi_0|),\quad |\xi_0|\le r_0,\,\,\,-a_1<{\rm Re}\tilde{\lambda}<0.$$
This implies that $\tilde{\lambda}$ is an eigenvalue of $\hat{B}(\xi_0)$ with ${\rm Re}\tilde{\lambda}>-a_1$, which contradicts to Lemma \ref{spectrum2}.
\end{proof}

\begin{lem}\label{SG_3}
 The operator $Q(\xi)=L_1-\i P_r(v\cdot\xi)P_r$ generates a strongly continuous contraction
semigroup on $N_1^\bot$, which satisfies for any $t>0$ and $f\in N_1^\bot\cap L^2(\R^3_v)$ that
  \bq
    \|e^{tQ(\xi)}f\|\leq e^{-\mu t}\|f\|. \label{decay_1}
 \eq
In addition, for any $x>-\mu $ and $f\in N_1^\bot\cap L^2(\R^3_v)$ it holds that
\bq \int^{+\infty}_{-\infty}\|[(x+\i y) P_r-Q(\xi)]^{-1}f\|^2dy\leq
\pi(x+\mu )^{-1}\|f\|^2.\label{S_4}\eq
\end{lem}
\begin{proof}
Since $Q(\xi)$ and $Q(\xi)^*=Q(-\xi)$ are dissipative operators satisfying \eqref{A_1}, we can prove \eqref{decay_1} and \eqref{S_4} by applying a similar argument as Lemma 3.1 in \cite{Li2}.
\end{proof}

\begin{thm}\label{rate1}
The semigroup $e^{t\hat{B}(\xi)}$ satisfies
\bq \|e^{t\hat{B}(\xi)}f\|_\xi\le Ce^{-\frac12a_1t}\|f\|_\xi, \quad f\in L^2_\xi(\R^3_v),\label{B_0}\eq  where $a_1>0$ is a constant given by Lemma \ref{spectrum2}.
\end{thm}
\begin{proof}Since $D(\hat{B}(\xi)^2)$ is dense in $L^2_\xi(\R^3_v)$ by Theorem 2.7 in p.6 of \cite{Pazy}, we only need to prove \eqref{B_0} for $f\in D(\hat{B}(\xi)^2)$. By Corollary 7.5 in p.29 of  \cite{Pazy},
the semigroup $e^{t\hat{B}(\xi)}$ can be represented by
\bq
e^{t\hat{B}(\xi)}f=\frac1{2\pi
\i}\int^{\kappa+\i\infty}_{\kappa-\i\infty}e^{\lambda
t}(\lambda-\hat{B}(\xi))^{-1}fd\lambda, \quad f\in D(\hat{B}(\xi)^2),\,\,\,
\kappa>0.\label{V_6}
\eq

By \eqref{S_8}, we rewrite  $(\lambda-\hat{B}(\xi))^{-1}$ for $|\xi|\le r_0$ as
\bma
(\lambda-\hat{B}(\xi))^{-1}=\lambda^{-1}P_{\rm d}+(\lambda P_r-Q(\xi))^{-1}P_r-Z_1(\lambda,\xi),\label{a_2}
\ema
where
\bmas
Z_1(\lambda,\xi)&=(\lambda^{-1}P_{\rm d}+(\lambda P_r-Q(\xi))^{-1}P_r)(I+Y_1(\lambda,\xi))^{-1}Y_1(\lambda,\xi),\\
Y_1(\lambda,\xi)&=\i P_{\rm d}(v\cdot\xi)P_r(\lambda P_r-Q(\xi))^{-1}P_r+\i \lambda^{-1}P_r(v\cdot\xi)(1+\frac1{|\xi|^2})P_{\rm d}.
\emas
Substituting \eqref{a_2} into \eqref{V_6}, we have the following decomposition of the semigroup $e^{t\hat{B}(\xi)}$ for $|\xi|\le r_0$
\bma
e^{t\hat{B}(\xi)}f= e^{tQ(\xi)}P_rf -\frac1{2\pi
\i}\int^{\kappa+\i\infty}_{\kappa-\i\infty}e^{\lambda
t}Z_2(\lambda,\xi)fd\lambda,\label{c}
\ema
with
$Z_2(\lambda,\xi)=Z_1(\lambda,\xi)-\lambda^{-1}P_{\rm d}.$
Denote
\bq
U_{\kappa,l}=\int^{l}_{-l}e^{(\kappa+\i y)t}Z_2(\kappa+\i y,\xi)fdy1_{|\xi|\le r_0},\label{UsN}
\eq
where the constant $l>0$ is chosen large enough so that $l>y_1$ with $y_1$ defined in Lemma \ref{spectrum2}. Since the operator $Z_2(\lambda,\xi)=(\lambda P_r-Q(\xi))^{-1}P_r-(\lambda -\hat{B}(\xi))^{-1}$ is analytic
on the domain ${\rm Re}\lambda>-a_1$ with $a_1>0$ given by Lemma \ref{spectrum2}, we can shift the integration
$U_{\kappa,l}$ from the line ${\rm Re}\lambda=\kappa>0$ to
${\rm Re}\lambda=-\frac{a_1}2$ to deduce
\bq
U_{\kappa,l}=U_{-\frac{a_1}2,l}+H_l,\label{UsN1}
\eq
where
\bmas H_l=\(\int^{\kappa+il}_{-\frac{a_1}2+il}-\int^{\kappa-il}_{-\frac{a_1}2-il}\)e^{\lambda
t}Z_2(\lambda,\xi)fd\lambda1_{|\xi|\le r_0}. \emas

 By Lemma \ref{LP}, it can be seen that
\be
 \|H_l\|_\xi\rightarrow0, \quad\mbox{as}\quad l\rightarrow\infty.  
 \ee
By Cauchy Theorem, we have
 \bmas
 &\lim_{l\to\infty}\bigg|\int^{-\frac{a_1 }2+\i l}_{-\frac{a_1 }2-\i l}
 e^{\lambda t}\lambda^{-1}d\lambda\bigg|
   =0,
  \emas
which gives rise to
 \be
 \lim_{l\to\infty} U_{-\frac{a_1 }2,l}(t)
 =U_{-\frac{a_1 }2,\infty}(t)
 =:\int^{-\frac{a_1 }2+\i \infty}_{-\frac{a_1 }2-\i \infty}
  e^{\lambda t}Z_1(\lambda,\xi)fd\lambda.            \label{UsN3}
\ee
By Lemma \ref{bound_2}, it holds that
$
\sup\limits_{|\xi|\le r_0, y\in\R}
\|[I+Y_1(-\frac{a_1 }2+\i y,\xi)]^{-1}\|_\xi\le C.
$
Thus, we have for any $f,g\in L^2_\xi(\R^3_v)$  that
  \bmas
 |(U_{-\frac{a_1 }2,\infty}(t)f,g)_\xi|
 &\leq
 C e^{-\frac{a_1 t}2}\int^{+\infty}_{-\infty}
  \( \|[\lambda P_r-Q(\xi)]^{-1}P_rf\|
    +|\lambda|^{-1}\|P_{\rm d}f\|_\xi\)
    \nnm \\
&\quad
   \times\(\|[\overline{\lambda} P_r-Q(-\xi)]^{-1}P_rg\|
            +|\overline{\lambda}|^{-1}\|P_{\rm d}g\|_\xi\) dy,\quad \lambda=-\frac{a_1}2+\i y.
 \emas
This together with \eqref{S_4} yields
 $
 |(U_{-\frac{a_1}2,\infty}(t)f,g)_\xi|
  \leq
 C e^{-\frac{a_1 t}2}\|f\|_\xi\|g\|_\xi,
 $
and
 \bq
\|U_{-\frac{a_1}2,\infty}(t)\|_\xi
   \le C e^{-\frac{a_1 t}2}. \label{UsN4}
 \eq
Therefore, we conclude from \eqref{c}--\eqref{UsN4} that
\bma
e^{t\hat{B}(\xi)}f=e^{tQ(\xi)}P_rf+U_{-\frac{a_1}2,\infty}(t),\quad |\xi|\le r_0.\label{low}
\ema

By \eqref{E_6}, we have for $|\xi|> r_0$ that
\bma
(\lambda-\hat{B}(\xi))^{-1}=(\lambda-c(\xi))^{-1}+Z_3(\lambda,\xi),\label{V_7}
\ema
where
\bmas
Z_3(\lambda,\xi)&=(\lambda-c(\xi))^{-1}[I-Y_2(\lambda,\xi)]^{-1}Y_2(\lambda,\xi),\\
Y_2(\lambda,\xi)&=(K_1-\i(v\cdot\xi)|\xi|^{-2}P_{\rm d})(\lambda-c(\xi))^{-1}.
\emas
Substituting \eqref{V_7} into \eqref{V_6} yields
\bma
e^{t\hat{B}(\xi)}f=e^{tc(\xi)}f+\frac1{2\pi
\i}\int^{\kappa+\i\infty}_{\kappa-\i\infty}e^{\lambda
t}Z_3(\lambda,\xi)fd\lambda.\label{V_3b}
\ema
Denote
\bq
V_{\kappa,l}=\int^{l}_{-l}e^{(\kappa+\i y)
t}Z_3(\kappa+\i y,\xi)dy1_{|\xi|> r_0}  \label{VsN}
\eq
for  sufficiently large constant $l>0$ as  in \eqref{UsN}. Since the operator $Z_3(\lambda,\xi)$ is analytic on the domain ${\rm Re}\lambda\ge -a_1$, we can again shift the integration $V_{\kappa,l}$ from the line ${\rm Re}\lambda=\kappa>0$ to Re$\lambda=-\frac{a_1}2$,  to deduce
\bq
V_{\kappa,l}=V_{-\frac{a_1}2,l}+I_l,
\eq
where
$$I_l=\(\int^{-\kappa+il}_{-\frac{a_1}2+il}-\int^{-\kappa-il}_{-\frac{a_1}2-il}\)e^{\lambda
t}Z_3(\lambda,\xi)fd\lambda1_{|\xi|> r_0}.$$
By Lemma \ref{LP03}, it can be shown that
\bgr
 \|I_l\|\rightarrow0, \quad \mbox{as}\quad  l\rightarrow\infty.   \label{VsNa}
  \egr
Using the facts that
(see Lemma 2.2.13 in   \cite{Ukai3} and Lemma 3.3 in \cite{Li2})
$$
\int^{+\infty}_{-\infty}\|(x+\i y-c(\xi))^{-1}f\|^2dy\leq
\pi(x+\nu_0)^{-1}\|f\|^2,\quad \sup_{|\xi|>r_0, y\in\R} \|[I-Y_2(-\frac{a_1}2+\i y,\xi)]^{-1}\| \le C,
$$
we have
\bma
|(V_{-\frac{a_1}2,\infty}(t),g)|&\leq
C(\|K_1\|+r_0^{-1})e^{-\frac{a_1}2 t}\int^{+\infty}_{-\infty}\|(\lambda-c(\xi))^{-1}f\|\|(\overline{\lambda}-c(-\xi))^{-1}g\|dy\nnm\\
&\leq Ce^{-\frac{a_1}2 t}(\nu_0-\frac{a_1}2)^{-1}\|f\|\|g\|,\quad \lambda=-\frac{a_1}2+\i y.\label{VsNb1}
\ema
By \eqref{VsNb1} and the fact
$\|f\|^2\leq\|f\|_\xi^2\leq(1+r_0^{-2})\|f\|^2$ for $|\xi|> r_0,$
we have
  \bq \|V_{-\frac{a_1}2,\infty}(t)\|_\xi\leq Ce^{-\frac{a_1t}2}(\nu_0-\frac{a_1}2)^{-1}.\label{VsN2}
\eq
We conclude from \eqref{V_3b} and \eqref{VsN}--\eqref{VsN2} that
\bma
e^{t\hat{B}(\xi)}f=e^{tc(\xi)}f1_{\{|\xi|> r_0\}}+V_{-\frac{a_1}2,\infty}(t), \quad |\xi|> r_0.  \label{high}
\ema

Finally, it follows from \eqref{low} and \eqref{high} that
\bma
e^{t\hat{B}(\xi)}f=(e^{tQ(\xi)}P_rf+U_{-\frac{a_1}2,\infty}(t))1_{\{|\xi|\leq r_0\}}+(e^{tc(\xi)}f+V_{-\frac{a_1}2,\infty}(t))1_{\{|\xi|> r_0\}}.
\ema
In particular, $e^{t\hat{B}(\xi)}f$  satisfies \eqref{B_0} in terms of \eqref{decay_1}, \eqref{UsN4},  \eqref{VsN2} and the estimate $\| e^{tc(\xi)}1_{\{|\xi|> r_0\}}\|_\xi \le  Ce^{-\nu_0t}$ because of \eqref{Cxi} and \eqref{nuv}.
\end{proof}

Define a Sobolev space of function $f=f(x,v)$ by $ H^N_P=\{f\in L^2(\R^3_x\times\R^3_v)\,|\, \|f\|_{H^N_P}<\infty\,\}\ (L^2_P=H^0_P)$
with the norm $\|\cdot\|_{H^N_P}$ defined by
\bmas
\|f\|_{H^N_P}&=\bigg(\intr (1+|\xi|^2)^N \|\hat{f}\|^2_\xi d\xi \bigg)^{1/2}\\&=\bigg(\intr (1+|\xi|^2)^N\bigg(\intr
|\hat{f}|^2dv+\frac1{|\xi|^2}\lt|\intr \hat{f}\sqrt{M}dv\rt|^2\bigg)d\xi \bigg)^{1/2},
\emas
where $\hat{f}=\hat{f}(\xi,v)$ is the Fourier
transformation of $f(x,v)$. Note that
$
\|f\|^2_{H^N_P}=\|f\|^2_{L^2_v(H^N_x)}+\| \nabla_x\Delta_x^{-1}
(f,\sqrt{M})\|^2_{H^N_x}.
$
For any $f_0\in H^N_P$, we denote $e^{tB}f_0$ by
 \bq
  e^{tB}f_0=(\mathcal{F}^{-1}e^{t\hat{B}(\xi)}\mathcal{F})f_0.
  \eq
By Lemma \ref{SG_2}, it holds
 $$
 \|e^{tB}f_0\|_{H^N_P}=\intr (1+|\xi|^2)^N\|e^{t\hat{B}(\xi)}\hat f_0\|^2_\xi d\xi\le \intr (1+|\xi|^2)^N\|\hat f_0\|^2_\xi d\xi
=\|f_0\|_{H^N_P}.
 $$
This means that the linear operator $B$ generates a strongly continuous contraction semigroup $e^{tB}$ in $H^N_P$, and therefore $f(x,v,t)=e^{tB}f_0(x,v)$ is a global solution to the IVP~\eqref{LVPB2} for any $f_0\in H^N_P$. What left is to establish the time-decay rates of the global solution.


\begin{proof}[\it \underline{Proof of Theorem~\ref{rate2}}] The property \eqref{sg1z}
of  the spectrum  to the operator $\hat{B}(\xi)$ follow from Lemma~\ref{spectrum2}. And the proof of 
\eqref{sg2} and \eqref{eigen_2} can be found in \cite{Ellis,Ukai1}.
\end{proof}


\begin{proof}[\it \underline{Proof of Theorem~\ref{rate2}}] First, the proof of \eqref{V_1}--\eqref{H_2} can be found in \cite{Zhong2012Sci}.
Next,  we prove \eqref{D_2}. By \eqref{B_0} and using the fact that
\bmas
\intr \frac{(\xi^\alpha)^2}{|\xi|^2}\lt|(\hat{f}_0,\sqrt{M})\rt|^2d\xi
&\leq \sup_{|\xi|\leq 1} \lt|(\hat{f}_0,\sqrt{M})\rt|^2\int_{|\xi|\leq
1}\frac{1}{|\xi|^2}d\xi+\int_{|\xi|> 1}
(\xi^\alpha)^2\lt|(\hat{f}_0,\sqrt{M})\rt|^2d\xi\\
&\leq C(\|(f_0,\sqrt{M})\|^2_{L^1_x}+\|\da_x f_0\|^2_{L^2_{x,v}}),
\emas
we have
\bmas
\|\dxa e^{tB}f_0\|^2_{L^2_{x,v}}+\|\dxa \Tdx\Phi\|^2_{L^2_x}&=\intr (\xi^{\alpha})^2(\|e^{t\hat{B}(\xi)}\hat f_0\|^2_{L^2_v}+\frac{1}{|\xi|^2}|(e^{t\hat{B}(\xi)}\hat f_0,\sqrt
M)|^2)d\xi \nnm\\
&\leq C\intr (\xi^\alpha)^2e^{-a_1 t}(\|\hat
f_0\|^2_{L^2_v}+\frac{1}{|\xi|^2}|(\hat{f}_0,\sqrt{M})|^2)d\xi\nnm\\
&\le C e^{-a_1 t}(\|f_0\|^2_{L^{2,1}}+\|\da_x f_0\|^2_{L^2_{x,v}}).
\emas
This proves the theorem.
\end{proof}

\subsection{Analysis of spectrum and semigroup for linear  mVPB system}
\label{sect3.m}
We establish in this subsection the analysis of spectrum and resolvent and investigate the long time properties of the semigroup related the linear operator $B_m$ to the mVPB~\eqref{m_VPB} by applying the similar arguments as the above.

Introduce a weighted Hilbert space $L^2_m(\R^3)$
as
$$ L^2_m(\R^3)=\{f\in
L^2(\R^3)\,|\,\|f\|_{m}=\sqrt{\<f,f\>_{\xi}}<\infty\} $$ equipped with the inner
product
$$
\<f,g\>_{\xi}=(f,g)+\frac1{1+|\xi|^2}(P_{\rm d} f,P_{\rm d}g).
$$

We have for any $f,g\in L^2_\xi(\R^3_v)\cap D(\hat{B}_m(\xi))$,
$$
 \<\hat{B}_m(\xi)f,g\>_\xi=(\hat{B}_m(\xi) f,g+\frac1{1+|\xi|^2}P_{\rm d} g)
 =(f,(L+\i(v\cdot\xi)+\frac{\i(v\cdot\xi)}{1+|\xi|^2}P_{\rm d})g)=\<f,\hat{B}_m(-\xi)g\>_\xi.
$$

Since
$$\|f\|\le \|f\|_m\le 2\|f\|,\quad \forall\ \xi\in \R^3,$$
we can regard $\hat B(\xi)$ as an operator from $L^2(\R^3)$ to itself or from $L^2_\xi(\R^3)$ to itself.
Then, similar to Lemma \ref{SG_2} and \ref{Egn}, we have the following two lemmas.

\begin{lem}\label{m_SG_1}
The operator $\hat{B}_m(\xi)$ generates a strongly continuous contraction semigroup on
$L^2_m(\R^3_v)$ satisfying $$
\|e^{t\hat{B}_m(\xi)}f\|_m\le\|f\|_m \quad\mbox{for any}\ t>0,\,f\in
L^2_m(\R^3_v). $$
\end{lem}

\begin{lem}\label{m_specr1}
For each $\xi\in\R^3$, the spectrum $\lambda\in\sigma(\hat{B}_m(\xi))$ on  the domain
$\mathrm{Re}\lambda\geq-\nu_0+\delta$ for any constant $\delta>0$ consists of isolated eigenvalues with
${\rm Re}\lambda(\xi)<0$ for $\xi\ne 0$ and $\lambda(\xi)=0$ only if $\xi=0$.
\end{lem}

Next, we deal with the spectrum and resolvent sets of $\hat{B}_m(\xi)$.  To this end, we  decompose  $\hat{B}_m(\xi)$ into
 \bma
 \lambda-\hat{B}_m(\xi)
&=\lambda-c(\xi)-K+\frac{\i(v\cdot\xi)}{1+|\xi|^2}P_{\rm d}\nnm\\
&=(I-K(\lambda-c(\xi))^{-1}+\frac{\i(v\cdot\xi)}{1+|\xi|^2}P_{\rm d}(\lambda-c(\xi))^{-1})(\lambda-c(\xi)), \label{m_B_d}
\ema
where $c(\xi)$ is defined by \eqref{Cxi}. Then, we have the estimates on the right hand terms of \eqref{m_B_d} as follows.

\begin{lem}
\label{m_LP}
 There exists a constant  $C>0$ so that it holds:
\begin{enumerate}
\item For any $\delta>0$, we have
 \bgr
  \sup_{x\geq-\nu_0+\delta,y\in\R}\|K(x+\i y-c(\xi))^{-1}\|
  \leq C\delta^{-15/13}(1+|\xi|)^{-2/13}, \label{m_T_7}
 \egr

\item For any $\delta>0,\, r_0>0$, there is a constant  $y_0=(2r_0)^{5/3}\delta^{-2/3}>0$ such that
if $|y|\geq y_0$, we have
 \bgr
 \sup_{x\geq -\nu_0+\delta,|\xi|\leq r_0}\|K(x+\i y-c(\xi))^{-1}\|
 \leq C\delta^{-7/5}(1+|y|)^{-2/5},\label{m_T_8}
 \egr
 \item  For any $\delta>0$, we have \bgr
 \sup_{x\geq-\nu_0+\delta,y\in\R}
 \|(v\cdot\xi)(1+|\xi|^2)^{-1}P_{\rm d}(x+\i y-c(\xi))^{-1}\|\leq
C\delta^{-1}|\xi|(1+|\xi|^2)^{-1},\label{m_L_9}
\\
\sup_{x\geq -\nu_0+\delta,\xi\in\R^3}\|(v\cdot\xi)(1+|\xi|^2)^{-1}P_{\rm d}(x+\i y-c(\xi))^{-1}\|\leq
C(\delta^{-1}+1)|y|^{-1}.\label{m_L_10}
\egr
\end{enumerate}
\end{lem}
\begin{proof} The proof of \eqref{m_T_7} and \eqref{m_T_8} can be done as that of Lemma 2.2.6 in \cite{Ukai3}, we omit the details. The  \eqref{m_L_9} can be obtained by virtue of  $\|(v\cdot\xi)(1+|\xi|^2)^{-1}P_{\rm d}\|\leq C|\xi|(1+|\xi|^2)^{-1}$  and
$\|(x+\i y-c(\xi))^{-1}\|\leq \delta^{-1}$ for $x\geq -\nu_0+\delta$. The \eqref{m_L_10} follows from the facts
$\frac{(v\cdot\xi)}{1+|\xi|^2}P_{\rm d}(\lambda-c(\xi))^{-1}=\frac1\lambda\frac{(v\cdot\xi)}{1+|\xi|^2}P_{\rm d}+
\frac1\lambda\frac{(v\cdot\xi)}{1+|\xi|^2}P_{\rm d}c(\xi)(\lambda-c(\xi))^{-1} $
and
$
 \|\frac{(v\cdot\xi)}{1+|\xi|^2}P_{\rm d}c(\xi)\|\leq C.
$
 \end{proof}

\begin{lem}\label{m_spectrum} It holds.
\begin{description}
  \item[1.] For any $\delta>0$ and all $\xi\in\R^3$, there exists $y_1(\delta)>0$ such that
  \bq
\rho(\hat{B}_m(\xi))\supset\{\lambda\in\mathbb{C}\,|\,\mathrm{Re}\lambda\ge-\nu_0
+\delta,\,|\mathrm{Im}\lambda|\geq y_1\}\cup\{\lambda\in\mathbb{C}\,|\,\mathrm{Re}\lambda>0\}.\label{m_rb1}
\eq

  \item[2.]   For any $r_0>0$, there
exists $\alpha=\alpha(r_0)>0$ such that it holds  for $|\xi|\geq r_0$ that
\bq \sigma(\hat{B}_m(\xi))\subset\{\lambda\in\mathbb{C}\,|\,{\rm Re}\lambda\le -\alpha\} .\label{m_sg1}\eq

  \item[3.] For any $\delta>0$, there exists $r_1(\delta)>0$ such that if $0<|\xi|\leq r_1$, then
  \bq
\sigma(\hat{B}_m(\xi))\cap\{\lambda\in\mathbb{C}\,|\,\mathrm{Re}\lambda\ge-\mu/2\}\subset
\{\lambda\in\mathbb{C}\,|\,|\lambda|<\delta\}.\label{m_eigen1}
\eq
\end{description}
\end{lem}
\begin{proof} By \eqref{m_B_d} and Lemma \ref{m_LP}, we can prove \eqref{m_rb1}. The properties \eqref{m_sg1} and \eqref{m_eigen1} can be obtained by the similar argument as Proposition 2.3 in \cite{Ellis}.
\end{proof}

By applying the basic ideas similar to those used in Section 2.2 in \cite{Li2}, we can make a detailed analysis of the spectral of the operator $\hat{B}_m(\xi)$ at lower frequency below,  the details of the proof are omitted.

\begin{lem}\label{m_eigen_3z}
 There exists a constant  $r_0>0$ so that the spectrum $\lambda\in\sigma(B_m(\xi))\subset\mathbb{C}$ for $\xi=s\omega$ with $|s|\leq r_0$ and $\omega\in \mathbb{S}^2$ consists of five points $\{\lambda_j(s),\ j=-1,0,1,2,3\}$ on the domain $\mathrm{Re}\lambda>-\mu /2$. The spectrum $\lambda_j(s)$ and the corresponding eigenfunction $\psi_j(s,\omega)$ are $C^\infty$ functions of $s$ for $|s|\leq r_0$. In particular, the eigenvalues admit the following asymptotical expansion for $|s|\leq r_0$
\be                                   \label{m_specr0}
 \left\{\bln
 \lambda_{\pm1}(s)&=\pm \i 2\sqrt{\frac23} s-a_{\pm1}s^2+o(s^2),\quad
 \overline{\lambda_1(s)}=\lambda_{-1}(s),\\
 \lambda_{0}(s) &=-a_0s^2+o(s^2),\\
 \lambda_2(s) &=\lambda_3(s)=-a_2s^2+o(s^2),
 \eln\right.
 \ee
with $a_j>0$, $-1\le j\le 2$, are defined by
\bq
\left\{\bln
a_{\pm1}&=-\frac18(L^{-1}\P_1(v_1\chi_4),v_1\chi_4)-\frac12(L^{-1}\P_1(v_1\chi_1),v_1\chi_1),\\
a_0&=-\frac34(L^{-1}\P_1(v_1\chi_4),v_1\chi_4),\quad a_2=-(L^{-1}\P_1(v_1\chi_2),v_1\chi_2).
\eln\right.
\eq
The eigenfunctions are  orthogonal each other and satisfy \be
 \left\{\bln
 &\<\psi_j(s,\omega),\overline{\psi_k(s,\omega)}\>_\xi=\delta_{jk},
  \quad  j, k=-1,0,1,2,3,                                  \label{m_eigfr0}
 \\
&\psi_j(s,\omega)
 =\psi_{j,0}(\omega)+\psi_{j,1}(\omega)s+O(s^2), \quad |s|\leq r_0,
 \eln\right.
 \ee where the coefficients $\psi_{j,n}$  are given as
\bq
  \left\{\bln                      \label{m_eigf1}
 &\psi_{0,0}=\frac{\sqrt2}4\sqrt{M}-\frac{\sqrt3}2\chi_4,\quad
\P_1(\psi_{0,1})=\i L^{-1}\P_1(v\cdot\omega)\psi_{0,0};
 \\
&\psi_{\pm1,0}=\frac{\sqrt3}4\sqrt{M}\mp\frac{\sqrt2}2(v\cdot\omega)\sqrt{M}+\frac{\sqrt2}4\chi_4,\quad
\P_1(\psi_{\pm1,1})=\i L^{-1}\P_1(v\cdot\omega)\psi_{\pm1,0},
 \\
&\psi_{j,0}=(v\cdot W^j)\sqrt{M},\quad
\P_1(\psi_{j,1})=\i L^{-1}\P_1(v\cdot\omega)(v\cdot W^j)\sqrt{M},\quad j=2,3,
  \eln\right.
  \eq and $W^j$ $(j=2,3)$ are orthogonal vectors satisfying $W^j\cdot\omega=0$.
\end{lem}

\begin{proof}[\underline{Proof of Theorem~\ref{m_eigen_3}}] The combination of Lemmas~\ref{m_spectrum}--\ref{m_eigen_3z} leads to Theorem~\ref{m_eigen_3}.
\end{proof}

\begin{rem}
In general, the electric potential equation in \eqref{m_VPB4} takes as $a\Delta_x\Phi=\intr f\sqrt Mdv-e^{-b\Phi}$ with two constants $a,b>0$, and the asymptotical expansion of the eigenvalues $\lambda_j(s)$  becomes
$$
 \left\{\bln
 \lambda_{\pm1}(s)&=\pm \i\sqrt{\frac1b+\frac53} s-a_{\pm1}s^2+o(s^2),\quad \overline{\lambda_1(s)}=\lambda_{-1}(s),\\
 \lambda_{0}(s) &=-a_0s^2+o(s^2),\\
 \lambda_2(s) &=\lambda_3(s) =-a_2s^2+o(s^2), \\
 a_{\pm1}&=-\frac{b}{5b+3}(L^{-1}\P_1(v_1\chi_4),v_1\chi_4)-\frac12(L^{-1}\P_1(v_1\chi_1),v_1\chi_1),\\
a_0&=-\frac{3(b +1)}{5b+3}(L^{-1}\P_1(v_1\chi_4),v_1\chi_4),\quad a_2=-(L^{-1}\P_1(v_1\chi_2),v_1\chi_2).
 \eln\right.
$$
\end{rem}
With the help of Lemmas~\ref{m_specr1}--\ref{m_eigen_3z}, we can make a detailed analysis on the semigroup $S(t,\xi)=e^{t\hat{B}_m(\xi)}$ with respect to the lower frequency and higher frequency in terms of an argument similar to that of Theorem~\ref{rate1}, we omit the details.

\begin{thm}\label{m_sg}
The semigroup $S(t,\xi)=e^{t\hat{B}_m(\xi)}$ with $\xi=s\omega\in \R^3,s=|\xi|$ has the following
decomposition
 \be
 S(t,\xi)f=S_1(t,\xi)f+S_2(t,\xi)f,
     \quad f\in L^2_m(\R^3_v), \ \ t\ge0, \label{m_E_3a}
 \ee
 where
 \bq
 S_1(t,\xi)f=\dsum^3_{j=-1}e^{t\lambda_j(s)}
              \<f,\overline{\psi_j(s,\omega)}\,\>_\xi\psi_j(s,\omega)
               1_{\{|\xi|\leq r_0\}},          \label{m_E_5a}
 \eq
and  $S_2(t,\xi)f =: S(t,\xi)f-S_1(t,\xi)f$ satisfies
 \bq
 \|S_2(t,\xi)f\|_m\leq Ce^{-\sigma_0t}\|f\|_m,\quad t\ge0,\label{m_B_3a}
 \eq
with $\sigma_0>0$ a constant independent of $\xi$.
\end{thm}

\subsection{Optimal time-decay rates for linear mVPB}
\label{m_sect4}

With the help of the spectral analysis and semigroup estimates in section~\ref{sect3.m}, we can prove  Theorems~\ref{m_rate1} on the optimal time decay rate of global solution.

\begin{proof}[\underline{Proof of Theorem \ref{m_rate1}}]
Write
 $\xi^\alpha=\xi_1^{\alpha_1}\xi_2^{\alpha_2}\xi_3^{\alpha_3}$ with $\xi=(\xi_1,\xi_2,\xi_3)$.
By Theorem \ref{m_sg} and the Planchel's equality,  we have
\bma
\|(\dxa e^{tB_m}f_0,\chi_j)\|_{L^2_{x}}&\le \|\xi^\alpha(S_1(t,\xi)\hat{f}_0,\chi_j)\|_{L^2_\xi}+\|\xi^\alpha(S_2(t,\xi)\hat{f}_0,\chi_j)\|_{L^2_\xi}
\nnm\\
&\le \|\xi^\alpha(S_1(t,\xi)\hat{f}_0,\chi_j)\|_{L^2_\xi}+\|\xi^\alpha S_2(t,\xi)\hat{f}_0\|_{L^2_{\xi,v}},\\
\|\P_1(\dxa e^{tB_m}f_0)\|_{L^2_{x,v}}&\le \|\xi^\alpha\P_1(S_1(t,\xi)\hat{f}_0 )\|_{L^2_{\xi,v}}+\|\xi^\alpha S_2(t,\xi)\hat{f}_0\|_{L^2_{\xi,v}}.
\ema
By \eqref{m_B_3a} 
we obtain
\bq \intr (\xi^\alpha)^2\|S_2(t,\xi)\hat{f}_0\|_m^2 d\xi\le C\intr e^{-2\sigma_0 t}(\xi^\alpha)^2\|\hat{f}_0\|_m^2 d\xi \le Ce^{-2\sigma_0 t}\|\da_x f\|^2_{L^2_{x,v}}.\label{m_D_1}\eq
By \eqref{m_eigfr0} and \eqref{m_E_5a}, we have for $ |\xi|\le r_0$ that
$$
S_1(t,\xi)\hat{f}_0=\dsum^3_{j=-1}e^{t\lambda_j(|\xi|)}[\<\hat{f}_0,\overline{\psi_{j,0}}\>_{\xi=0}\psi_{j,0}+|\xi|T_j(\xi)\hat f_0],
$$
where 
$\lambda_j(|\xi|)$ is given by \eqref{m_specr0}, $\<f,g\>_{\xi=0}=(f,g)+(P_{\rm d}f,P_{\rm d}g)$ and $T_j(\xi)$ is the linear operator with  the norm $\|T_j(\xi)\|$
uniformly bounded for $|\xi|\leq r_0$ and $-1\leq j\leq3$.  Therefore
\bma
(S_1(t,\xi)\hat{f}_0,\sqrt M)
&=\frac{\sqrt3}4\sum_{j=\pm1}e^{\lambda_j(|\xi|)t}\Big[\frac{\sqrt3}2\hat{n}_0-j\frac{\sqrt2}2(\hat{m}_0\cdot
\omega)+\frac{\sqrt2}4\hat{q}_0\Big]
\nnm\\
&\quad+\frac{\sqrt2}4e^{\lambda_0(|\xi|)t}\Big(\frac{\sqrt{2}}2\hat{n}_0-\frac{\sqrt{3}}2\hat{q}_0\Big) +|\xi|\sum_{j=-1}^3e^{\lambda_j(|\xi|)t} (T_j(\xi)\hat{f}_0,\sqrt M),\label{m_S_4a}
\\
(S_1(t,\xi)\hat{f}_0,v\sqrt M)
&=-\frac{\sqrt2}2\sum_{j=\pm1}e^{\lambda_j(|\xi|)t}j\Big[\frac{\sqrt3}2\hat{n}_0-j\frac{\sqrt2}2(\hat{m}_0\cdot
\omega)+\frac{\sqrt2}4\hat{q}_0\Big]\omega\nnm\\
&\quad+\sum_{j=2,3}e^{\lambda_j(|\xi|)t}(\hat m_0\cdot W^j)W^j+|\xi|\sum_{j=-1}^3e^{\lambda_j(|\xi|)t} (T_j(\xi)\hat{f}_0,v\sqrt M),\label{m_S_4b}\\
(S_1(t,\xi)\hat{f}_0,\chi_4)
&=\frac{\sqrt2}4\sum_{j=\pm1}e^{\lambda_j(|\xi|)t}\Big[\frac{\sqrt3}2\hat{n}_0-j\frac{\sqrt2}2(\hat{m}_0\cdot
\omega)+\frac{\sqrt2}4\hat{q}_0\Big]\nnm\\
&\quad-\frac{\sqrt3}2e^{\lambda_0(|\xi|)t}\Big(\frac{\sqrt{2}}2\hat{n}_0 -\frac{\sqrt{3}}2\hat{q}_0\Big)
+|\xi|\sum_{j=-1}^3e^{\lambda_j(|\xi|)t} (T_j(\xi)\hat{f}_0,\chi_4),\label{m_S_4c}
\ema
where we recall that $\lambda_j(|\xi|)$ is given by \eqref{m_specr0}, $(\hat{n}_0,\hat{m}_0,\hat{q}_0)$ is the Fourier transform of
the macroscopic density, momentum and energy $(n_0,m_0,q_0)$ of the initial data $f_0$ defined by
$
  (n_0,m_0,q_0)=((f_0,\chi_0),(f_0,v\sqrt{M}),(f_0,\chi_4))
$
and $W^j$ is given by \eqref{m_eigf1}, and
\bma
 \P_1(S_1(t,\xi)\hat{f}_0)=|\xi|\sum_{j=-1}^3e^{\lambda_j(|\xi|)t} \P_1(T_j(\xi)\hat{f}_0).\label{m_S_5}
\ema
Noting by \eqref{m_specr0}  that
 \bq
\mathrm{Re}\lambda_j(|\xi|)=a_j|\xi|^2(1+O(|\xi|))\le -\beta |\xi|^2,\quad |\xi|\leq
r_0,\label{m_ee}\eq  where and below $\beta>0$ denotes a generic constant, we obtain by \eqref{m_S_4a}--\eqref{m_S_5} that
\bma
&|(S_1(t,\xi)\hat{f}_0,\chi_j)|^2
\le Ce^{-2\beta|\xi|^2t}(| (\hat{n}_0,\hat{m}_0,\hat{q}_0)|^2+|\xi|^2\|\hat{f}_0\|^2_{L^2_v}),\quad j=0,1,2,3,4,\label{m_S_2}\\
&\|\P_1(S_1(t,\xi)\hat{f}_0)\|^2_{L^2_v}\le C|\xi|^2e^{-2\beta
|\xi|^2t}\|\hat{f}_0\|^2_{L^2_v}.\label{m_S_3}
\ema
Thus by \eqref{m_S_2}, H\"{o}lder and Hausdorff-Young
inequalities, we have
$$ \|\xi^\alpha (S_1(t,\xi)\hat{f}_0,\chi_j)\|^2_{L^2_{\xi}}\le
C(1+t)^{-3\(\frac1q-\frac12\)-k}\|\dx^{\alpha'}f_0\|^2_{L^{2,q}},
$$
with $k=|\alpha-\alpha'|$.  This proves \eqref{m_V_4}.
Similarly, by \eqref{m_S_2} and \eqref{m_S_3} we can prove \eqref{m_V_4a} and \eqref{m_V_5}.

%
Now we turn to show the lower bound of time-decay rates for the global solution under the assumptions of Theorem \ref{m_rate1}.  Note that
\bma
 \|\Tdx^k(e^{tB_m}f_0,\chi_j)\|_{L^2_{x}}&\ge \||\xi|^k(S_1(t,\xi)\hat{f}_0,\chi_j)\|_{L^2_{\xi}}-\||\xi|^kS_2(t,\xi)\hat{f}_0\|_{L^2_{\xi,v}}
\nnm\\
&\ge \||\xi|^k(S_1(t,\xi)\hat{f}_0,\chi_j)\|_{L^2_{\xi}}-Ce^{-\sigma_0 t}\|\Tdx^k f_0\|_{L^2_{x,v}},\quad j=0,1,2,3,4,\label{m_H_2a}
\\
\|\Tdx^k\P_1 (e^{tB_m}f_0)\|^2_{L^2_{x,v}}
&\ge
 \||\xi|^k\P_1(S_1(t,\xi)\hat{f}_0)\|_{L^2_{\xi,v}}^2-\||\xi|^kS_2(t,\xi)\hat{f}_0\|_{L^2_{\xi,v}}\nnm\\
 &\ge \||\xi|^k\P_1(S_1(t,\xi)\hat{f}_0)\|_{L^2_{\xi,v}}^2-Ce^{-\sigma_0 t}\|\Tdx^k f_0\|^2_{L^2_{x,v}},\label{m_H_2c}
 \ema
where we have used \eqref{m_D_1} for $|\alpha|=k$.

First, we prove \eqref{m_H_1}--\eqref{m_H_1a} as follows. By \eqref{m_S_4a} and \eqref{m_ee}, we have
\bma
 |(S_1(t,\xi)\hat f_0,\chi_0)|^2
\ge &\,
  \frac1{16}|\hat q_0|^2\Big|
     e^{{\rm Re}\lambda_1(|\xi|)t}\cos({\rm Im}\lambda_1(|\xi|)t)
    -e^{\lambda_0(|\xi|)t}\Big|^2
  -e^{-2\beta |\xi|^2t}(2|\hat n_0|^2+C|\xi|^2\|\hat f_0\|^2_{L^2_v})
\nnm\\
\ge &\,
  \frac1{20}|\hat q_0|^2\Big|
     e^{{\rm Re}\lambda_1(|\xi|)t}\cos\bigg(2\sqrt{\frac23}\,|\xi|t\bigg)
    -e^{\lambda_0(|\xi|)t}\Big|^2
\nnm\\
 &  -e^{-2\beta |\xi|^2t}
    (C(|\xi|^3t)^2|\hat q_0|^2+2|\hat n_0|^2
      +C|\xi|^2\|\hat f_0\|^2_{L^2_v}),           \label{m_H_6a}
\ema
due to the fact
$ \cos(\mathrm{Im}\lambda_1(|\xi|)t)
  \sim\cos\(2\sqrt{\frac23}|\xi|t\)+O(|\xi|^3t),
$
which implies
\bma
\||\xi|^{k}(S_1(t,\xi)\hat{f}_0,\chi_0)\|^2_{L^2_{\xi}}
\ge &\,
 \frac1{20}\inf_{|\xi|\le r_0}|\hat q_0|^2\int_{|\xi|\le r_0}|\xi|^{2k}\bigg|
    e^{{\rm Re}\lambda_1(|\xi|)t}\cos\bigg(2\sqrt{\frac23}|\xi|t\bigg)
   -e^{\lambda_0(|\xi|)t}\bigg|^2d\xi
\nnm\\
& -C_3\sup_{|\xi|\le r_0}|\hat n_0|^2(1+t)^{-3/2-k}- C(\|q_0\|^2_{L^1_x} + \|f_0\|^2_{L^{2,1}})(1+t)^{-5/2-k}
\nnm\\
=:&\,\frac1{20}\inf_{|\xi|\le r_0}|\hat q_0|^2I_1
   -C_3\sup_{|\xi|\le r_0}|\hat n_0|^2(1+t)^{-3/2-k} - C(1+t)^{-5/2-k}.\label{m_F_4}
\ema
For the term $I_1$, it holds for $t\ge t_0=:\frac{1}{r^2_0}$ that
\bma
I_1
&\ge
4\pi\sum_{k=0}^{[\frac1\pi\sqrt{\frac23}r_0t]-1}
  \int^{(2k\pi+\frac32\pi)/(2\sqrt{\frac23}t)}_{(2k\pi+\frac12\pi)/(2\sqrt{\frac23}t)} r^{2+2k}\bigg| e^{{\rm Re}\lambda_1(r)t}\cos\bigg(2\sqrt{\frac23}rt\bigg)
        -e^{\lambda_0(r)t}\bigg|^2dr
\nnm\\
&\ge
4\pi\sum_{k=0}^{[\frac1\pi\sqrt{\frac23}r_0t]-1}
  \int^{(k\pi+\frac34\pi)/(\sqrt{\frac23}t)}_{(k\pi+\frac14\pi)/(\sqrt{\frac23}t)}r^{2+2k}e^{2\lambda_0(r)t}dr
   \ge
 \frac12\pi e^{-2\eta}(1+t)^{-3/2-k}, \label{m_F_2}
\ema
where we have used by \eqref{m_specr0}
\be
\mathrm{Re}\lambda_j(|\xi|)=a_j|\xi|^2(1+O(|\xi|))\ge -\eta |\xi|^2,
\quad |\xi|\leq r_0. \label{m_eez}
\ee
It follows from \eqref{m_F_4} and \eqref{m_F_2} that
\bma
 \||\xi|^{k}(S_1(t,\xi)\hat{f}_0,\chi_0)\|^2_{L^2_{\xi}}\ge&
  C_4\inf_{|\xi|\le r_0}|q_0|^2(1+t)^{-3/2-k}
 -C_3\sup_{|\xi|\le r_0}|\hat n_0|^2(1+t)^{-3/2-k}\nnm\\
 &-C(1+t)^{-5/2-k},\label{m_H_7}\ema
which and \eqref{m_H_2a} with $j=0$ lead to \eqref{m_H_1} and \eqref{m_H_2} for $d_1>\frac{C_3}{C_4}>0$ and $t>0$ large enough.

By \eqref{m_S_4b}, we obtain
\bma
|(S_1(t,\xi)\hat f_0,v\sqrt M)|^2
&\ge \frac18e^{2{\rm Re}\lambda_1(|\xi|)t}\sin^2({\rm Im}\lambda_1(|\xi|)t)\(|\hat q_0|-\sqrt6 |\hat n_0|\)^2-C|\xi|^2e^{-2\beta|\xi|^2t}\|\hat f_0\|^2_{L^2_v}.\label{m_B_2a}
\ema
In terms of \eqref{m_eez} and the fact
$$
 \sin^2(\mathrm{Im}\lambda_1(|\xi|)t)\geq
\frac12\sin^2\bigg(2\sqrt{\frac23}|\xi|t\bigg)-O([|\xi|^3t]^2),
$$
we obtain by \eqref{m_B_2a} that
\bma
 \||\xi|^{k}(S_1(t,\xi)\hat f_0,v\sqrt M)\|^2_{L^2_{\xi}}
\ge
 &\frac1{16}(d_1- \sqrt{6})^2d_0^2\int_{|\xi|\leq r_0}|\xi|^{2k}
  e^{-2\eta|\xi|^2t}\sin^2\bigg(2\sqrt{\frac23}|\xi|t\bigg)d\xi
\nnm\\
 &-C(\|n_0\|^2_{L^1_x} +\|q_0\|^2_{L^1_x} + \|f_0\|^2_{L^{2,1}})(1+t)^{-5/2-k}\nnm\\
=:&\frac1{16}(d_1- \sqrt{6})^2d_0^2I_2-C(1+t)^{-5/2-k}.\label{m_H_3}
\ema
Since it holds for $t\ge t_0=:\frac{L^2}{r_0^2}$ with the constant $L>4\pi$ that
\bma
I_2
&=  t^{-3/2-k}\int_{|\zeta|\le r_0\sqrt t}|\zeta|^{2k}e^{-2\eta|\zeta|^2}\sin^2\bigg(2\sqrt{\frac23}|\zeta|\sqrt{t}\bigg)d\zeta\nnm\\
&\ge\pi (1+t)^{-3/2-k} L^{2+2k} e^{-2\eta L^2}\int^L_{L/2}\sin^2\bigg(2\sqrt{\frac23}r\sqrt{t}\bigg)dr\nnm\\
&\ge\pi (1+t)^{-3/2-k} L^{2+2k} e^{-2\eta L^2}\int^\pi_0\sin^2ydy>C_3(1+t)^{-3/2-k}, \label{m_H_4}
\ema
we obtain \eqref{m_H_1} for $t>0$ large enough, with the help of \eqref{m_H_3} and \eqref{m_H_2a} for $j=1,2,3$.

By \eqref{m_S_4c}, we obtain
\bmas
|(S_1(t,\xi)\hat f_0,\chi_4)|^2
&\ge \frac1{32}|\hat q_0|^2e^{2{\rm Re}\lambda_1(|\xi|)t}\cos^2({\rm Im}\lambda_1(|\xi|)t)\Big|e^{{\rm Re}\lambda_1(|\xi|)t}\cos({\rm Im}\lambda_1(|\xi|)t)+3e^{\lambda_0(|\xi|)t}\Big|^2
\\
&\quad-2e^{-2\beta|\xi|^2t}|\hat n_0|^2-C|\xi|^2e^{-2\beta|\xi|^2t}\|\hat f_0\|^2_{L^2_v},
\emas
with which, we can obtain by the similar argument to estimate the terms in \eqref{m_F_2} that
\bmas
 \||\xi|^{k}(S_1(t,\xi)\hat{f}_0,\chi_4)\|^2_{L^2_{\xi}}\ge&
  C_4\inf_{|\xi|\le r_0}|q_0|^2(1+t)^{-3/2-k}
 -C_3\sup_{|\xi|\le r_0}|\hat n_0|^2(1+t)^{-3/2-k}\nnm\\
 &-C(1+t)^{-5/2-k},\emas
which together with \eqref{m_H_2a} for $j=4$ lead to \eqref{m_H_1} for $t>0$ large enough.
\par

Next, we prove \eqref{m_H_1a}. By \eqref{m_S_5} and  \eqref{m_eigf1}, we have
 \bmas
\P_1(S_1(t,\xi)\hat{f}_0)
 &=\i |\xi|\sum_{j=\pm1}e^{\lambda_j(|\xi|)t}\Big(\frac{\sqrt3}2\hat{n}_0+\frac{\sqrt2}4\hat{q}_0\Big)\Big[\frac{\sqrt2}4 L^{-1}\P_1(v\cdot\omega)\chi_4-j\frac{\sqrt2}2L^{-1}\P_1(v\cdot\omega)^2\sqrt M\Big]
\\
 &\quad +\i |\xi|\sum_{j=2,3}e^{\lambda_j(|\xi|)t}(\hat m_0\cdot W^j)L^{-1}\P_1(v\cdot\omega)(v\cdot W^j)\sqrt M
\\
&\quad-\i\frac{\sqrt3}2 |\xi|e^{\lambda_0(|\xi|)t}\Big(\frac{\sqrt{2}}2\hat{n}_0-\frac{\sqrt{3}}2\hat{q}_0\Big)L^{-1}\P_1(v\cdot\omega)\chi_4+|\xi|^2T_4(t,\xi)\hat f_0,
\emas
where $T_4(t,\xi)\hat f_0$ is the remainder term satisfying $\|T_4(t,\xi)\hat f_0\|^2_{L^2_v}\le Ce^{-2\beta |\xi|^2t}\|\hat f_0\|^2_{L^2_v}$.
Noting that the terms $L^{-1}\P_1(v\cdot\omega)\chi_4$, $L^{-1}\P_1(v\cdot\omega)^2\sqrt M$, and $L^{-1}\P_1(v\cdot\omega)(v\cdot W^j)\sqrt M$ are orthogonal to each other, we have
\bmas
  \|\P_1(S_1(t,\xi)\hat{f}_0)\|^2_{L^2_v}
\ge\,
 &\frac18\|L^{-1}\P_1(v_1\chi_1)\|^2_{L^2_v}|\xi|^2
  e^{2{\rm Re}\lambda_1(|\xi|)t}\sin^2({\rm Im}\lambda_1(|\xi|)t)
  \Big(|\hat{q}_0|-\sqrt6|\hat{n}_0|\Big)^2\\
 &-C|\xi|^4e^{-2\beta |\xi|^2t}\|\hat f_0\|^2_{L^2_v},
\emas
and then 
\bmas
 \||\xi|^{k}\P_1(S_1(t,\xi)\hat{f}_0)\|^2_{L^2_{\xi,v}}
\ge
 &\, C_4(1+t)^{-5/2-k}-C(1+t)^{-7/2-k}.
\emas
This together with \eqref{m_H_2c} leads to \eqref{m_H_2} for $t>0$ large enough.
 This completes the proof of the theorem.
\end{proof}

\section{The nonlinear problem for bVPB system}\setcounter{equation}{0}
\label{sect4}
\subsection{Energy estimates}

Let $N$ be a positive integer, and
\bmas
E_k(f_1,f_2)&=\sum_{|\alpha|+|\beta|\le N}\|w^k\dxa\dvb (f_1,f_2)\|^2_{L^2_{x,v}}+\sum_{|\alpha|\le N}\|w^k\dxa\Tdx \Phi\|^2_{L^2_x},\\
H_k(f_1,f_2)&= \sum_{|\alpha|+|\beta|\le N}\|w^k\dxa\dvb
(\P_1f_1,P_rf_2)\|^2_{L^2_{x,v}}\nnm\\
&\quad +\sum_{|\alpha|\le N-1}(\|\dxa\Tdx  (\P_0f_1,P_{\rm d}f_2)\|^2_{L^2_{x,v}}+\|\dxa  P_{\rm d}f_2\|^2_{L^2_{x,v}}+\|\dxa\Tdx \Phi\|^2_{L^2_x}),\\
D_k(f_1,f_2)&=\sum_{|\alpha|+|\beta|\le N}\|w^{\frac12+k}\dxa\dvb  (\P_1f_1,P_rf_2)\|^2_{L^2_{x,v}}\nnm\\
&\quad+\sum_{|\alpha|\le N-1}(\|\dxa\Tdx  (\P_0f_1,P_{\rm d}f_2)\|^2_{L^2_{x,v}}+\|\dxa  P_{\rm d}f_2\|^2_{L^2_{x,v}}+\|\dxa\Tdx \Phi\|^2_{L^2_x}).
\emas
for $k\ge 0$. For brevity, we write $E(f_1,f_2)=E_0(f_1,f_2)$, $H(f_1,f_2)=H_0(f_1,f_2)$ and $D(f_1,f_2)=D_0(f_1,f_2)$ for $k=0$.

Firstly, by taking inner product between $\chi_j\ (j=0,1,2,3,4)$ and \eqref{VPB7}, we obtain  the system of  compressible Euler-Poisson type (EP) as
\bma
 \dt n_1+\divx  m_1&=0,\label{G_3}
 \\
\dt  m_1+\Tdx n_1+\sqrt{\frac23}\Tdx q_1&=n_2\Tdx \Phi-\intr v\cdot\Tdx(  \P_1f_1) v\sqrt Mdv,\label{G_5}
\\
\dt q_1+\sqrt{\frac23}\divx m_1&=\sqrt{\frac23}\Tdx \Phi\cdot m_2-\intr v\cdot\Tdx(  \P_1f_1) \chi_4 dv, \label{G_6}
\ema
where
$$(n_1,m_1,q_1)=((f_1,\sqrt M),(f_1,v\sqrt M),(f_1,\chi_4)),\quad (n_2,m_2)=((f_2,\sqrt M),(f_2,v\sqrt M)).$$
Taking the microscopic projection $  \P_1$ to \eqref{VPB7}, we have
\bma
\dt(  \P_1f_1)+  \P_1(v\cdot\Tdx   \P_1f_1)-L(  \P_1f_1)&=-  \P_1(v\cdot\Tdx  \P_0f_1)+  \P_1 G_1 ,\label{GG1}
\ema
where the nonlinear term $G_1$ is denoted by
\bq
G_1=\frac12 (v\cdot\Tdx\Phi)f_2-\Tdx\Phi\cdot\Tdv f_2+\Gamma(f_1,f_1).\label{G1}
\eq
By \eqref{GG1}, we can express the microscopic part $  \P_1f_1$ as
\bq
  \P_1f_1=L^{-1}[\dt(  \P_1f_1)+  \P_1(v\cdot\Tdx   \P_1f_1)-  \P_1 G_1]+L^{-1}  \P_1(v\cdot\Tdx  \P_0f_1). \label{p_1}
  \eq
Substituting \eqref{p_1} into \eqref{G_3}--\eqref{G_6}, we obtain the  system  of the compressible Navier-Stokes-Poisson type (NSP) as
\bma
\dt n_1+\divx  m_1&=0,\label{G_9}
\\
\dt  m_1+\dt R_1+\Tdx n_1+\sqrt{\frac23}\Tdx q_1&=\kappa_1 (\Delta_x m_1+\frac13\Tdx{\rm div}_x m_1)+n_2\Tdx \Phi+R_2,\label{G_7}
\\
\dt q_1+\dt R_3+\sqrt{\frac23}\divx m_1&=\kappa_2 \Delta_x q_1+\sqrt{\frac23}\Tdx \Phi\cdot m_2+R_4,\label{G_8}
\ema
where the viscosity coefficients $\kappa_1,\kappa_2>0$ and the remainder terms $R_1, R_2, R_3, R_4$ are defined by
\bmas
\kappa_1&=-(L^{-1}  \P_1(v_1\chi_2),v_1\chi_2),\quad \kappa_2=-(L^{-1}  \P_1(v_1\chi_4),v_1\chi_4),\\
R_1&=( v\cdot\Tdx L^{-1} \P_1f_1,v\sqrt M),\quad R_2=-(v\cdot\Tdx L^{-1}(  \P_1(v\cdot\Tdx \P_1f_1)-  \P_1 G_1),v\sqrt M),\\
R_3&=( v\cdot\Tdx L^{-1} \P_1f_1,\chi_4),\quad R_4=-(v\cdot\Tdx L^{-1}(  \P_1(v\cdot\Tdx \P_1f_1)-  \P_1 G_1),\chi_4).
\emas

By taking inner product between $\sqrt M$ and \eqref{VPB8}, we obtain
\bma
\dt n_2+\divx m_2&=0,\label{G_3a}
\ema
Taking the microscopic projection $ P_r$ to \eqref{VPB8}, we have
\bma
\dt( P_rf_2)+ P_r(v\cdot\Tdx  P_rf_2)-v\sqrt M\cdot\Tdx\Phi-L_1( P_rf_2)&=- P_r(v\cdot\Tdx  P_{\rm d}f_2)+ P_r G_2,\label{GG2}
\ema
where the nonlinear term $G_2$ is denoted by
\bq G_2=\frac12 (v\cdot\Tdx\Phi)f_1-\Tdx\Phi\cdot\Tdv f_1+\Gamma(f_2,f_1).\label{G2}\eq
By \eqref{GG2}, we can express the microscopic part $ P_rf_2$ as
\bq   P_rf_2=L_1^{-1}[\dt( P_rf_2)+ P_r(v\cdot\Tdx  P_rf_2)- P_r G_2]+L_1^{-1} P_r(v\cdot\Tdx  P_{\rm d}f_2)-L_1^{-1} (v\sqrt M\cdot\Tdx\Phi). \label{p_c}\eq
Substituting \eqref{p_c} into \eqref{G_3a}, we obtain
\bma
\dt n_2+ \dt\divx R_5 +\kappa_3 n_2-\kappa_3 \Delta_x n_2=-\divx((P_r(v\cdot\Tdx P_rf_2)-P_rG_2),v\sqrt M),\label{G_9a}
\ema
where the viscosity coefficient $\kappa_3>0$  and the term $R_5$ are  defined by
\bq
\kappa_3=-(L_1^{-1}\chi_1,\chi_1),\quad  R_5=(L_1^{-1}P_rf_2,v\sqrt M).
\eq

\begin{lem}[\cite{Duan2,Guo2}]\label{e1}  It holds that
 \bma
\|\nu^{k}\dvb\Gamma(f,g)\|_{L^2_v}&\le C\sum_{\beta_1+\beta_2\le\beta}
(\|\dv^{\beta_1}f\|_{L^2_v}\|\nu^{k+1}\dv^{\beta_2}g\|_{L^2_v}+\|\nu^{k+1}\dv^{\beta_1}f\|_{L^2_v}\|\dv^{\beta_2}g\|_{L^2_v}),\label{a}
\ema for $k\ge -1$, and
\bq \|\Gamma(f,g)\|_{L^{2,1}}\le C(\|f\|_{L^2_{x,v}}\|\nu g\|_{L^2_{x,v}}+\|\nu
f\|_{L^2_{x,v}}\|g\|_{L^2_{x,v}}).\label{b}\eq
\end{lem}

\begin{lem}[Macroscopic dissipation] \label{macro-en} Let $(n_1,m_1,q_1)$ and $n_2$ be the strong solutions to \eqref{G_9}--\eqref{G_8} and \eqref{G_9a} respectively. Then, there is a constant $p_0>0$ so that
\bma
&\Dt \sum_{k\le |\alpha|\le N-1}p_0\bigg(\|\dxa(n_1, m_1,q_1)\|^2_{L^2_x}+2\intr \dxa R_1\dxa m_1dx+2\intr \dxa R_3\dxa q_1dx\bigg)\nnm\\
&+\Dt \sum_{k\le |\alpha|\le N-1}4\intr \dxa m_1 \dxa\Tdx n_1dx+\sum_{k\le |\alpha|\le N-1} \|\dxa\Tdx (n_1, m_1,q_1)\|^2_{L^2_x}
\nnm\\
\le & C\sqrt{E(f_1,f_2)}D(f_1,f_2)+C\sum_{k\le |\alpha|\le N-1}\|\dxa\Tdx \P_1f_1\|^2_{L^2_{x,v}},\label{E_1}
\\
&\Dt \sum_{k\le |\alpha|\le N-1}(\|\dxa n_2\|^2_{L^2_x} +\|\dxa \Tdx\Phi\|^2_{L^2_x}+2 \intr\dxa\divx R_5\dxa n_2 dx+2 \intr\dxa R_5\dxa \Tdx\Phi dx)\nnm\\
&+\kappa_3\sum_{k\le |\alpha|\le N-1}(\|\dxa n_2\|^2_{L^2_x}+\|\dxa \Tdx n_2\|^2_{L^2_x} +\|\dxa \Tdx\Phi\|^2_{L^2_x})
\nnm\\
\le & CE(f_1,f_2)D(f_1,f_2)+C\sum_{k\le |\alpha|\le N-1}(\|\dxa  P_rf_2\|^2_{L^2_{x,v}}+\|\dxa\Tdx P_rf_2\|^2_{L^2_{x,v}}),\label{E_1a}
\ema
with $0\le k\le N-1$.
\end{lem}

\begin{proof}
First of all, we prove \eqref{E_1}. Taking the inner product between $\dxa m_1$ and $\dxa\eqref{G_7}$ with $|\alpha|\le N-1$, we have
\bma
&\frac12\Dt \|\dxa m_1\|^2_{L^2_x}
 + \intr \dxa\dt R_1 \dxa m_1dx
 +\intr \dxa\Tdx n_1 \dxa m_1dx+\sqrt{\frac23}\intr \dxa\Tdx q_1\dxa m_1dx\nnm\\
&\quad+\kappa_1 (\|\dxa\Tdx  m_1\|^2_{L^2_x}+\frac13\|\dxa\divx  m_1\|^2_{L^2_x})\nnm\\
&=\intr \dxa(n_2\Tdx \Phi) \dxa m_1dx+\intr\dxa R_1\dxa m_1dx.  \label{en_1}
\ema
For the second and third terms in the left hand side of \eqref{en_1}, we have
\bma
 \intr \dxa\dt R_1 \dxa m_1dx&=\Dt\intr\dxa R_1\dxa m_1dx\nnm\\
&\quad-\intr \dxa R_1\dxa[-\Tdx n_1-\sqrt{\frac23}\Tdx q_1+n_2\Tdx\Phi-(v\cdot\Tdx \P_1f_1,v\sqrt M)]dx\nnm\\
&\ge \Dt\intr\dxa R_1\dxa m_1dx-\epsilon(\|\dxa\Tdx n_1\|^2_{L^2_x}+\|\dxa\Tdx q_1\|^2_{L^2_x})\nnm\\
&\quad-C\sqrt{E(f_1,f_2)}D(f_1,f_2)-\frac{C}{\epsilon}\|\dxa\Tdx \P_1f_1\|^2_{L^2_x},  \label{en_2}
\ema
and
\bma
\intr \dxa\Tdx n_1 \dxa m_1dx&=-\intr\dxa n\dxa\divx mdx=\intr\dxa n_1\dxa\dt n_1dx=\frac12\Dt \|\dxa n_1 \|^2_{L^2_x}.\label{en_2a}
\ema
The first  term in the right hand side  of \eqref{en_1} are bounded by $C\sqrt{E(f_1,f_2)}D(f_1,f_2)$. The second term can be estimated by
\bma
\intr\dxa R_2\dxa m_1dx
&\le C \|\dxa\Tdx   \P_1f_1\|_{L^2_{x,v}}
\|\dxa\Tdx m_1\|_{L^2_x}\nnm\\
&\quad+C(\|\dxa(\Tdx\Phi f_2)\|_{L^2_{x,v}}+\|w^{-\frac12}\dxa \Gamma(f_1,f_1)\|_{L^2_{x,v}})\|\dxa\Tdx m_1\|_{L^2_x}\nnm\\
&\le \frac{\kappa_1}2\|\dxa\Tdx m_1\|_{L^2_x}^2+C\|\dxa\Tdx  \P_1f_1\|^2_{L^2_{x,v}}+C\sqrt{E(f_1,f_2)}D(f_1,f_2),\label{en_4}
\ema
where we make use of Lemma \ref{e1} to obtain
\bma
&\|w^{-\frac12}\dxa \Gamma(f_1,f_1)\|^2_{L^2_{x,v}}+\|\dxa( \Tdx\Phi f_2)\|^2_{L^2_{x,v}}\nnm\\
&\le C\|  f_1\|^2_{L^2_v(H^{N}_x)}\|w^{\frac12}\Tdx f_1\|^2_{L^2_v(H^{N-1}_x)}+ C\|\Tdx\Phi\|^2_{H^N_x}\|\Tdx f_2\|^2_{L^2_v(H^{N-1}_x)}\nnm\\
&\le C E(f_1,f_2) D(f_1,f_2).\label{gamma}
\ema
Therefore, it follows from \eqref{en_1}--\eqref{en_4} that
\bma
&\frac12\Dt (\|\dxa m_1\|^2_{L^2_x}+\|\dxa n_1\|^2_{L^2_x})+\Dt\intr\dxa R_1\dxa m_1dx+\sqrt{\frac23}\intr \dxa\Tdx q_1\dxa m_1dx\nnm\\
&\quad+\frac{\kappa_1}2 (\|\dxa\Tdx  m_1\|^2_{L^2_x}+\frac13\|\dxa\divx  m_1\|^2_{L^2_x})\nnm\\
&\le C\sqrt{E(f_1,f_2)}D(f_1,f_2)+\frac{C}{\epsilon} \|\dxa\Tdx   \P_1f_1\|^2_{L^2_{x,v}}+\epsilon(\|\dxa\Tdx n_1\|^2_{L^2_x}+\|\dxa\Tdx q_1\|^2_{L^2_x}),\label{m_1}
\ema

Similarly, taking the inner product between $\dxa q_1$ and $\dxa\eqref{G_8}$ with $|\alpha|\le N-1$, we have
\bma
&\frac12\Dt \|\dxa q_1\|^2_{L^2_x}+\Dt\intr \dxa R_3\dxa q_1dx+\sqrt{\frac23}\intr \dxa\divx m_1\dxa q_1dx+\frac{\kappa_2}2 \|\dxa\Tdx q_1\|^2_{L^2_x}
\nnm\\
&\le C\sqrt{E(f_1,f_2)}D(f_1,f_2)+\frac{C}{\epsilon}\|\dxa\Tdx   \P_1f_1\|^2_{L^2_{x,v}}+\epsilon\|\dxa\Tdx m_1\|^2_{L^2_x}.\label{q_1}
\ema

Again, taking the inner product between $\dxa\Tdx n_1$ and $\dxa\eqref{G_5}$ with $|\alpha|\le N-1$ to get
\bma &\Dt\intr \dxa m_1 \dxa\Tdx n_1dx+\frac12\|\dxa\Tdx n_1\|^2_{L^2_x}\nnm\\
&\le C \sqrt{E(f_1,f_2)}D(f_1,f_2)+\|\dxa\divx  m_1\|^2_{L^2_x}+\|\dxa \Tdx q_1\|^2_{L^2_x}+C\|\dxa\Tdx  \P_1f_1\|^2_{L^2_{x,v}}.\label{abc}
\ema

Taking the summation of $p_0\sum\limits_{k\le |\alpha|\le N-1}[\eqref{m_1}+\eqref{q_1}]+4\sum\limits_{k\le |\alpha|\le N-1}\eqref{abc}$ with $p_0>0$ large enough, $\epsilon>0$ small enough and $0\le k\le N-1$, we obtain \eqref{E_1}.

Next, we turn to show \eqref{E_1a}. Taking the inner product between $\dxa n_2$ and $\dxa\eqref{G_9a}$ with $|\alpha|\le N-1$, we have
\bma
&\Dt \|\dxa n_2\|^2_{L^2_x}+2\Dt \intr\dxa\divx R_5\dxa n_2 dx+\kappa_3\|\dxa n_2\|^2_{L^2_x}+\kappa_3\|\dxa \Tdx n_2\|^2_{L^2_x}\nnm\\
&\le C\|\dxa\Tdx P_rf_2\|^2_{L^2_{x,v}}+C E(f_1,f_2)D(f_1,f_2).\label{en_5}
\ema
Similarly, taking the inner product between $-\dxa \Phi$ and $\dxa\eqref{G_9a}$ with $|\alpha|\le N-1$ we obtain
\bma
&\Dt \|\dxa \Tdx\Phi\|^2_{L^2_x}+2\Dt \intr\dxa R_5\dxa \Tdx\Phi dx+\kappa_3\|\dxa \Tdx\Phi\|^2_{L^2_x}+\kappa_3\|\dxa  n_2\|^2_{L^2_x}\nnm\\
&\le C(\|\dxa P_rf_2\|^2_{L^2_{x,v}}+\|\dxa\Tdx P_rf_2\|^2_{L^2_{x,v}})+CE(f_1,f_2)D(f_1,f_2).\label{en_6}
\ema
Taking the summation of $\sum_{k\le |\alpha|\le N-1}[\eqref{en_5}+\eqref{en_6}]$ with $0\le k\le N-1$, we obtain \eqref{E_1a}.
\end{proof}

In the following, we shall estimate the  microscopic terms to enclose the energy inequality of the solution $f$ to bVPB system \eqref{VPB7}--\eqref{VPB9}.

\begin{lem}[Microscopic dissipation]
\label{micro-en}
Let $(f_1,f_2,\Phi)$ be a strong solution to bVPB system  \eqref{VPB7}--\eqref{VPB9}.
Then, there are constants $p_k>0$, $1\le k\le N$ so that
 \bma
&\frac12\Dt  (\|(f_1,f_2)\|^2_{L^2_{x,v}}+\|\Tdx\Phi\|^2_{L^2_x})+\mu \|w^{\frac12} (\P_1f_1,P_rf_2)\|^2_{L^2_{x,v}}
\le C\sqrt{E(f_1,f_2)}D(f_1,f_2),\label{E_3}\\
 &\frac12\Dt \sum_{1\le|\alpha|\le N}(\|\dxa (f_1,f_2)\|^2_{L^2_{x,v}}+\|\dxa\Tdx\Phi\|^2_{L^2_x})+\mu\sum_{1\le|\alpha|\le N} \|w^{\frac12}\dxa (\P_1f_1,P_rf_2)\|^2_{L^2_{x,v}}\nnm\\
\le &C\sqrt{E(f_1,f_2)}D(f_1,f_2),\label{E_2}\\
&\Dt \sum_{1\le k\le N}p_k\sum_{|\beta|=k \atop |\alpha|+|\beta|\le N}\|\dxa\dvb (\P_1f_1,P_rf_2)\|^2_{L^2_{x,v}}
+\mu\sum_{1\le k\le N}p_k\sum_{|\beta|=k \atop |\alpha|+|\beta|\le N}\|w^{\frac12}\dxa\dvb (\P_1f_1,P_rf_2)\|^2_{L^2_{x,v}}\nnm\\
\le &C\sum_{|\alpha|\le N-1}(\|\dxa\Tdx ( f_1, f_2)\|^2_{L^2_{x,v}}+\|\dxa\Tdx\Phi\|^2_{L^2_x}) +C\sqrt{E(f_1,f_2)}D(f_1,f_2). \label{E_5}
\ema
\end{lem}

\begin{proof}
Taking inner product between $\dxa  f_1$ and $\dxa\eqref{VPB7}$ with $|\alpha|\le N$ $(N\ge 4)$, we have
\bma
&\frac12\Dt \|\dxa f_1\|^2_{L^2_{x,v}}-\intr (L\dxa f_1)\dxa f_1dxdv
\nnm\\
&=\frac12\intrr \dxa (v\cdot\Tdx \Phi f_2)\dxa f_1dxdv-\intrr \dxa (\Tdx \Phi\cdot \Tdv f_2)\dxa f_1dxdv+\intrr\dxa \Gamma(f_1,f_1)\dxa f_1dxdv\nnm\\
&=:I_1+I_2+I_3.\label{G_0}
\ema
 For $I_1$, it holds
\bma
I_1
&\le C\sum_{1\le |\alpha'|\le |\alpha|-1}\intr |v|\|\dx^{\alpha'}\Tdx\Phi\|_{L^3_x} \|\dx^{\alpha-\alpha'}f_2\|_{L^{6}_x}\|\dxa f_1\|_{L^2_{x}}dv\nnm\\
&\quad+C\intr |v|(\|\Tdx\Phi\|_{L^\infty_x}\|\dxa f_2\|_{L^2_{x}}+\|\dxa\Tdx\Phi\|_{L^2_x}\|f_2\|_{L^{\infty}_x})\|\dxa f_1\|_{L^2_{x}}dv\nnm\\
&\le C\|\Tdx\Phi\|_{H^{N}_x}\|w^{\frac12}\Tdx f_2\|_{L^2_v(H^{N-1}_x)}\|w^{\frac12}\dxa f_1\|_{L^2_{x,v}}\le C\sqrt{E(f_1,f_2)}D(f_1,f_2),\label{I_1}
\ema
for $|\alpha|\ge1$, and
\bma
I_1\le \intr |v|\|\Tdx\Phi\|_{L^3_x} \|f_2\|_{L^{2}_x}\| f_1\|_{L^6_{x}}dv\le C\sqrt{E(f_1,f_2)}D(f_1,f_2),
\ema
for $|\alpha|=0$. For $I_2$, it holds
\bma
I_2
&\le C\sum_{ |\alpha'|\le N/2}\intr\|\dx^{\alpha'}\Tdx\Phi\|_{L^\infty_x}\|\dx^{\alpha-\alpha'}\Tdv f_2\|_{L^{2}_x}\|\dxa f_1\|_{L^2_{x}}dv
\nnm\\
&\quad+C\sum_{|\alpha'|\ge N/2}\intr \|\dx^{\alpha'}\Tdx\Phi\|_{L^2_x}\|\dx^{\alpha-\alpha'}\Tdv f_2\|_{L^\infty_{x}}\|\dxa f_1\|_{L^2_{x}}dv-\intrr \Tdx\Phi\dxa\Tdv f_2\dxa f_1dxdv\nnm\\
&\le C\sqrt{E(f_1,f_2)}D(f_1,f_2)-\intrr \Tdx\Phi\dxa\Tdv f_2\dxa f_1dxdv.
\ema
For $I_3$, by \eqref{gamma} we obtain
\bma
I_3\le \|w^{-\frac12}\dxa \Gamma(f_1,f_1)\|_{L^2_{x,v}}\|w^{\frac12}\dxa   \P_1 f_1\|_{L^2_{x,v}}
\le C\sqrt{E(f_1,f_2)}D(f_1,f_2).\label{I_2}
\ema
Therefore, it follows from \eqref{G_0}--\eqref{I_2} that
\bma
\frac12\Dt \|\dxa f_1\|^2_{L^2_{x,v}}+\mu \|w^{\frac12}\dxa \P_1 f_1\|^2_{L^2_{x,v}} \le C\sqrt{E(f_1,f_2)}D(f_1,f_2)- \intrr \Tdx\Phi\dxa\Tdv f_2\dxa f_1dxdv.\label{G_01}
\ema

Similarly, taking inner product between $\dxa f_2$ and $\dxa\eqref{VPB8}$ with $|\alpha|\le N$ $(N\ge 4)$ we have
\bma
&\frac12\Dt (\|\dxa f_2\|^2_{L^2_{x,v}}+\|\dxa\Tdx\Phi\|^2_{L^2_x})-\intr (L_1\dxa f_2)\dxa f_2dxdv
\nnm\\
&\le C\sqrt{E(f_1,f_2)}D(f_1,f_2)-\intrr \Tdx\Phi\dxa\Tdv f_1\dxa f_2dxdv.\label{G_00}
\ema
Taking the summation \eqref{G_01}+\eqref{G_00} for $|\alpha|=0$ and $\sum_{1\le|\alpha|\le N}[\eqref{G_01}+\eqref{G_00}]$, we obtain \eqref{E_3} and \eqref{E_2} respectively.

 In order to enclose the energy inequality, we need to estimate the terms $\dxa\Tdv f$ with $|\alpha|\le N-1$. To this end, we rewrite \eqref{p_1} and \eqref{p_c} as
 \bma
&\partial_t(  \P_1f_1)+v\cdot\Tdx   \P_1f_1+\Tdx \Phi\cdot\Tdv
P_rf_2-L(  \P_1f_1)\nnm\\
&=\Gamma(f_1,f_1)+\frac12v\cdot\Tdx\Phi P_rf_2+ \P_0(v\cdot\Tdx   \P_1f_1-\frac12v\cdot\Tdx\Phi P_rf_2+\Tdx\Phi\cdot\Tdv P_r f_2)-  \P_1(v\cdot\Tdx  \P_0f_1),\label{G_2}
\ema
and
\bma
&\dt (P_r f_2)+v\cdot\Tdx P_rf_2-v\sqrt M\cdot\Tdx\Phi+\Tdx\Phi\cdot\Tdv  \P_1f_1+L_1(P_rf_2)\nnm\\
&=\Gamma(f_2,f_1)+\frac12v\cdot\Tdx\Phi  \P_1f_1+P_{\rm d}(v\cdot\Tdx P_rf_2)-(v\cdot\Tdx P_{\rm d}f_2-\frac12v\cdot\Tdx\Phi \P_0f_1
+\Tdx\Phi\cdot\Tdv \P_0f_1).\label{b_1}
\ema

Let $1\le k\le N$, and choose $\alpha,\beta $ with $|\beta|=k$ and $|\alpha|+|\beta|\le N$. Taking inner product between $\dxa\dvb \P_1f_1$ and \eqref{G_2} and between $\dxa\dvb P_rf_2$ and \eqref{b_1} respectively, and then taking summation of the resulted equations, we have
\bma
&\Dt\sum_{|\beta|=k \atop |\alpha|+|\beta|\le N}\|\dxa\dvb (\P_1f_1,P_rf_2)\|^2_{L^2_{x,v}}+\mu\sum_{|\beta|=k \atop |\alpha|+|\beta|\le N}\|w^{\frac12}\dxa\dvb (\P_1f_1,P_rf_2)\|^2_{L^2_{x,v}}\nnm\\
&\le C\sum_{|\alpha|\le N-k}(\|\dxa\Tdx (\P_0f_1,P_{\rm d}f_2)\|^2_{L^2_{x,v}}+\|\dxa\Tdx (\P_1f_1,P_rf_2)\|^2_{L^2_{x,v}}+\|\dxa\Tdx\Phi\|^2_{L^2_x})\nnm\\
&\quad+C_k\sum_{ |\beta|\le k-1\atop |\alpha|+|\beta|\le N}\|\dxa\dvb (\P_1f_1,P_rf_2)\|^2_{L^2_{x,v}}+C\sqrt{E(f_1,f_2)}D(f_1,f_2),\label{aa}
\ema
Then taking summation $\sum_{1\le k\le N}p_k\eqref{aa}$ with  constants $p_k$ chosen by
$$\mu p_k\ge 2\sum_{1\le j\le N-k}p_{k+j}C_{k+j},\quad 1\le k\le N-1,\quad p_N=1,$$
we obtain \eqref{E_5}.
The proof of the lemma is completed.
\end{proof}

With the help of Lemma \ref{macro-en}--\ref{micro-en}, we have the following result.

\begin{lem}\label{energy1}
Let $N\ge 4$. Then, there are two equivalent energy functionals
$E^f_0(\cdot)\sim E(\cdot)$ and $H^f_0(\cdot)\sim H(\cdot)$
such that the following holds. If
$E(f_{1,0},f_{2,0})$ is sufficiently small, then the Cauchy problem
\eqref{VPB7}--\eqref{VPB10} of the bVPB system admits a unique global
solution $(f_1,f_2)(x, v, t)$ satisfying
\bma
\Dt E^f_0(f_1,f_2)(t) + \mu D(f_1,f_2)(t) &\le 0, \label{G_1}
\\
\Dt H^f_0(f_1,f_2)(t)+\mu D(f_1,f_2)(t)&\le C\|\Tdx (n_1,m_1,q_1)\|^2_{L^2_{x}}.\label{G_4}
\ema
\end{lem}
\begin{proof}
Assume that $$E(f_1,f_2)(t)\le \delta$$ for $\delta>0$ small.

Taking the summation of $A_1[\eqref{E_1}+\eqref{E_1a}]+A_2[\eqref{E_3}+\eqref{E_2}]+\eqref{E_5}$ with $A_2>C_0A_1>0$ large enough and taking $k=0$ in \eqref{E_1} and \eqref{E_1a}, we prove \eqref{G_1}.

Taking $|\alpha|=0$ in \eqref{G_00} and noting that the terms in the right hand side are bounded by $\sqrt{E(f_1,f_2)}D(f_1,f_2)$, we have
\bq \frac12\Dt (\|f_2\|^2_{L^2_{x,v}}+\|\Tdx\Phi\|^2_{L^2_x})+\mu \|w^{\frac12}P_rf_2\|^2_{L^2_{x,v}}\le \sqrt{E(f_1,f_2)}D(f_1,f_2).\label{E_4}\eq
Taking the inner product between \eqref{G_2} and $  \P_1f_1$, we have
\bma
\Dt \|   \P_1f_1\|^2_{L^2_{x,v}}+\|w^{\frac12} \P_1f_1\|^2_{L^2_{x,v}}\le C\|\Tdx \P_0f_1\|^2_{L^2_{x,v}}+E(f_1,f_2)D(f_1,f_2). \label{low1}
\ema

Taking the summation of $A_1[\eqref{E_1}+\eqref{E_1a}]+A_2[\eqref{E_4}+\eqref{low1}+\eqref{E_2}]+\eqref{E_5}$ with $A_2>C_0A_1>0$ large enough and taking $k=1$ in \eqref{E_1} and $k=0$ in \eqref{E_1a}, we prove \eqref{G_4}.
\end{proof}

Repeating the proof of Lemmas \ref{macro-en}--\ref{micro-en}, we can show
\begin{lem}\label{energy2}
Let $N\ge 4$.  There are the equivalent energy functionals $E^f_1(\cdot)\sim E_1(\cdot)$, $H^f_1(\cdot)\sim H_1(\cdot)$
such that if $E_1(f_{1,0},f_{2,0})$ is sufficiently small, then the solution $(f_1,f_2)(x,v,t)$ to the bVPB system \eqref{VPB7}--\eqref{VPB10} satisfies
\bma &\Dt E^f_1(f_1,f_2)(t)+\mu D_1(f_1,f_2)(t)\le 0,\label{G_4b}\\
&\Dt H^f_1(f_1,f_2)(t)+\mu D_1(f_1,f_2)(t)\le C\|\Tdx (n_1,m_1,q_1)\|^2_{L^2_{x}}.\label{G_4a}
\ema
\end{lem}

\subsection{Convergence rates}
\begin{thm}\label{time5}
Assume that $(f_{1,0},f_{2,0})\in H^N_w\cap L^{2,1}$ for $N\ge 4$ and  $\|(f_{1,0},f_{2,0})\|_{H^N_{w}\cap L^{2,1}}\le \delta_0$ for a constant  $\delta_0>0$ small enough. Then, the global solution $(f_1,f_2)(x,v,t)$ to the bVPB system~\eqref{VPB7}--\eqref{VPB10} satisfies
 \bma
\|\dx^k f_2(t)\|_{L^2_{x,v}}+\|\dx^k\Tdx\Phi(t)\|_{L^2_x}&\le C\delta_0e^{-dt},\quad t>0, \label{t_3a}
\ema
for $k=0,1$ and $d>0$ a constant.
\end{thm}
\begin{proof}
Let $(f_1,f_2)$ be a solution to the Cauchy problem \eqref{VPB7}--\eqref{VPB10} for $t>0$.  Firstly, we prove that there is a constant $K_2>0$ and energy functionals $E_2(f_2),D_2(f_2)$ of $f_2$ defined by
\bmas
E_2(f_2)&=K_2\sum_{|\alpha|\le 1}(\|\dxa f_2\|^2_{L^2_{x,v}}+\|\dxa\Tdx\Phi\|^2_{L^2_x})+\|n_2\|^2_{L^2_x}+\|\Tdx\Phi\|^2_{L^2_x}\\
&\quad+2 \intr\divx R_5 n_2 dx+2 \intr R_5 \Tdx\Phi dx, \\
D_2(f_2)&=\sum_{|\alpha|\le1}\|w^{\frac12}\dxa P_rf_2\|^2_{L^2_{x,v}}+\| n_2\|^2_{L^2_x}+\|\Tdx n_2\|^2_{L^2_x}+\|\Tdx\Phi\|^2_{L^2_x},
\emas
such that
\bma
\Dt E_2(f_2(t))+\kappa_3 D_2(f_2(t))\le 0\quad {\rm and}\quad E_2(f_2)\le CD_2(f_2).\label{energy4}
\ema
Indeed, taking the inner product between $ n_2$ and \eqref{G_9} and using Cauchy-Schwarz inequality and Sobolev embedding, we obtain 
\bma
&\Dt (\| n_2\|^2_{L^2_x}+2 \intr \divx R_5 n_2 dx)+\kappa_3(\| n_2\|^2_{L^2_x}+\|\Tdx n_2\|^2_{L^2_x})
\nnm\\
&\le C\|\Tdx P_rf_2\|^2_{L^2_{x,v}}+CE_1(f_1,f_2)D_2(f_2).\label{G_0d}
\ema
Similarly, taking the inner product between $ -\Phi$ and \eqref{G_9} we obtain
\bma
&\Dt (\| \Tdx\Phi\|^2_{L^2_x}+2 \intr R_5 \Tdx\Phi dx)+\kappa_3(\| \Tdx\Phi\|^2_{L^2_x}+\| n_2\|^2_{L^2_x})
\nnm\\
&\le C\| P_rf_2\|^2_{L^2_{x,v}}+CE_1(f_1,f_2)D_2(f_2).\label{G_0e}
\ema
Taking the inner product between $ f_2$ and \eqref{VPB8}, we have
\bma
\frac12\Dt (\| f_2\|^2_{L^2_{x,v}}+\|\Tdx\Phi\|^2_{L^2_x})-\intr (L_1 f_2) f_2dxdv
\le C\sqrt{E_1(f_1,f_2)}D_2(f_2).\label{G_0a}
\ema
Again, taking the inner product between $\dxa f_2$ and $\dxa\eqref{VPB8}$ with $|\alpha|= 1$ to get
\bma
\frac12\Dt (\|\dxa f_2\|^2_{L^2_{x,v}}+\|\dxa\Tdx\Phi\|^2_{L^2_x})-\intr (L_1\dxa f_2)\dxa f_2dxdv
\le C\sqrt{E_1(f_1,f_2)}D_2(f_2).\label{G_0b}
\ema

Making the summation of $K_2 [\eqref{G_0a}+\eqref{G_0b}]+\eqref{G_0d}+\eqref{G_0e}$ for a sufficiently large constant $K_2>0$, we obtain
\eqref{energy4} for $E_1(f_1,f_2)>0$ small enough due to Lemma \ref{energy2}. Then \eqref{t_3a} follows from \eqref{energy4} by Gronwall's inequality.
\end{proof}

\begin{thm}\label{time6}
Assume that $(f_{1,0},f_{2,0})\in H^N_w\cap L^{2,1}$ for $N\ge 4$ and  $\|(f_{1,0},f_{2,0})\|_{H^N_{w}\cap L^{2,1}}\le \delta_0$ for a constant  $\delta_0>0$ small enough. Then, the global solution $(f_1,f_2)(x,v,t)$ to the bVPB system~\eqref{VPB7}--\eqref{VPB10} satisfies
 \bgr
\|\dx^k(f_1(t),\chi_j)\|_{L^2_{x}}\le C\delta_0(1+t)^{-\frac34-\frac k2},\quad j=0,1,2,3,4,\label{t_1a}\\
\|\dx^k  \P_1f_1(t)\|_{L^2_{x,v}}\le C\delta_0(1+t)^{-\frac54-\frac k2},\label{t_2a}\\
\|(\P_1f_1,P_rf_2)(t)\|_{H^N_w} +\|\Tdx (\P_0f_1,P_{\rm d}f_2)(t)\|_{L^2_v(H^{N-1}_x)}\le
C\delta_0(1+t)^{-\frac54},\label{t_4a}
\egr
for $k=0,1$.
\end{thm}
\begin{proof}
Let $(f_1,f_2)$ be a solution to the Cauchy problem \eqref{VPB7}--\eqref{VPB10} for $t>0$. We can represent the solution in terms of the semigroup $e^{tE},e^{tB}$  as
 \bma
 f_1(t)=e^{tE}f_{1,0}+\intt e^{(t-s)E}G_1(s)ds, \label{Duh1}\\
 f_2(t)=e^{tB}f_{2,0}+\intt e^{(t-s)B}G_2(s)ds, \label{Duh2}
 \ema
where the nonlinear terms $G_1$ and $G_2$ are given by \eqref{G1} and \eqref{G2} respectively.
 For this global solution $f_1,f_2$, 
we define a functional $Q(t)$  for any $t>0$ as
\bmas Q(t)=\sup_{0\le s\le t}\sum_{k=0}^1&\Big\{\sum_{j=0}^4\|\dx^k(f_1(s),\chi_j)\|_{L^2_{x}}(1+s)^{\frac34+\frac k2}+\|\dx^k  \P_1f_1(s)\|(1+s)^{\frac54+\frac k2}\\
&+(\|(\P_1f_1,P_rf_2)(s)\|_{H^N_w} +\|\Tdx \P_0f_1(s)\|_{L^2_v(H^{N-1}_x)}+\|P_{\rm d}f_2(s)\|_{L^2_v(H^{N}_x)})(1+s)^{\frac54}\Big\}. \emas
We claim that it holds under the assumptions of Theorem~\ref{time6} that
  \bq
 Q(t)\le C\delta_0.  \label{assume}
 \eq

First, let us deal with the time-decay rate of the macroscopic density, momentum and energy of $f_1$, which in terms of \eqref{Duh1} satisfy the following equations
 \be
(f_1(t),\chi_j)=(e^{tE}f_{1,0},\chi_j)+\intt (e^{(t-s)E}G_1(s),\chi_j)ds,\quad j=0,1,2,3,4. \label{maceq}
 \ee
By Lemma \ref{e1} and \eqref{t_3},
we can estimate the nonlinear term $G_1(s)$ for $0\le s\le t$ in terms of $Q(t)$ as
 \bma
 \| G_1(s)\|_{L^2_{x,v}}
 &\le
 C\{\|wf_1\|_{L^{2,3}}\|f_1\|_{L^{2,6}}
    +\|\Tdx\Phi\|_{L^3_x}(\|wf_2\|_{L^{2,6}}
    +\|\Tdv f_2\|_{L^{2,6}})\}
    \nnm\\
&\le C(1+s)^{-2}Q(t)^2+C\delta_0e^{-ds}(1+s)^{-\frac54}Q(t),  \label{GG_1}
\\
 \| G_1(s)\|_{L^{2,1}}
 & \le
 C\{ \|f_1\|_{L^2_{x,v}}\|w f_1\|_{L^2_{x,v}}
    +\|\Tdx\Phi\|_{L^3_x}(\|w f_2\|_{L^2_{x,v}}
    +\|\Tdv  f_1\|_{L^2_{x,v}})\}
    \nnm\\
 &\le C(1+s)^{-\frac32}Q(t)^2+C\delta_0e^{-ds}(1+s)^{-\frac34}Q(t),\label{GG_2}
 \ema
Then, it follows from \eqref{V_1}, \eqref{GG_1} and \eqref{GG_2} that
\bma
\|(f_1(t),\chi_j)\|_{L^2_x}&\le
C(1+t)^{-\frac34}(\|f_{1,0}\|_{L^2_{x,v}}+\|f_{1,0}\|_{L^{2,1}})\nnm\\
&\quad+C\intt (1+t-s)^{-\frac34}(\| G_1(s)\|_{L^2_{x,v}}+\|G_1(s)\|_{L^{2,1}})ds\nnm\\
&\le C\delta_0(1+t)^{-\frac34}+C\intt (1+t-s)^{-\frac34}[(1+s)^{-\frac32}Q(t)^2+C\delta_0e^{-ds}(1+s)^{-\frac34}Q(t)]ds\nnm\\
&\le
C\delta_0(1+t)^{-\frac34}+C\delta_0(1+t)^{-\frac34}Q(t)+C(1+t)^{-\frac34}Q(t)^2. \label{macro_1}
 \ema
 Similarly, we have
 \bma
 \|(\Tdx f_1(t),\chi_j)\|_{L^2_x}
 &\le
 C(1+t)^{-\frac54}(\|\Tdx f_{1,0}\|_{L^2_{x,v}}+\|f_{1,0}\|_{L^{2,1}})
 \nnm\\
&\quad+C\intt (1+t-s)^{-\frac54}(\|\Tdx G_1(s)\|_{L^2_{x,v}}
  +\| G_1(s)\|_{L^{2,1}})ds
 \nnm\\
&\le
    C\delta_0(1+t)^{-\frac54}
   +C\intt (1+t-s)^{-\frac54}[(1+s)^{-\frac32}Q(t)^2+C\delta_0e^{-ds}(1+s)^{-\frac34}Q(t)]ds
 \nnm\\
&\le
 C\delta_0(1+t)^{-\frac54}+C\delta_0(1+t)^{-\frac54}Q(t)+C (1+t)^{-\frac54}Q(t)^2,\label{macro_2}
 \ema
 where we use
 $$\|\Tdx G_1(s)\|_{L^2_{x,v}}+\| G_1(s)\|_{L^{2,1}}\le C(1+s)^{-2}Q(t)^2+C\delta_0e^{-ds}(1+s)^{-\frac54}Q(t).$$

Second, we estimate the microscopic part $  \P_1f_1(t)$ as below. Since
$$  \P_1f_1(t)=  \P_1(e^{tE}f_{1,0})+\intt   \P_1(e^{(t-s)E}G_1(s))ds,$$
it follows from 
\eqref{V_2}, \eqref{GG_1} and \eqref{GG_2} that
\bma
\|  \P_1f_1(t)\|_{L^2_{x,v}}&\le
C(1+t)^{-\frac54}(\|f_{1,0}\|_{L^2_{x,v}}+\|f_{1,0}\|_{L^{2,1}})\nnm\\
&\quad+C\intt (1+t-s)^{-\frac54}(\| G_1(s)\|_{L^2_{x,v}}+\|G_1(s)\|_{L^{2,1}})ds\nnm\\
&\le
C\delta_0(1+t)^{-\frac54}+C\delta_0(1+t)^{-\frac54}Q(t)+C(1+t)^{-\frac54}Q(t)^2,\label{micro_1}
\ema
and
\bma
\|\Tdx   \P_1f_1(t)\|_{L^2_{x,v}}&\le
C(1+t)^{-\frac74}(\|\Tdx f_{1,0}\|_{L^2_{x,v}}+\|f_{1,0}\|_{L^{2,1}})\nnm\\
&\quad+C\int^{t/2}_0 (1+t-s)^{-\frac74}(\|\Tdx G_1(s)\|_{L^2_{x,v}}+\|G_1(s)\|_{L^{2,1}})ds\nnm\\
&\quad+C\int^t_{t/2} (1+t-s)^{-\frac54}(\|\Tdx G_1(s)\|_{L^2_{x,v}}+\|\Tdx G_1(s)\|_{L^{2,1}})ds\nnm\\
&\le
C\delta_0(1+t)^{-\frac74}+C\delta_0(1+t)^{-\frac74}Q(t)+C(1+t)^{-\frac74}Q(t)^2,\label{micro_2}\ema
where we have used
$$
\|\Tdx G_1(s)\|_{L^2_{x,v}}+\|\Tdx G_1(s)\|_{L^{2,1}}\le
C(1+s)^{-2}Q(t)^2+C\delta_0e^{-ds}(1+s)^{-\frac54}Q(t).
$$
Next, we estimate the higher order terms as below. By \eqref{G_4a} and $cH^f_1(f_1,f_2)\le D_1(f_1,f_2)$ with $c>0$, we have
\bma
H^f_1(f_1,f_2)(t)&\le e^{-c\mu t}H^f_1(f_{1,0},f_{2,0})+\intt e^{-c\mu(t-s)}\|\Tdx (n_1,m_1,q_1)(s)\|^2_{L^2_{x}}ds\nnm\\
&\le C\delta_0^2e^{-c\mu t} +\intt
e^{-c\mu(t-s)}(1+s)^{-\frac52}(\delta_0+\delta_0Q(t)+Q(t)^2)^2ds\nnm\\
&\le C(1+t)^{-\frac52}(\delta_0+\delta_0Q(t)+Q(t)^2)^2.\label{other}
\ema
By summing \eqref{macro_1}, \eqref{macro_2}, \eqref{micro_1}, \eqref{micro_2} and \eqref{other}, we have
$$Q(t)\le C\delta_0+C\delta_0Q(t)+CQ(t)^2,$$
which proves \eqref{assume} for $\delta_0>0$ small enough. This completes the proof of the theorem.
\end{proof}

\medskip

\begin{proof}[\it\underline{Proof of Theorem \ref{rate3}}]
Firstly, \eqref{t_1}--\eqref{t_4} and \eqref{t_7} follow from Theorem \ref{time5} and Theorem \ref{time6}.
 Since
\bq f_+=\frac12(f_1+f_2),\quad f_-=\frac12(f_1-f_2),\label{fab}\eq
this and Theorem \ref{time6} imply \eqref{t_5z}--\eqref{t_8}.
\end{proof}

\medskip

\begin{proof}[\it \underline{Proof of Theorem \ref{rate4}}]
By Eq.~\eqref{Duh1}, Theorem~\ref{time5} and Theorem~\ref{time6}, we can establish the lower bounds of the time decay rates of macroscopic density, momentum and energy of the global solution $(f_1,f_2)$ to the bVPB system \eqref{VPB7}--\eqref{VPB10} and its microscopic part for $t>0$ large enough. Indeed, it holds for $k=0,1$  that
 \bmas
\|\Tdx^k(f_1(t),\chi_j)\|_{L^2_{x}}
&\ge
 \|\Tdx^k (e^{tE}f_{1,0},\chi_j)\|_{L^2_{x}}
 -\intt\|\Tdx^k(e^{(t-s)E}G_1(s),\chi_j)\|_{L^2_{x}}ds
 \\
&\ge
 C_1\delta_0(1+t)^{-\frac34-\frac k2}-C_2\delta_0^2(1+t)^{-\frac34-\frac k2},
 \\
\| \Tdx^k \P_1f_1(t)\|_{L^2_{x,v}}
&\ge
 \|\Tdx^k\P_1(e^{tE}f_{1,0})\|_{L^2_{x,v}}
 -\intt\| \Tdx^k \P_1(e^{(t-s)E}G_1(s))\|_{L^2_{x,v}}ds
 \\
&\ge C_1\delta_0(1+t)^{-\frac54-\frac k2}-C_2\delta_0^2(1+t)^{-\frac54-\frac k2},
\emas
which give rise to
\bmas
\|f_1(t)\|_{H^N_w}&\ge\| \P_0f_1(t)\|_{L^2_{x,v}}-\|w  \P_1f_1(t)\|_{L^2_{x,v}}-\sum_{1\le |\alpha|\le N}\|w\dxa f_1(t)\|_{L^2_{x,v}}\\
&\ge C_1\delta_0(1+t)^{-3/4}-C_2\delta_0^2(1+t)^{-3/4}-C_3\delta_0(1+t)^{-5/4}.\emas
Therefore, for $\delta_0>0$
small and $t>0$ large, we obtain \eqref{B_1}--\eqref{B_4}.
By Theorem \ref{time5}, Theorem \ref{time6} and \eqref{fab}, we can prove \eqref{B_4a}--\eqref{B_3a}. 
\end{proof}

\section{The nonlinear problem for mVPB system}\setcounter{equation}{0}
\label{mvpb}

\subsection{Energy estimates}
\label{Lmvpb}

Let $N$ be a positive integer, and let \bmas E_k(f)&=\sum_{|\alpha|+|\beta|\le N}\|w^k\dxa\dvb f\|^2_{L^2_{x,v}} +\sum_{|\alpha|\le N}\|\dxa\Phi\|^2_{H^1_x},\\
H_k(f)&=\sum_{|\alpha|+|\beta|\le N}\|w^k\dxa\dvb
\P_1f\|^2_{L^2_{x,v}} +\sum_{|\alpha|\le N-1}(\|\dxa\Tdx \P_0f\|^2_{L^2_{x,v}}+\|\dxa \Tdx\Phi\|^2_{H^1_x}),\\
D_k(f)&=\sum_{|\alpha|+|\beta|\le N}\|w^{\frac12+k}\dxa\dvb \P_1f\|^2_{L^2_{x,v}}+\sum_{|\alpha|\le N-1}(\|\dxa\Tdx \P_0f\|^2_{L^2_{x,v}}+\|\dxa\Tdx\Phi\|^2_{H^1_x}),
\emas
for $k\ge 0$. For brevity, we write $E(f)=E_0(f)$, $H(f)=H_0(f)$ and $D(f)=D_0(f )$ for $k=0$.

Applying the similar arguments as to derive the equation~\eqref{G_9}--\eqref{G_8} and making use of the system \eqref{m_VPB4}--\eqref{m_VPB5}, we can obtain the  compressible Navier-Stokes-Poisson (NSP) equations with inhomogeneous terms for the macroscopic density, momentum and energy $(n,  m, q)=:((f,\chi_0),(f,v\chi_0),(f,\chi_4))$ as follows
\bma \dt n+\divx  m&=0,\label{m_G_9}\\
\dt  m+\dt  R_6+\Tdx n+\sqrt{\frac23}\Tdx q-\Tdx\Phi&=\kappa_5 (\Delta_x m+\frac13\Tdx{\rm div}_x m)+n\Tdx \Phi+R_7,\label{m_G_7}\\
\dt q+\dt  R_8+\sqrt{\frac23}\divx m&=\kappa_6 \Delta_x q+\sqrt{\frac23}\Tdx \Phi\cdot m+R_9,\label{m_G_8}
\ema
where $\kappa_5$, $\kappa_6>0$ are the viscosity coefficients and the remainder terms $R_6, R_7, R_8, R_9$ are defined by
\bmas \kappa_5&=-(L^{-1}\P_1(v_1\chi_2),v_1\chi_2),\quad \kappa_6=-(L^{-1}\P_1(v_1\chi_4),v_1\chi_4),\\
R_6&=(v\cdot\Tdx L^{-1}\P_1f,v\sqrt M),\quad R_7=-( v\cdot\Tdx L^{-1}(\P_1(v\cdot\Tdx \P_1f)-\P_1 G),v\sqrt M),\\
R_8&=(v\cdot\Tdx L^{-1}\P_1f,\chi_4),\quad R_9=-( v\cdot\Tdx L^{-1}(\P_1(v\cdot\Tdx \P_1f)-\P_1 G),\chi_4).\emas

We have the following estimates.

\begin{lem}[Macroscopic dissipation] \label{m_MacDis}
Let $(n,m,q)$ be a strong solution to \eqref{m_G_9}--\eqref{m_G_8} and assume that the energy $E(f)>0$ is small enough. Then, there is a constant $p_0>0$ so that it holds for $t>0$ that
\bma
&\Dt \sum_{k\le |\alpha|\le N-1}p_0(\|\dxa(n, m,q)\|^2_{L^2_x}+\|\dxa\Phi\|^2_{H^1_x}+2\intr \dxa R_6\dxa mdx+2\intr \dxa R_8\dxa qdx)\nnm\\
&+\Dt \sum_{k\le |\alpha|\le N-1}4\intr \dxa m \dxa\Tdx ndx
+\sum_{k\le |\alpha|\le N-1}( \|\dxa\Tdx (n, m,q)\|^2_{L^2_x}+\|\dxa \Tdx\Phi\|^2_{H^1_x})
\nnm\\
\le& C\sqrt{E(f)}D(f)+C\sum_{k\le |\alpha|\le N-1}\|\dxa\Tdx \P_1f\|^2_{L^2_{x,v}}\label{m_E_1}
\ema
with $0\le k\le N-1$.

\end{lem}
\begin{proof}
Taking the inner product between $\dxa m$ and $\dxa\eqref{m_G_7}$ with $|\alpha|\le N-1$, we have after a direct computation
 \bma
 &\frac12\Dt ( \|\dxa m\|^2_{L^2_x} + \|\dxa n \|^2_{L^2_x} + \|\dxa\Tdx\Phi\|^2_{L^2_x}+\|\dxa\Phi\|^2_{L^2_x} )
 +\Dt \intr \dxa R_6\dxa mdx
\nnm\\
&+\sqrt{\frac23}\intr \dxa\Tdx q\dxa mdx+\kappa_5 (\|\dxa\Tdx  m\|^2_{L^2_x}+\frac13\|\dxa\divx  m\|^2_{L^2_x})
\nnm\\
=&\intr \dxa(n\Tdx \Phi) \dxa mdx+\intr\dxa R_7\dxa mdx+\intr \dxa R_6\dxa\dt m dx
  - \intr \dxa\Phi\dxa[(e^{-\Phi}-1)\dt\Phi] dx.\label{m_en_1}
   \ema
The first  term in the right hand side  of \eqref{m_en_1} are bounded by $C\sqrt{E(f)}D(f)$. The second and third terms can be estimated by
\bma
\intr \dxa R_6\dxa\dt m dx&\le \frac{C}{\epsilon}\|\dxa\Tdx \P_1f\|^2_{L^2_{x,v}}+C\sqrt{E(f)}D(f)\nnm\\
&\quad+\epsilon(\|\dxa\Tdx n\|^2_{L^2_x}+\|\dxa \Tdx\Phi\|^2_{L^2_x}+\|\dxa\Tdx q\|^2_{L^2_x}),\label{m_en_3}
\\
\intr\dxa R_7\dxa mdx
&\le C \|\dxa\Tdx \P_1f\|_{L^2_{x,v}}\|\dxa\Tdx m\|_{L^2_x}\nnm\\
&\quad+C(\|\dxa(\Tdx\Phi f)\|_{L^2_{x,v}}+\|w^{-\frac12}\dxa \Gamma(f,f)\|_{L^2_{x,v}})\|\dxa\Tdx m\|_{L^2_x}\nnm\\
&\le \frac{\kappa_5}2\|\dxa\Tdx m\|^2_{L^2_x}+C\|\dxa\Tdx \P_1f\|^2_{L^2_{x,v}}+C\sqrt{E(f)}D(f),\label{m_en_4}
\ema
where we have used Lemma \ref{e1} to obtain
\bmas
\|w^{-\frac12}\dxa \Gamma(f,f)\|^2_{L^2_{x,v}}+\|\dxa(\Tdx\Phi f)\|^2_{L^2_{x,v}}
\le
 C E(f)D(f).
\emas
For the last term, we make use of \eqref{m_VPB5} to obtain
\bq
\dt\Phi-\Delta_x\dt\Phi=-\dt n-(e^{-\Phi}-1)\dt\Phi=\divx  m-(e^{-\Phi}-1)\dt\Phi, \label{m_Phi}
\eq
which after taking the inner product $\dxa\dt\Phi$ and $\dxa\eqref{m_Phi}$ leads  for $E(f)$ small enough to
\be
 \sum_{|\alpha|\le N-1}(\|\dxa\Tdx\dt\Phi\|^2_{L^2_x}+\|\dxa\dt\Phi\|^2_{L^2_x})
\le
 C\sum_{|\alpha|\le N-1}\|\dxa\divx m\|^2_{L^2_x},\label{m_phi_t}
 \ee
and
\bq
  \intr \dxa\Phi\dxa[(e^{-\Phi}-1)\dt\Phi] dx
\le
 C\|\Tdx\Phi\|_{H^{N}_x}\|\Phi\|_{H^{N}_x}\|\Tdx m\|_{H^{N-1}_x}\le C\sqrt{E(f)}D(f). \label{m_en_6}
\eq


Therefore, it follows from \eqref{m_en_1}, \eqref{m_en_3}, \eqref{m_en_4} and \eqref{m_en_6} that
\bma &\frac12\Dt (\|\dxa m\|^2_{L^2_x}+\|\dxa n\|^2_{L^2_x}+\|\dxa\Tdx\Phi\|^2_{L^2_x}+\|\dxa\Phi\|^2_{L^2_x})+\Dt \intr \dxa R_6\dxa mdx\nnm\\
&\quad+\sqrt{\frac23}\intr \dxa\Tdx q\dxa mdx+\frac{\kappa_5}2 (\|\dxa\Tdx  m\|^2_{L^2_x}+\frac13\|\dxa\divx  m\|^2_{L^2_x})\nnm\\
&\le C\sqrt{E(f)}D(f)
+C\|\dxa\Tdx \P_1f\|^2_{L^2_{x,v}}+\epsilon(\|\dxa\Tdx n\|^2_{L^2_x}+\|\dxa \Tdx\Phi\|^2_{L^2_x}+\|\dxa\Tdx q\|^2_{L^2_x}).\label{m_m_1}
\ema

Similarly, taking the inner product between $\dxa q$ and $\dxa\eqref{m_G_8}$ with $|\alpha|\le N-1$, we have
\bma
&\frac12\Dt \|\dxa q\|^2_{L^2_x}+\Dt \intr \dxa R_8\dxa qdx
 +\sqrt{\frac23}\intr \dxa\divx m\dxa qdx+ \frac12\kappa_6 \|\dxa\Tdx q\|^2_{L^2_x}\nnm\\
&\le C\sqrt{E(f)}D(f)+C\|w^{\frac12}\dxa\Tdx \P_1f\|^2_{L^2_{x,v}}+\epsilon\|\dxa\Tdx m\|^2_{L^2_x}.\label{m_q_1}
\ema

Agian, taking the inner product between $\dxa\Tdx n$ and $\dxa\eqref{G_5}$ with $|\alpha|\le N-1$ to get
 \bma
&\Dt\intr \dxa m \dxa\Tdx ndx
 +  \frac12\|\dxa\Tdx n\|^2_{L^2_x}
 +\|\dxa \Delta_x\Phi\|^2_{L^2_x}+\|\dxa \Tdx\Phi\|^2_{L^2_x}
  \nnm\\
&\le
 C\sqrt{E(f)}D(f)+ \|\dxa\divx  m\|^2_{L^2_x}+\|\dxa \Tdx q\|^2_{L^2_x} + C\|\dxa\Tdx (\P_1f)\|^2_{L^2_{x,v}}.    \label{m_abc}
  \ema
Making the summation $p_0\sum\limits_{k\le |\alpha|\le N-1}[\eqref{m_m_1}+\eqref{m_q_1}]+4\sum\limits_{k\le |\alpha|\le N-1}\eqref{m_abc}$ with the constant $p_0>0$ large enough and $\epsilon>0$ small enough, we can obtain \eqref{m_E_1}.
The proof of the lemma is completed.
\end{proof}

In the followings, we shall estimate  the microscopic part $\P_1f$ appearing in \eqref{m_E_1} in order to enclose the energy estimates of the solution $f$  to mVPB~\eqref{m_VPB4}--\eqref{m_VPB5}.

\begin{lem}[Microscopic dissipation] \label{m_MicDis}
Let $(f,\Phi)$ be a strong solution to mVPB~\eqref{m_VPB4}--\eqref{m_VPB5}.
Then, there are constants $p_k>0$, $1\le k\le N$ so that it holds for $t>0$ that
 \bma
 & \frac12\Dt \sum_{1\le |\alpha|\le N}(\|\dxa f\|^2_{L^2_{x,v}}
  +\|\dxa\Phi\|^2_{H^1_x})+\mu\sum_{1\le |\alpha|\le N} \|w^{\frac12}\dxa \P_1f\|^2_{L^2_{x,v}}\le C\sqrt{E(f)}D(f),\label{m_E_2}
  \\
&\Dt \| \P_1f\|^2_{L^2_{x,v}}+\mu\|w^{\frac12}\P_1f\|^2_{L^2_{x,v}}\le C\|\Tdx \P_0f\|^2_{L^2_{x,v}}+ CE(f)D(f),\label{m_E_4}
\\
&\Dt\sum_{1\le k\le N}p_k\sum_{|\beta|=k \atop |\alpha|+|\beta|\le N} \|\dxa\dvb \P_1f\|^2_{L^2_{x,v}}
+\mu\sum_{1\le k\le N}p_k\sum_{|\beta|=k \atop |\alpha|+|\beta|\le N}\|w^{\frac12}\dxa\dvb \P_1f\|^2_{L^2_{x,v}}\nnm\\
&\le C\sum_{|\alpha|\le N-1}(\|\dxa\Tdx \P_0f\|^2_{L^2_{x,v}}+\|\dxa\Tdx \P_1f\|^2_{L^2_{x,v}})+C\sqrt{E(f)}D(f). \label{m_E_5}
\ema
\end{lem}
\begin{proof}
Taking the inner product between $\dxa  f$ and $\dxa\eqref{m_VPB4}$ with $1\le |\alpha|\le N$ $(N\ge 4)$, we have after a tedious computation that
\bma
\frac12\Dt \|\dxa f\|^2_{L^2_{x,v}}
 +\frac12\Dt(\|\dxa \Tdx\Phi\|^2_{L^2_x}+\|\dxa\Phi\|^2_{L^2_x})
 + \mu\sum_{1\le |\alpha|\le N} \|w^{\frac12}\dxa \P_1f\|^2_{L^2_{x,v}}
\le
   C\sqrt{E(f)}D(f).
 \ema

In order to enclose the energy estimates, we need to estimate the terms $\dxa\Tdv f$ with $|\alpha|\le N-1$. To this end, we rewrite \eqref{m_VPB4} as
 \bma
&\partial_t(\P_1f)+v\cdot\Tdx \P_1f+\Tdx \Phi\cdot\Tdv
\P_1f-L(\P_1f)\nnm\\
=&\Gamma(f,f)+\frac12v\cdot\Tdx\Phi\P_1 f+\P_0(v\cdot\Tdx \P_1f+\Tdx\Phi\cdot\Tdv \P_1f-\frac12v\cdot\Tdx\Phi\P_1 f)\nnm\\
&\quad+\P_1(\frac12v\cdot\Tdx\Phi\P_0 f-v\cdot\Tdx \P_0f-\Tdx\Phi\cdot\Tdv \P_0f).\label{m_G_2}
\ema
Taking the inner product between $ \P_1 f$ and \eqref{m_G_2} and using Cauchy-Schwarz inequality, we can obtain \eqref{m_E_4}.

Let $1\le k\le N$, and choose $\alpha,\beta $ with $|\beta|=k$ and $|\alpha|+|\beta|\le N$. Taking the inner product between $\dxa\dvb \P_1 f$ and $\dxa\dvb\eqref{m_G_2}$ and summing the resulted equations,   we obtain  after a tedious computation
\bma
& \sum_{|\beta|=k \atop |\alpha|+|\beta|\le N}\Dt\|\dxa\dvb \P_1f\|^2_{L^2_{x,v}}
+\mu\sum_{|\beta|=k \atop |\alpha|+|\beta|\le N}\|w^{\frac12}\dxa\dvb \P_1f\|^2_{L^2_{x,v}}\nnm\\
&\le C\sum_{|\alpha|\le N-k}(\|\dxa\Tdx \P_0f\|^2_{L^2_{x,v}}+\|\dxa\Tdx \P_1f\|^2_{L^2_{x,v}})+C_k\sum_{ |\beta|\le k-1\atop |\alpha|+|\beta|\le N}\|\dxa\dvb\P_1f\|^2_{L^2_{x,v}}\nnm\\
&\quad+C\sqrt{E(f)}D(f).\label{m_aa}
\ema
Finally, taking the summation $\sum_{1\le k\le N}p_k\eqref{m_aa}$ with
$$\mu p_k\ge 2\sum_{1\le j\le N-k}p_{k+j}C_{k+j},\quad 1\le k\le N-1,\quad p_N=1,$$
we obtain \eqref{m_E_5}.
The proof is completed.
\end{proof}

With the help of Lemmas~\ref{m_MacDis}--\ref{m_MicDis}  we have the following global existence result.
\begin{prop}\label{m_energy1}
Let $N\ge 4$. Then, there are equivalent energy functionals
$E^f_0(\cdot)\sim E(\cdot)$, $H^f_0(\cdot)\sim H(\cdot)$ and $H^f_1(\cdot)\sim H_1(\cdot)$ so that
if the initial energy $E(f_0)$ is sufficiently small, then the Cauchy problem
\eqref{m_VPB4}--\eqref{m_VPB6} of the mVPB system admits a unique global
solution $f(x, v, t)$ satisfying
\bma
  &\Dt E^f_0(f(t)) + \mu D(f(t)) \le 0, \label{m_G_1}\\
%
 & \Dt H^f_0(f(t))
 +\mu D(f(t)) \le C\|\Tdx \P_0f(t)\|^2_{L^2_{x,v}},\label{m_G_4}\\
&  \Dt H^f_1(f(t))
 +\mu D_1(f(t)) \le C\|\Tdx \P_0f(t)\|^2_{L^2_{x,v}}.\label{m_G_4a}
  \ema
\end{prop}

\subsection{Convergence rates}
\label{NLmvpb}

With the help of the energy estimates established in Sect.~\ref{Lmvpb}, we are able to show Theorems~\ref{m_rate3}--\ref{m_rate4} for  the Cauchy problem of the nonlinear mVPB system~\eqref{m_VPB4}-\eqref{m_VPB6} in this subsection. 

\begin{proof}[\underline{Proof of Theorem \ref{m_rate3}}]
Let $f$ be the global solution to the IVP problem \eqref{m_VPB4}--\eqref{m_VPB6} for $t>0$. We can represent it in terms of the semigroup $e^{tB_m}$  as
 \bq
 f(t)=e^{tB_m}f_0+\intt e^{(t-s)B_m}G(s)ds+\intt e^{(t-s)B_m}\divx(I-\Delta_x)^{-1}V(s)ds,     \label{m_Duh}
 \eq
where the nonlinear terms $G$ and $V$  are defined by
\bmas
 G=\frac12 (v\cdot\Tdx\Phi)f-\Tdx\Phi\cdot\Tdv f+\Gamma(f,f),
\quad
V=v\sqrt M (e^{\Phi}+\Phi-1).
\emas

Define a functional $Q_1(t)$ for the global solution $f$ to the IVP problem \eqref{m_VPB4}--\eqref{m_VPB6} for any $t>0$ as
 \bmas
  Q_1(t)=\sup_{0\le s\le t}\sum_{k=0}^1&\Big\{(\sum_{j=0}^4\|\dx^k(f(t),\chi_j)\|_{L^2_{x}}+\|\dx^k\Phi(s)\|_{H^1_x})(1+s)^{\frac34+\frac k2}+\|\dx^k \P_1f(s)\|(1+s)^{\frac54+\frac k2}\\
&+(\|\P_1f(s)\|_{H^N_w} +\|\Tdx \P_0f(s)\|_{L^2_v(H^{N-1}_x)})(1+s)^{\frac54}\Big\}.
\emas
We shall show below  that it holds under the assumptions of Theorem~\ref{m_rate3} that
  \bq
 Q_1(t)\le C\delta_0.  \label{m_assume}
 \eq
To this end, we first deal with the time-decay rates of the macroscopic density, momentum and energy, which in terms of \eqref{m_Duh} satisfy the following equations
$$(f(t),\chi_j)=(e^{tB_m}f_0,\chi_j)+\intt (e^{(t-s)B_m}G(s),\chi_j)ds+\intt (e^{(t-s)B_m}\divx(I-\Delta_x)^{-1}V(s),\chi_j)ds.$$
By \eqref{m_S_4a}--\eqref{m_S_4c}, we have for any $\alpha,\alpha'\in \N^3$ with $\alpha'\le \alpha$ that
\bma
&\|\dxa(e^{tB_m}\dx(I-\Delta_x)^{-1} f_0,\chi_j)\|_{L^2_{x}}\le C(1+t)^{-\frac54-\frac{k}2}(\|\dxa f_0\|_{L^2_{x,v}}+\|\dx^{\alpha'}f_0\|_{L^{2,1}}),\quad j=0,1,2,3,4,\label{m_phi_3}\\
&\|\dxa(I-\Delta_x)^{-1}(e^{tB_m}\dx(I-\Delta_x)^{-1} f_0,\sqrt M)\|_{H^1_{x}}\le C(1+t)^{-\frac54-\frac{k}2}(\|\dxa f_0\|_{L^2_{x,v}}+\|\dx^{\alpha'}f_0\|_{L^{2,1}}),\label{m_phi_3a}\\
&\|\dxa\P_1(e^{tB_m}\dx(I-\Delta_x)^{-1} f_0)\|_{L^2_{x,v}}\le C(1+t)^{-\frac74-\frac{k}2}(\|\dxa f_0\|_{L^2_{x,v}}+\|\dx^{\alpha'}f_0\|_{L^{2,1}}),\label{m_phi_4}
\ema with $k=|\alpha-\alpha'|$.
Then, we can estimate the nonlinear terms $G(s)$ and $V(s)$ for $0\le s\le t$ as below
\bma
\| G(s)\|_{L^2_{x,v}}&\le C\|w
f\|_{L^{2,3}}\|f\|_{L^{2,6}}+\|\Tdx\Phi\|_{L^3_x}(\|w
f\|_{L^{2,6}}+\|\Tdv f\|_{L^{2,6}}) 
\le C(1+s)^{-2}Q_1(t)^2,
\label{m_G1}\\
 \| G(s)\|_{L^{2,1}} &
 \le C\|f\|_{L^2_{x,v}}\|w f\|_{L^2_{x,v}}+\|\Tdx\Phi\|_{L^2_x}(\|w f\|_{L^2_{x,v}}
     +\|\Tdv  f\|_{L^2_{x,v}}) 
 \le C(1+s)^{-\frac32}Q_1(t)^2,\label{m_G2}
\\
 \| V(s)\|_{L^2_{x,v}}&=\sqrt3\|e^{-\Phi}+\Phi-1\|_{L^2_x} \le C e^{\|\Phi\|_{L^\infty_x}}\|\Phi\|_{L^3_x}\|\Phi\|_{L^6_x}
 \le C(1+s)^{-2}e^{CQ_1(t)}Q_1(t)^2,\label{m_G3}
 \\
 \| V(s)\|_{L^{2,1}}&=\sqrt3\|e^{-\Phi}+\Phi-1\|_{L^1_x} \le C e^{\|\Phi\|_{L^\infty_x}}\|\Phi\|_{L^2_x}\|\Phi\|_{L^2_x}
 \le C(1+s)^{-\frac32}e^{CQ_1(t)}Q_1(t)^2.\label{m_G4}
 \ema
It follows from
\eqref{m_V_4}, \eqref{m_phi_3} and \eqref{m_G1}--\eqref{m_G4} that \bma
\|(f(t),\chi_j)\|_{L^2_{x}}&\le
C(1+t)^{-\frac34}(\|f_0\|_{L^2_{x,v}}+\|f_0\|_{L^{2,1}})\nnm\\
&\quad+C\intt (1+t-s)^{-\frac34}(\| G(s)\|_{L^2_{x,v}}+\|G(s)\|_{L^{2,1}})ds\nnm\\
&\quad+C\intt (1+t-s)^{-\frac54}(\| V(s)\|_{L^2_{x,v}}+\|V(s)\|_{L^{2,1}})ds\nnm\\
&\le
C\delta_0(1+t)^{-\frac34}+C(1+t)^{-\frac34}e^{CQ_1(t)}Q_1(t)^2.\label{m_macro_1}
 \ema
Similarly, we have
 \bma
  \|\Tdx (f(t),\chi_j)\|_{L^2_{x}}&\le
C(1+t)^{-\frac54}(\|\Tdx f_0\|_{L^2_{x,v}}+\|f_0\|_{L^{2,1}})\nnm\\
&\quad+C\intt (1+t-s)^{-\frac54}(\|\Tdx G(s)\|_{L^2_{x,v}}+\| G(s)\|_{L^{2,1}})ds\nnm\\
&\quad+C\intt (1+t-s)^{-\frac74}(\|\Tdx V(s)\|_{L^2_{x}}+\|V(s)\|_{L^{2,1}})ds\nnm\\
&\le C\delta_0(1+t)^{-\frac54}+C (1+t)^{-\frac54}e^{CQ_1(t)}Q_1(t)^2,\label{m_macro_2}\ema where
we have used
\bmas \|\Tdx G(s)\|_{L^2_{x,v}}+\| G(s)\|_{L^{2,1}}
&\le C(1+s)^{-\frac32}Q_1(t)^2,\\
\|\Tdx V(s)\|_{L^2_{x,v}}+\| V(s)\|_{L^{2,1}}
&\le C(1+s)^{-\frac32}e^{CQ_1(t)}Q_1(t)^2.
\emas

By $\eqref{m_VPB5}$, we obtain
$$
 \Phi=-(I-\Delta_x)^{-1} (f,\sqrt M)+(I-\Delta_x)^{-1} (e^{-\Phi}+\Phi-1),
 $$
which implies that
\bma
\Phi(t)&=-(I-\Delta_x)^{-1}(e^{tB_m}f_0,\sqrt M)-\intt(I-\Delta_x)^{-1}(e^{(t-s)B_m}G(s),\sqrt M)ds\nnm\\
&\quad+\intt(I-\Delta_x)^{-1}(e^{(t-s)B_m}\divx(I-\Delta_x)^{-1}V(s),\sqrt M)ds+(I-\Delta_x)^{-1} (e^{-\Phi}+\Phi-1). \label{m_phi_5}
\ema
Due to the fact
\bma
\|\Tdx^k(I-\Delta_x)^{-1} (e^{-\Phi}+\Phi-1)\|_{H^1_x}\le \|e^{-\Phi}+\Phi-1\|_{L^2_{x}}\le C(1+t)^{-2}e^{CQ_1(t)}Q^2(t),\quad k=0,1, \label{m_phi_0}
 \ema we obtain by \eqref{m_V_4a} and \eqref{m_phi_3a} that
\bma
\|\Phi(t)\|_{H^1_x}
&\le C\delta_0(1+t)^{-\frac34}+C(1+t)^{-\frac34}e^{CQ_1(t)}Q^2(t)+C(1+t)^{-2}e^{CQ_1(t)}Q^2(t),\label{m_Phi_2}\\
\|\Tdx\Phi(t)\|_{H^1_x}
&\le C\delta_0(1+t)^{-\frac54}+C(1+t)^{-\frac54}e^{CQ_1(t)}Q^2(t)+C(1+t)^{-2}e^{CQ_1(t)}Q^2(t).\label{m_Phi_2a}
\ema

Next, we estimate the microscopic part $\P_1f(t)$. Since  $\P_1f(t)$  satisfies
$$\P_1f(t)=\P_1(e^{tB_m}f_0)+\intt \P_1(e^{(t-s)B_m}G(s))ds+\intt \P_1(e^{(t-s)B_m}\divx(I-\Delta_x)^{-1}V(s))ds,$$
it follows from 
\eqref{m_V_5} and \eqref{m_phi_4} that
 \bma
\|\P_1f(t)\|_{L^2_{x,v}}&\le
C(1+t)^{-\frac54}(\|f_0\|_{L^2_{x,v}}+\|f_0\|_{L^{2,1}})\nnm\\
&\quad+C\intt (1+t-s)^{-\frac54}(\| G(s)\|_{L^2_{x,v}}+\|G(s)\|_{L^{2,1}})ds\nnm\\
&\quad+C\intt (1+t-s)^{-\frac74}(\| V(s)\|_{L^2_{x,v}}+\|V(s)\|_{L^{2,1}})ds\nnm\\
&\le
C\delta_0(1+t)^{-\frac54}+C(1+t)^{-\frac54}e^{CQ_1(t)}Q_1(t)^2,\label{m_micro_1}
 \ema
and
 \bma
\|\Tdx \P_1f(t)\|_{L^2_{x,v}}&\le
C(1+t)^{-\frac74}(\|\Tdx f_0\|_{L^2_{x,v}}+\|f_0\|_{L^{2,1}})\nnm\\
&\quad+C\int^{t/2}_0 (1+t-s)^{-\frac74}(\|\Tdx G(s)\|_{L^2_{x,v}}+\|G(s)\|_{L^{2,1}})ds\nnm\\
&\quad+C\int^t_{t/2} (1+t-s)^{-\frac54}(\|\Tdx G(s)\|_{L^2_{x,v}}+\|\Tdx G(s)\|_{L^{2,1}})ds\nnm\\
&\quad+C\intt (1+t-s)^{-\frac74}(\|\Tdx V(s)\|_{L^2_{x,v}}+\|\Tdx V(s)\|_{L^{2,1}})ds\nnm\\
&\le
C\delta_0(1+t)^{-\frac74}+C(1+t)^{-\frac74}e^{CQ_1(t)}Q_1(t)^2,\label{m_micro_2}
 \ema
due to the facts
\bmas
\|\Tdx G(s)\|_{L^2_{x,v}}+\|\Tdx G(s)\|_{L^{2,1}}&\le
C(1+s)^{-2}Q_1(t)^2,\\
\|\Tdx V(s)\|_{L^2_{x,v}}+\|\Tdx V(s)\|_{L^{2,1}}&\le
C(1+s)^{-2}e^{CQ_1(t)}Q_1(t)^2.
\emas

Finally,  the higher order estimates can be established in terms of  \eqref{m_G_4a} and $d_1H^f_1(f)\le D_1(f)$ for some constant  $d_1>0$  as
\bma
H^f_1(f(t))&\le e^{-d_1\mu t}H^f_1(f_0)+\intt e^{-d_1\mu(t-s)}\|\Tdx \P_0f(s)\|^2_{L^2_{x,v}}ds\nnm\\
&\le C\delta_0^2e^{-d_1\mu t} +\intt e^{-d_1\mu(t-s)}(1+s)^{-\frac52}(\delta_0+Q_1(t)^2)^2ds\nnm\\
&\le C(1+t)^{-\frac52}(\delta_0+e^{CQ_1(t)}Q_1(t)^2)^2.\label{m_other}
\ema
By summing \eqref{m_macro_1}, \eqref{m_macro_2}, \eqref{m_Phi_2}, \eqref{m_Phi_2a}, \eqref{m_micro_1}, \eqref{m_micro_2} and \eqref{m_other} together, we have
$$Q_1(t)\le C\delta_0+Ce^{CQ_1(t)}Q_1(t)^2,$$ which leads to  \eqref{m_assume} for $\delta_0>0$ small enough. This completes the proof of the proposition.
\end{proof}

\begin{proof}[\underline{Proof of Theorem \ref{m_rate4}}]
By Eq.~\eqref{m_Duh} and Theorem~\ref{m_rate3}, we can establish the lower bounds of the time decay rates of macroscopic density, momentum and energy of the global solution $f$ and its microscopic part. Indeed, it holds for $t>0$ large enough and $k=0,1$ that
\bmas
\|\Tdx^k(f(t),\chi_j)\|_{L^2_{x}}&\ge\|\Tdx^k(e^{tB_m}f_0,\chi_j)\|_{L^2_{x}}-\intt\|\Tdx^k(
e^{(t-s)B_m}G(s),\chi_j)\|_{L^2_{x}}ds\\
&\quad-\intt\|\Tdx^k(
e^{(t-s)B_m}\divx(I-\Delta_x)^{-1}V(s),\chi_j)\|_{L^2_{x}}ds\\
&\ge C_1\delta_0(1+t)^{-\frac34-\frac k2}-C_2\delta_0^2(1+t)^{-\frac34-\frac k2},\\
\|\Tdx^k\P_1f(t)\|_{L^2_{x,v}}&\ge\|\Tdx^k\P_1(e^{tB_m}f_0)\|_{L^2_{x,v}}-\intt\|\Tdx^k\P_1(
e^{(t-s)B_m}G(s))\|_{L^2_{x,v}}ds\\
&\quad-\intt\|\Tdx^k\P_1(
e^{(t-s)B_m}\divx(I-\Delta_x)^{-1}V(s))\|_{L^2_{x,v}}ds\\
&\ge C_1\delta_0(1+t)^{-\frac54-\frac k2}-C_2\delta_0^2(1+t)^{-\frac54-\frac k2},
\emas
and by \eqref{m_phi_5} and \eqref{m_phi_0} that
\bma
\|\Tdx^k\Phi(t)\|_{H^1_{x}}&\ge \|\Tdx^k(I-\Delta_x)^{-1}(e^{tB_m}f_0,\sqrt M)\|_{H^1_{x}}-\intt\|\Tdx^k(I-\Delta_x)^{-1}(
e^{(t-s)B_m}G(s),\sqrt M)\|_{H^1_{x}}ds\nnm\\
&\quad-\intt\|\Tdx^k(I-\Delta_x)^{-1}(e^{(t-s)B_m}\divx(I-\Delta_x)^{-1}V(s),\sqrt M)\|_{H^1_{x}}ds\nnm\\
&\quad-\|\Tdx^k(I-\Delta_x)^{-1} (e^{-\Phi}+\Phi-1)\|_{H^1_{x}}\nnm\\
&\ge C_1\delta_0(1+t)^{-\frac34-\frac k2}-C_2\delta_0^2(1+t)^{-\frac34-\frac k2}-C_2\delta_0^2(1+t)^{-2}.
\ema
These and Theorem~\ref{m_rate3} give rise to
\bmas
\|f(t)\|_{H^N_w}&\ge\|\P_0f(t)\|_{L^2_{x,v}}-\|w\P_1f(t)\|_{L^2_{x,v}}-\sum_{1\le |\alpha|\le N}\|w\dxa f(t)\|_{L^2_{x,v}}\\
&\ge C_1\delta_0(1+t)^{-3/4}-C_2\delta_0^2(1+t)^{-3/4}-C_3\delta_0(1+t)^{-5/4}.
 \emas
Therefore, we obtain \eqref{m_B_1}--\eqref{m_B_4} for $\delta_0>0$ sufficiently and $t>0$ large enough.
\end{proof}

\medskip
\noindent {\bf Acknowledgements:}
The research of the first author
was partially supported by the National Natural Science Foundation of China  grants No. 11171228, 11231006
and 11225102, and by the Key Project of Beijing Municipal Education
Commission. The research of the second author was supported by the
General Research Fund of Hong Kong, CityU No.103412. And research of
the third author was supported by the National Natural Science Foundation of China  grants No. 11301094 and Project supported by Beijing Postdoctoral Research Foundation No. 2014ZZ-96.



\begin{thebibliography}{99}
\setlength{\itemsep}{-4pt}
\renewcommand{\baselinestretch}{1}
\small
\bibitem{Cercignani} C. Cercignani, R. Illner and M. Pulvirenti, The Mathematical Theory of Dilute Gases, Applied
Mathematical Sciences, 106. Springer-Verlag, New York, 1994.


\bibitem{Codier-Grenier} S. Codier and E. Grenier, Quasineutral limit of an Euler-Poisson system arising from plasma physics.
Commun. Part. Diff. Eq., 25 (2000), 1099-1113.



\bibitem{Duan1} R.J. Duan and R. M. Strain, Optimal time decay of the Vlasov-Poisson-Boltzmann system in $\R^3$,
Arch. Ration. Mech. Anal., 199 (2011), no. 1, 291-328

\bibitem{Duan2} R.J. Duan and T. Yang, Stability of the one-species
Vlasov-Poisson-Boltzmann system, SIAM J. Math. Anal., 41 (2010), 2353-2387.


\bibitem{Duan4}  R.J. Duan, T. Yang and C.J. Zhu, Boltzmann equation with external force and Vlasov-Poisson-Boltzmann system in infinite vacuum, Discrete Contin. Dyn. Syst.,
16 (2006), 253-277.

\bibitem{Ellis} R.S. Ellis and M.A. Pinsky, The first and second fluid approximations to the linearized Boltzmann equation. {\it J. Math. pure et appl.}, 54 (1975), 125-156.


\bibitem{Guo2} Y. Guo, The Vlasov-Poisson-Boltzmann system near Maxwellians, Comm. Pure Appl. Math.,
55 (9) (2002), 1104-1135.

\bibitem{Guo3} Y. Guo, The Vlasov-Poisson-Boltzmann system near vacuum, Comm. Math. Phys., 218 (2)
(2001), 293-313.




\bibitem{Li2} H.L. Li, T. Yang and M.Y. Zhong, Spectral analysis for the  Vlasov-Poisson-Boltzmann system.
Preprint.

\bibitem{Liu1} T.-P. Liu and S.-H. Yu, The Green’s function and large-time behavior of solutions for the
one-dimensional Boltzmann equation, Comm. Pure Appl. Math., 57 (2004), 1543-1608.

\bibitem{Liu2} T.-P. Liu, T. Yang and S.-H. Yu, Energy method for the Boltzmann equation, Physica D, 188 (3-4) (2004), 178-192.


\bibitem{Markowich} P.A. Markowich, C.A. Ringhofer and C. Schmeiser, Semiconductor Equations, Springer-Verlag, Vienna, 1990. x+248 pp.

\bibitem{Mischler} S. Mischler, On the initial boundary value problem for the Vlasov-Poisson-Boltzmann system,
Comm. Math. Phys., 210 (2000), 447-466.



\bibitem{Pazy} A. Pazy, Semigroups of Linear Operators and
Applications to Partial Differential Equations. Applied
Mathematical Sciences, 44. Springer-Verlag, New York, 1983.



\bibitem{Ukai1} S. Ukai, On the existence of global solutions of mixed problem for non-linear Boltzmann
equation, Proceedings of the Japan Academy, 50 (1974), 179-184.

\bibitem{Ukai2} S. Ukai and T. Yang, The Boltzmann equation in the space $L^2\cap L^\infty_\beta$: Global and time-periodic solutions, Analysis and Applications, 4 (2006), 263-310.

\bibitem{Ukai3} S. Ukai, T. Yang, Mathematical Theory of Boltzmann
Equation. Lecture Notes Series-No. 8,
Hong Kong: Liu Bie Ju Center for Mathematical Sciences, City University of Hong Kong, March 2006.


\bibitem{Yang1} T. Yang, H.J. Yu and H.J. Zhao, Cauchy problem for the Vlasov-Poisson-Boltzmann system,
Arch. Rational Mech. Anal., 182 (2006), 415-470.


\bibitem{Yang3} T. Yang and H.J. Zhao, Global existence of classical solutions to the Vlasov-Poisson-
Boltzmann system, Comm. Math. Phys., 268 (2006), 569-605.

\bibitem{Yang4} T. Yang, H.J. Yu, Optimal convergence rates of classical solutions
for Vlasov-Poisson-Boltzmann system. Commun. Math. Phys. 301 (2011), 319-355.

\bibitem{Yu}Alexander Sotirov and Shih-Hsien Yu, On the Solution of a Boltzmann System for Gas Mixtures. Arch. Rational Mech. Anal., 195 (2010) 675-700.


\bibitem{Zhong2012Sci} M. Zhong, Optimal time-decay rate of the Boltzmann equation. Sci China Math, 2014, 57: 807-822, doi: 10.1007/s11425-013-4621-1.
\end{thebibliography}
\end{document}